\documentclass[12pt,reqno,a4paper]{amsart}
\usepackage[utf8]{inputenc}
\usepackage[headings]{fullpage} 
 
\usepackage[english]{babel}
\usepackage[T1]{fontenc}         
\usepackage{lmodern}             
\usepackage{amsmath,amsthm,amsfonts,amssymb}           
\usepackage{mathrsfs}

\usepackage{graphicx}
\usepackage{tikz-cd}
\usepackage{rotating}

\usepackage{array}

\usepackage{url}
\usepackage[colorlinks,backref=page,allcolors=blue]{hyperref}
\usepackage[alphabetic,backrefs,lite]{amsrefs}
\usepackage[capitalize,nameinlink,noabbrev]{cleveref} 
\crefname{page}{page}{pages}

\newtheorem{mythm}{Theorem}[section]
\newtheorem{mycor}[mythm]{Corollary}
\newtheorem{mylemma}[mythm]{Lemma}

\newtheorem{mythmintro}{Theorem}
\newtheorem{myquest}{Question}

\newtheorem{myptn}[mythm]{Proposition}

\theoremstyle{definition}
\newtheorem{mydef}[mythm]{Definition}
\newtheorem{mydefintro}{Definition}
\newtheorem{mynotation}[mythm]{Notation}
\newtheorem{myidentification}[mythm]{Identification}
\newtheorem{myexample}[mythm]{Example}
\newtheorem{myexampleintro}{Example}
\newtheorem{myremark}[mythm]{Remark}
\newtheorem{mysetting}[mythm]{Setting}

\DeclareMathOperator{\Sym}{\mathrm{Sym}}

\DeclareMathOperator{\im}{\mathrm{Im}} 
\DeclareMathOperator{\CEff}{\mathrm{Eff}} 
\DeclareMathOperator{\Hom}{\mathrm{Hom}} 
\DeclareMathOperator{\HHom}{\mathbf{Hom}} 
\DeclareMathOperator{\Spec}{\mathrm{Spec}} 

\DeclareMathOperator{\Pic}{\mathrm{Pic}} 
\DeclareMathOperator{\NS}{\mathrm{NS}} 
\DeclareMathOperator{\vol}{\mu} 
\DeclareMathOperator{\pr}{\mathrm{pr}} 
\DeclareMathOperator{\Supp}{\mathrm{Supp}} 
\DeclareMathOperator{\res}{\mathrm{res}} 
\DeclareMathOperator{\rk}{\mathrm{rk}}  
\DeclareMathOperator{\rg}{\mathrm{rk}}  
\DeclareMathOperator{\ord}{\mathrm{ord}} 

\DeclareMathOperator{\Gr}{\mathrm{Gr}} 
\DeclareMathOperator{\HOM}{\mathcal{H}\mathit{om}} 
\DeclareMathOperator{\ordjac}{\mathrm{ordjac}} 
\DeclareMathOperator{\length}{\mathrm{length}} 
\DeclareMathOperator{\coker}{\mathrm{coker}} 

\newcommand{\KVar}[1]{K_0  \mathbf{Var}_{#1} } 
\newcommand{\KVaruh}[1]{K_0^{\mathrm{uh}}  \mathbf{Var}_{#1} } 

\newcommand{\Perv}{\mathbf{Perv}}
\DeclareMathOperator{\MHM}{\mathbf{MHM}}

\DeclareMathOperator{\rat}{\mathrm{rat}} 
\DeclareMathOperator{\Sy}{\mathrm{S}} 

\DeclareMathOperator{\degg}{\mathbf{deg}}  

\newcommand{\dd}{\mathbf{d}}  
\newcommand{\ee}{\mathbf{e}}  
\newcommand{\ii}{\mathbf{i}}  
  
\newcommand{\mm}{\mathbf{m}}  
\newcommand{\nn}{\mathbf{n}}  
\newcommand{\pp}{\mathbf{p}} 
\newcommand{\qq}{\mathbf{q}}   

\newcommand{\CC}{\mathbf{C}}
\newcommand{\DD}{\mathbf{D}}
\newcommand{\PP}{\mathbf{P}}

\newcommand{\ZZ}{\mathbf{Z}}
\newcommand{\NN}{\mathbf{N}}
\newcommand{\TT}{\mathbf{t}}
\newcommand{\LL}{\mathbf{L}}
\newcommand{\GG}{\mathbf{G}}

\newcommand{\QQ}{\mathbf{Q}}

\newcommand{\AAA}{\mathscr{A}}
\newcommand{\BBB}{\mathscr{B}}
\newcommand{\CCC}{\mathscr{C}}
\newcommand{\DDD}{\mathscr{D}}

\newcommand{\FFF}{\mathscr{F}}
\newcommand{\GGG}{\mathscr{G}}

\newcommand{\LLL}{\mathscr{L}}
\newcommand{\MMM}{\mathscr{M}}
\newcommand{\OOO}{\mathscr{O}}

\newcommand{\RRR}{\mathscr{R}}
\newcommand{\SSS}{\mathscr{S}}
\newcommand{\TTT}{\mathscr{T}}
\newcommand{\UUU}{\mathscr{U}}
\newcommand{\VVV}{\mathscr{V}}
\newcommand{\WWW}{\mathscr{W}}
\newcommand{\XXX}{\mathscr{X}}
\newcommand{\YYY}{\mathscr{Y}} 
\newcommand{\ZZZ}{\mathscr{Z}} 

\newcommand{\Hdg}{\mathrm{Hdg}}
\newcommand{\Zar}{\mathrm{Zar}}

\newcommand{\mdeg}{\delta}
\newcommand{\scdot}{\,\cdot \,}
\newcommand{\kb}{k} 
\newcommand{\kp}{{\kappa_p}}

\newcommand{\Rp}{{R_p}} 
\newcommand{\Fp}{{F_p}} 

\begin{document}

\title[Motivic distribution of rational curves]{Motivic distribution of rational curves \\
and twisted products of toric varieties}


\begin{abstract}
	This work concerns asymptotical stabilisation phenomena occurring in the moduli space of sections of certain algebraic families over a smooth projective curve, whenever the generic fibre of the family is a smooth projective Fano variety, or not far from being Fano. 
	
	We describe the expected behaviour of the class, in a ring of motivic integration, 
	of the moduli space of sections of given numerical class. 
	Up to an adequate normalisation, it should converge, 
	when the class of the sections goes arbitrarily far from the boundary of the dual of the effective cone,
	to an effective element given by a motivic Euler product. 
	Such a principle can be seen as an analogue for rational curves of the Batyrev-Manin-Peyre principle for rational points. 
		
	The central tool of this article is the property of equidistribution of curves. 
	We show that this notion does not depend on the choice of a model of the generic fibre,
	and that equidistribution of curves holds for smooth projective split toric varieties. 
	As an application, we study the Batyrev-Manin-Peyre principle for curves on a certain kind of twisted products.
	\end{abstract}


\author[L. Faisant]{Loïs \textsc{Faisant}}
\address{IST Austria, Am Campus 1, 3400 Klosterneuburg, Austria}
\email{lois.faisant@ista.ac.at or @m4x.org}
\date{\today}

\makeatletter
\@namedef{subjclassname@2020}{%
  \textup{2020} Mathematics Subject Classification}
\makeatother



\maketitle


\setcounter{tocdepth}{2}
\tableofcontents


\section*{Introduction}

The aim of this work is to describe the expected behavior of moduli spaces of morphisms from a smooth projective curve to a smooth Fano variety 
when the numerical class of the curves goes to infinity.
More generally, we study sections of a faithfully flat morphism to such a curve.
What we consider are the classes of these moduli spaces in a variant of the Grothendieck ring of varieties. 

\smallskip 

This is related to the classical subject of homological stability.
Recent developments concern
homological stability for moduli spaces,
as described for example by Ellenberg, Venkatesh and Westerland in \cite{ellenberg2016homological}.
The underlying philosophy is that given
a sequence $(X_n)$ 
of algebraic varieties defined over the finite field $\mathbf F_q$, of growing dimension with respect to $n$
-- for instance, a sequence of moduli spaces -- 
then 
the quantity $| X_n ( \mathbf F_q ) | q^{-\dim ( X_n )}$ 
should admit a limit as $n\to \infty$
precisely when the cohomology groups of $X_n$ stabilise \cite[\S 1.8]{ellenberg2016homological}. 
For example, 
when the $X_n$ are smooth and geometrically irreducible varieties of dimension $n$,
Deligne's proof of the Weil conjecture
\cite{deligne1974conjecture}
provides an upper bound on the eigenvalues of the Frobenius morphism, 
acting on
the étale cohomology groups of $X_n$,
and the Grothendieck-Lefschetz formula 
allows one to pass from
homological stabilisation to point-counting stabilisation.

The study of the asymptotic behaviour of the number of $\mathbf F_q ( t ) $-rational points of bounded height on Fano varieties over $\mathbf F_q$,
in the framework suggested by Manin and his collaborators in the late eighties,
is expected to be an illustrative instance of this phenomenon: in that case, $\mathbf F_q ( t ) $-points can be interpreted as $\mathbf F_q$-points of the moduli space of 
morphisms from $\PP^1_{\mathbf F_q}$
to the variety. 
One of the goals of this article is to formulate a motivic analogue of this framework which allows ``point-counting'' in any characteristic (in particular, of $\CC (t )$-points). Actually, the motivic point of view tends to incorporate both the homological approach and the point-counting approach. Then, techniques coming from arithmetic 
help one get answers to questions concerning geometry and topology: dimension, number of components, homological stabilisation, or even stabilisation of Hodge structures. 
 
\medskip 
Let us now recall that the modern formulation of a Batyrev--Manin--Peyre principle 
\textit{over number fields}
 is the sum of works carried out by
Franke--Manin--Tschinkel \cite{franke1989rational}, 
Batyrev--Manin \cite{batyrev1990nombre},
Batyrev--Tschinkel \cite{batyrev1998tamagawa},
Peyre \cite{peyre1995hauteurs},
Salberger \cite{salberger1998tamagawa},
Lehmann--Tanimoto \cite{lehmann2017geometry}
and Lehmann--Sengupta--Tanimoto \cite{lehmann2022geometric}, among many others. 
This principle admits many variants, for example when one works 
\textit{over a function field of positive characteristic}
\cite{peyre2012points}, 
when one considers 
\textit{integral points} \cite{chambert2012integral}
or 
\textit{Campana points}
\cite{pieropan2021campana},
or counts rational points on stacks
\cite{ellenberg2021heights,darda2022yasuda}.

This literature is a rich source of inspiration for anyone wishing to consider a geometric analogue of this principle, namely the asymptotic behaviour of \textit{rational curves} of arbitrarily large degree  -- which
corresponds to rational points of varieties defined \textit{over a function field of arbitrary characteristic}. This is no longer strictly speaking a counting problem if the characteristic is zero, 
but a quite natural framework is provided by the theory of motivic integration,
introduced by Kontsevich (during his talk in Orsay, 1995) and then developed and expanded by Denef, Loeser \cite{denef1999germs}, Looijenga \cite{looijenga2002motivic} and their collaborators \cite{sebag2004integration,sebag2007motivic,chambert2018motivic}. 
Then \textit{counting curves} means taking the classes, in the ring of motivic integration, of the components of the moduli space of rational curves. 

\medskip

Following 
remarks by Batyrev about the dimension of the moduli space
and
a question raised by Peyre in the early 2000s, one may ask whether the class of the moduli space of curves of a given degree stabilises when the degree, which is an element of the dual of the Picard group of the variety, goes arbitrarily far away from the boundary of the dual of the effective cone.

\medskip

Any positive answer to this question
for a given class of varieties
would be 
in the spirit of
a larger family of recent stabilisation results in motivic statistics, 
such as those appearing in several works by Vakil and Wood concerning moduli of hypersurfaces \cite{vakil2015discriminants},
by Bilu and Howe about
moduli of sections of vector bundles \cite{bilu2021motivic},
by Bilu, Das and Howe about
configuration spaces and Hadamard convergence \cite{bilu2020zeta},
or by Landesman, Vakil and Wood regarding
low degree Hurwitz spaces \cite{landesman2022low},
to cite a few of them. 

\medskip

Until recently, an obstruction to a precise formulation of a geometric or motivic Batyrev-Manin-Peyre principle 
was the absence of a sufficiently robust \textit{notion of motivic Euler product} which would play the role of Peyre's adelic constant. 
A path was
opened up by Bourqui \cite{bourqui2009produit,bourqui2010fonctions,bourqui2011asymptotic} in the late 2000's
and the notion reached a certain degree of maturity and robustness quite recently with Bilu's thesis \cite{bilu2018motivic}. In this article we make extensive use of this notion in order to formulate a motivic Batyrev-Manin-Peyre principle for curves, as well as a stronger notion of equidistribution of curves.
We test the validity of these concepts on smooth split projective toric varieties and twisted products of toric varieties, in a continuation of \cite{bourqui2003fonctions,bourqui2009produit,chambert2001fonctions,chambert2016motivic,bilu2020zeta}. 

\subsection*{Framework}
We fix once and for all a smooth projective 
and geometrically integral
curve $\CCC$ over a base field $\kb$, with function field $F = \kb ( \CCC ) $.
\begin{mydefintro}\label{def:Fano-like}
If $V$ is a smooth (irreducible) projective variety over $F $, we will say that $V$ is a \textit{Fano-like variety} if 
\begin{itemize}
\item $ V ( F ) $ is Zariski-dense in $V$; 
\item both cohomology groups $H^1 ( V , \OOO_V ) $ and $ H^2 ( V , \OOO_V ) $ are trivial;
\item the Picard group of $V$ coincides with its geometric Picard group;
\item the geometric Picard group of $V$ has no torsion, and its geometric Brauer group is trivial;
\item the class of the anticanonical line bundle of $V$ lies in the interior of the effective cone $\CEff ( V ) $, which itself is rational polyhedral: there exists a finite set of effective line bundles spanning it. 
\end{itemize}
\end{mydefintro}

In the case of a projective variety $V$ defined over the base field $\kb$, $F$-points of $V$ correspond to
morphisms $f:\CCC \to V$. Such a morphism will be somewhat abusively called
a \textit{rational curve} if $\CCC$ is the projective line $ \PP^1_\kb $. 
In general, we will consider sections of models $\pi : \VVV \to \CCC $ of $V$. 
By \textit{model}, we mean a separated and faithfully flat morphism of finite type whose generic fibre is isomorphic to $V$. 
If furthermore $\pi$ is proper, $F $-points of $V$ will correspond to sections of $\pi$. 
For the sake of conciseness, in this introduction we start by assuming that $V$ is actually defined over the base field $\kb$. 

\subsubsection*{Multidegrees of curves}
If $L$ is a invertible sheaf on $V$ and $f:\CCC \to V$ a morphism, 
then a relevant invariant of $f$ is its degree with respect to $L$, that is, the degree
\[
\deg ( f^* L ) 
\]
of the pull-back of $L$ to $\CCC$ through $f$. 
However, except for the canonical sheaf $\omega_V$ of $V$, 
there is no canonical choice for $L$, 
and the canonical degree has no particular reason to be an invariant sharp enough, 
as the analogy with the arithmetic setting may suggest (see for example \cite[\S 4]{peyre2021beyond}). 
A mean of addressing this issue would be to introduce more invariants.  

Our approach consists in taking \textit{all the degrees} 
in the following manner: 
actually, a morphism $f:\CCC \to V$ 
defines an element of the dual of the Picard group of $V$ 
\[
\degg f : [L] \in \Pic ( V ) \mapsto  \deg ( f^* L ) ,
\]
called the \textit{multidegree}. 
Given a class $\mdeg \in \Pic ( V ) ^\vee $, 
one can show that 
morphisms $f:\CCC \to V$ 
of multidegree $\delta $ are parametrized 
by a quasiprojective scheme $\Hom^\delta_\kb ( \CCC , V ) $ \cite[Chap. 2]{debarre2001higher}.
The dimension of $\Hom^\delta_\kb ( \CCC , V ) $ admits the lower bound 
\[
\mdeg
\scdot 
\omega_V^{-1} 
+ 
\dim ( V ) ( 1  - g ( \CCC ))
\]
called the \textit{expected dimension}.

In this paper, we are interested in the behaviour of this sequence of moduli spaces
when $\delta $ 
belongs to the dual of the effective cone $\CEff ( V )$ 
and 
goes arbitrarily far from its boundaries. 
More precisely, we study the sequence of the corresponding \textit{classes} in a ring of varieties. 

\subsubsection*{The ring of motivic integration} 

Let $S$ be a scheme. 
The Grothendieck group of $S$-varieties 
\[
\KVar S
\] 
is defined as the abelian group generated by the isomorphism classes of \textit{$S$-varieties} (by this, we mean $S$-schemes of finite presentation), together with \textit{scissors relations}
\[
X - Y - U
\] 
whenever $X$ is an $S$-variety, $Y $ is a closed subscheme of $X$ and $U$ is its open complement in $X$. 
The class of an $S$-variety $X$ is denoted by $[X]$.
The class of the affine line $\mathbf A^1 _S $  is denoted by $\LL_S$
and when the base scheme is clear from the context we may drop the index.
Any constructible subset $X$ of an $S$-variety admits a class $[X]$ in such a group \cite[p. 59]{chambert2018motivic}. 
In our case, a constructible subset is a finite union of locally closed subset of an $S$-variety. 

The product $[X][Y]=[X\times_S Y]$ defines a ring structure on $\KVar S $ with unit element the class of $S$ over itself with natural structural map. 
The localised Grothendieck ring of varieties $ \mathscr M_S  $ is by definition the ring $\KVar S $ localised at the class $\LL_S$ of the affine line.

The ring $\mathscr M_S$ admits a decreasing filtration by the virtual dimension: 
for $m\in \ZZ$, let $\mathcal F^m \MMM_S $ be the subgroup of $\MMM_S$ generated by elements of the form
\[
[ X ]\LL_S^{-i} 
\]
where $X$ is an $S$-variety and $i$ an integer such that $\dim_S ( X ) - i \leqslant - m$.
The completion of $\MMM_S$ with respect to this decreasing dimensional filtration is the projective limit
\[
\widehat{\MMM_S}^{\dim } = 
\underset{\longleftarrow}{\lim } \, \MMM_S / \mathcal F^m \MMM_S 
\] 
which comes with a morphism $\MMM_S \to \widehat{\MMM_S}^{\dim }$. 
The dimensional filtration is one of the filtrations we are going to use, another one being the filtration by the weight of the Hodge realisation, see \cref{subsection:weight-topology}.

In positive characteristic, we will work with modified versions of $\KVar S$ and $\MMM_S$ (see \S 4.4 of \cite[Chap. 2]{chambert2018motivic}).
An $S$-morphism $f:X\to Y$ between $S$-varieties is called a \textit{universal homeomorphism} 
if for any $S$-morphism $Y' \to Y$ the induced morphism $X\times_Y Y ' \to Y' $ is a homeomorphism.
Then the modified ring of varieties $\KVaruh S$ is the quotient of $\KVar S$ by the ideal
given by differences $[X] - [Y]$ of classes of $S$-varieties such that there exists a $S$-morphism $X\to  Y$ which is a universal homeomorphism. 
If $S$ is a $\QQ$-scheme, this ideal is trivial so that $\KVaruh S \simeq \KVar S$ \cite[Chap. 2, Cor. 4.4.7]{chambert2018motivic}.
An equivalent description is given by the quotient of $\KVar S$ by \textit{radicial surjective morphisms}, see \cite[Remark 2.1.4]{bilu2021motivic}.
Note that we will systematically drop the "uh" exponent in this article. 

\subsubsection*{Expected asymptotical behaviour}

Now we are able to formulate a first version of what could be a \textit{motivic Batyrev-Manin-Peyre principle for curves}. 
Again by analogy with the arithmetical side, 
one has to take into account the hypothetical existence of \textit{accumulating subsets}. 
A description of these subsets in the geometric context is given 
in the works of Lehmann, Tanimoto and Tschinkel \cite{lehmann2018balanced,lehmann2019geometric}.
In particular, we will have to focus on curves intersecting a well-chosen non-empty open subset $U$ of $V$, denoting by $\Hom^\mdeg _\kb ( \CCC , V ) _U$ the corresponding moduli space. 
 
Let $\CEff ( V )^\vee_\ZZ $ be the subset of $\CEff ( V )^\vee$ consisting of points in $\Pic ( V )^\vee$. 
In characteristic zero,
curves whose classe belongs to this subset are \textit{moveable} by \cite{boucksom2012pseudo}.  
With Question 5.4 in \cite{peyre2021beyond}, Peyre raised the following question. 

\begin{myquest}\label{BMP-motivic-weak}
Assume that $V$ is a Fano-like variety, defined over the base field $\kb$. 
Does the symbol 
\[
\left [ 
\Hom^\mdeg _\kb ( \CCC , V ) _U 
\right ]
\LL_\kb^{- \, \mdeg \scdot   \omega_V^{-1} } 
\]
converge in $\widehat{\mathscr M_\kb}^{\dim} $ as the class $\mdeg \in \CEff (V) ^\vee_\ZZ $ goes arbitrarily far from the boundaries of $ \CEff (V) ^\vee $?  
\end{myquest}
Note that in some cases this formulation is too optimistic, in particular the convergence may only hold for a filtration by the weight, which is finer than the one by the dimension, as it is the case for example in \cite{bilu2018motivic,faisant2022geometric,bilu2023circle},
rather than for the coarser dimensional filtration.
We refer to \cref{conj:BMP-motivic-strong-weight-top} \cpageref{conj:BMP-motivic-strong-weight-top} for a more accurate version.

In what follows, we will say that $V$ verifies the motivic Batyrev-Manin-Peyre principle for curves
if the answer to the previous question is positive for $V$. 
We now give a few examples for which it is the case. 

\begin{myexampleintro}
The first and simplest example is given by the projective space of dimension $n\geqslant 1$. Since its Picard group is generated by the class of a hyperplane $H$, the class of a moveable curve is given by a degree $d\in \NN$ while an anti-canonical divisor is $(n+1)H$, so that the normalisation factor is $\LL^{(n+1)d}_\kb $.  
One can easily compute the class of the moduli space $ \Hom^d _\kb( \PP^1_\kb , \PP^n_\kb ) $ in $\MMM_\kb$ and show that for $d\geqslant 1$
\[
\left [ \Hom^d _\kb ( \PP^1_\kb , \PP^n_\kb ) \right ] \LL_\kb^{-(n+1)d} = 
\left [ 
\PP^n_\kb 
\right ]
\left ( 1 - \LL_\kb^{-n} \right )  ,
\]
see \cite[Proposition 5.5]{peyre2021beyond}. The answer to \cref{BMP-motivic-weak} is positive in this case.
\end{myexampleintro} 

\begin{myexampleintro}
More generaly, \cref{BMP-motivic-weak} has a positive answer for smooth projective split toric varieties over any field $\kb$.
The work of Bourqui 
\cite{bourqui2009produit},
together with the unpublished notes \cite{bourqui2011asymptotic} only provided the result in the ring of Chow motives, when $\kb$ has characteristic zero. 
Bilu, Das and Howe \cite[\S 5]{bilu2020zeta}  checked a residue-type result 
at the level of the ring of varieties 
and we finally show in \cref{section-BMP-toric} that Bourqui's study actually provides convergence of the normalized class $\left [ \Hom^\mdeg _\kb ( \PP^1_\kb , V ) \right ]\LL_\kb^{-\, \mdeg \scdot \omega_V^{-1}}$ in $\widehat{\mathscr M_\kb}^{\dim}$
(see \cref{thm-intro:BMP-toric-varieties} below and \cref{thm:motivic-BMP-toric} \cpageref{thm:motivic-BMP-toric}). 
\end{myexampleintro}

\begin{myexampleintro}
	Building on previous works by Chambert-Loir-Loeser \cite{chambert2016motivic} and Bilu \cite{bilu2018motivic}, we show in \cite{faisant2022geometric} that the answer to \cref{BMP-motivic-weak} is positive for equivariant compactifications of vector spaces, 
	when 
	$\kb$ is algebraically closed of characteristic zero 
	and
	one considers the finer filtration on the ring of motivic integration given by the weight of the mixed Hodge realisation. 	
	See \cref{example:compact-ev}. 
\end{myexampleintro}

\begin{myexampleintro}
	Bilu and Browning develop in \cite{bilu2023circle} 
	a motivic circle method and apply it to show that the answer to \cref{BMP-motivic-weak}
	is positive when $V \subset \PP^{n-1}_\CC$ is a hypersurface of degree $d\geqslant 3$ with $n > 2^d ( d - 1 )$, if one considers the filtration on $\MMM_\CC$ coming from the weight of the mixed Hodge realisation. 
\end{myexampleintro}

\subsection*{Main results and structure of the paper} 

In this article we show that our motivic Batyrev-Manin-Peyre principle for curves
is compatible with a certain kind of twisted products \cite{chambert2001torseurs}. 
A little bit more precisely, we show:
\begin{mythmintro}[{\cref{thm:BMP-twisted-products}}]
\label{thm-intro:BMP-twisted-products-toric}
	If $\BBB$ is a Fano-like variety over $\kb$ verifying the motivic Batyrev-Manin-Peyre principle for curves, 
$X$ is a smooth projective  split toric variety over $\kb$, with torus $T$, 
and $\TTT$ is a Zariski-locally trivial $T$-torsor over $\BBB$,
then the twisted product  
\[ \XXX = X \times^\TTT \BBB \]
verifies the motivic Batyrev-Manin-Peyre principle as well: 
the answer to \cref{BMP-motivic-weak} is positive 
for rational curves on $\XXX$. 
\end{mythmintro}
The proof of this result is to be found in the very last section. 
It requires two ingredients.

\subsubsection*{Formulating a Batyrev-Manin-Peyre principle for curves}

The first ingredient is presented in \cref{section:BMP-general} and consists in a precise formulation of a motivic Batyrev-Manin-Peyre principle in a non-constant setting, namely when one considers a proper model $\VVV \to \CCC$ of a $F $-variety $V$
and studies the moduli space of its sections $\sigma : \CCC \to \VVV$
generically intersecting a convenient open subset $U$ of $V$, 
as it is done for example in \cite{chambert2016motivic,bilu2018motivic,faisant2022geometric} for equivariant compactifications of vector spaces. 
Fixing line bundles forming a basis of $\Pic ( V ) $ 
and choosing models of them over $\VVV$,
one is able to define the multidegree of a section, generalising the previous notion (see \cref{def:multidegree-relative} \cpageref{def:multidegree-relative}). 
If $\VVV$ is projective over the base field $\kb$,
we show in the first section of this article that the corresponding moduli space of sections 
$\Hom_\CCC^{\mdeg} ( \CCC , \VVV )_U $ \textit{of multidegree} $\mdeg$
exists as a quasi-projective $\kb$-scheme, 
and formulate a relative geometric Batyrev-Manin-Peyre principle concerning the behaviour of the class 
\begin{equation}\tag{1}\label{intro:normalised-class}
	\left [ \Hom_\CCC^{\mdeg} ( \CCC , \VVV )_U \right ] \LL_\kb^{-\, \mdeg \scdot \omega_V^{-1} } 
\end{equation}
when the multidegree $\mdeg$ \textit{becomes arbitrarily large} (see \cref{conj:BMP-motivic-strong-weight-top} \cpageref{conj:BMP-motivic-strong-weight-top}). 

\smallskip 

Let $\tau ( \VVV )$ be the expected limit of \eqref{intro:normalised-class}.
We give an explicit description of $\tau ( \VVV ) $ as a motivic Euler product, using Bilu's construction \cite{bilu2018motivic} of such an object. 
This motivic Tamagawa number can be interpreted as a motivic analogue of Peyre's constant \cite{peyre1995hauteurs}. For a constant family ${\VVV = V \times_\kb \CCC } $, where $V$ is actually defined over $\kb$, it takes the form
\[
\tau ( \VVV ) =  \frac{\LL_\kb^{( 1 - g ( \CCC ) )\dim V } \left [ \Pic^0 ( \CCC ) \right ]^{\rg ( \Pic ( V  ))} }{\left ( 1 - \LL_\kb^{-1} \right )^{\rg ( \Pic ( V ) )}} \prod_{p\in \CCC } \left ( 1 - \LL_p^{-1} \right )^{\rg ( \Pic ( V  ))} \frac{\left [ V _ p \right ]}{\LL_{p}^{\dim ( V )}} 
\]
where $g ( \CCC ) $ is the genus of the curve $\CCC$ 
and $\Pic^0 ( \CCC )$ the connected component of the Picard group of  $ \CCC $. 
Note that the class $\left ( 1 - \LL_p^{-1} \right )^{\rg ( \Pic ( V  ))} \frac{\left [ V _ p \right ]}{\LL_{p}^{\dim ( V )}} $ 
can be rewritten 
\[
[ \mathcal T_V ] \LL^{-(\dim ( V) + \rg ( \Pic ( V ))} 
\]
where $\mathcal T_V$ is the\footnote{Since $\Pic ( V ) = \Pic ( \overline{ V } )$, the universal torsor of $V$ is unique up to isomorphism. For more about this alternative definition of the local factor in a more general framework, see Peyre's talk \cite[around 23']{peyre2021banff}.} 
universal torsor of $V$.

However, the definition of $\tau ( \VVV )$ may require the use of a finer filtration on $\mathscr M_\kb$, namely the filtration by the weight of the Hodge realisation of $\VVV / \CCC$, when $\kb$ is a subfield of $\CC$. 
We refer to \cref{def:motivic-Tamagawa-number-of-a-model} \cpageref{def:motivic-Tamagawa-number-of-a-model}.

\subsubsection*{Equidistribution of curves and models}

The second ingredient is a change-of-model theorem, given by \cref{thm-intro:change-of-model} below. 
Such a result is grounded on the
central idea of this article, which is the concept of equidistribution of curves.
This notion, which is presented in \cref{section:equidistribution}, describes the behaviour of constrained curves of large multidegree: if $\SSS$ is a zero-dimensional subscheme of the curve $\CCC$, and $\varphi : \SSS \to \VVV$ a $\CCC$-morphism, broadly,
the equidistribution principle says that the class of the moduli space of curves 
of multidegree $\mdeg$
whose restriction to $\SSS$ is $\varphi$, normalised $\LL_\kb^{\mdeg \scdot \omega_V^{-1}}$, converges to the restriction of the product $\tau ( \VVV  ) $ to the complement of the closed points of $\SSS$. 
More generally, one considers curves whose restriction to $\SSS$ belongs to any constructible subset of $\Hom_{\CCC} ( \SSS , \VVV)$. 
In particular, the notion of equidistribution is much stronger than the Batyrev-Manin-Peyre principle, see \cref{def:equidistribution-products-constructible-subsets} \cpageref{def:equidistribution-products-constructible-subsets}. 
Smooth projective split toric varieties provide a first example of varieties for which this principle holds (see \cref{section-BMP-toric}).  

\begin{mythmintro}[{\cref{thm-equidistribution-toric}}]
\label{thm-intro:BMP-toric-varieties}
The principle of equidistribution of curves holds for smooth and projective split toric varieties defined over any base field $\kb$, with respect to the dimensional filtration on $\MMM_\kb$. 

In particular, for such varieties, the answer to \cref{BMP-motivic-weak} is positive.
\end{mythmintro}

Our main fundamental result allows us to switch between models, whenever one has equidistribution of curves on one of them.

\begin{mythmintro}[{\cref{thm:equidistribution-and-models}}]
\label{thm-intro:change-of-model}
	Assume that 
	two proper models $\VVV / \CCC $ and $\VVV ' / \CCC$ 
	of the same Fano-like $F$-variety $V$ are given,
	together with models of a $\ZZ$-basis of $\Pic ( V ) $ on each of them, defining two multidegrees $\mdeg$ and $\mdeg '$, for sections of $\VVV $ and $\VVV ' $, respectively.  
	
	Then there is equidistribution of curves on $\VVV $ with respect to $\mdeg$ if and only if there is equidistribution of curves on $\VVV ' $ with respect to $\mdeg '$. 
\end{mythmintro}

\subsubsection*{Convention} In this article, without explicit mention, if $R$ is a discrete valuation ring, we will always assume that it has equal characteristic. 


\subsection*{Acknowledgements}
I am very grateful to my PhD advisor Emmanuel Peyre for all the remarks and suggestions he made during the writing process of this article. 
I warmly thank Margaret Bilu and Tim Browning for some valuable comments they made on a preliminary version of this work. 
I would like to thank David Bourqui as well for 
several helpful conversations. 
Finally, I thank an anonymous referee 
for their very careful reading, their numerous comments and suggestions 
which helped me a lot in improving the exposition,
besides fixing several typos.

During the revision process of this work, 
the author received 
funding from the European Union's Horizon 2020 
research and innovation programme under the Marie Sk\l odowska-Curie Grant Agreement No.101034413.

\section{Models and moduli spaces of sections}

\subsection{Global models and degrees}
Let $\CCC$ be a smooth projective 
and geometrically integral
curve over a field $\kb$, with function field $F$, 
and let $V$ be a smooth $F$-variety. 
As in the introduction, a \textit{model of $V$ over $\CCC$} is a separated, faithfully flat and finite type $\CCC$-scheme $\VVV$ whose generic fibre is isomorphic to $V$. 

\begin{myexample}
Let $V\hookrightarrow \PP^N_F$ be an embedding of $V$ in some projective space. Take $\VVV$ to be the Zariski closure of $V$ in $\PP^N_\CCC$. Then the composition $\VVV \to \PP^N_\CCC \to \CCC$ is a projective model of $V$. 
\end{myexample}

\begin{myremark}
	If $\pi : \VVV \to \CCC$ is a proper model,
	the functors $\pi_!$ and $\pi_*$ from the category of sheaves on $\VVV$ to the ones on $\CCC$ coincide \cite[Chap. 6, \S 3]{milne1980etale}.
	 Since $\CCC$ is integral, by \cite[Theorem 4.18.2]{kleiman2005picard} there exists a nonempty Zariski open subset $\CCC  ' \subset \CCC$ such that the Picard scheme $\Pic_{\VVV _{\CCC '} / \CCC ' }$
	representing the Picard functor $ {\mathbf{Pic}_{(\VVV_{\CCC '} / \CCC ' ) } }_{\mathrm{(fppf)} } $ exists and is a disjoint union of open quasi-projective subschemes. 
	Here we recall that $ {\mathbf{Pic}_{ ( X / S ) }}_{\mathrm{(fppf)} } $ is the sheaf associated to the functor
	\[
	( T / S ) \mapsto \Pic ( X \times_S T ) / \Pic ( T ) 
	\] in the fppf (faithfully flat of finite type) topology,
	given a separated map of finite type $X\to S$ between locally Noetherian schemes
	 \cite[Definition 2.2]{kleiman2005picard}.
	
	Moreover, we assume that $\pi$ has (local) sections, so that the Picard functor ${\mathbf{Pic}_{(\VVV_{\CCC } / \CCC ) } }_{\mathrm{(fppf)} }$ is actually 
\[
\mathbf{Pic}_{\VVV / \CCC } ( T ) = H^0 ( T , \mathcal R^1 \pi _* ( \GG_m )) 
\]
for the Zariski topology \cite[p.204]{bosch1990neron}. 
\end{myremark}

\begin{myremark}\label{remark:local-triviality-of-derived-fuctor}
Assume there exists a closed point $p_0 \in \CCC$ such that $H^1 ( \VVV_{p_0}, \OOO_{\VVV_{p_0}} ) = H^2 ( \VVV_{p_0}, \OOO_{\VVV_{p_0}} ) = 0$.
By the Semicontinuity Theorem \cite[(7.7.5-I)]{EGA-III},
there exists an open neighborhood $\CCC ''  $ of $p_0$ such that  $H^1 ( \VVV_p, \OOO_{\VVV_p} ) = H^2 ( \VVV_p, \OOO_{\VVV_p} ) = 0$ for all $p\in \CCC '' $. 
This shows in particular that $( R^1 \pi_! \OOO_{\VVV} ) _{| \CCC '' } = ( R^2 \pi_! \OOO_{\VVV} )_{ | \CCC '' } = 0$.
By flat base-change
 \cite[\href{https://stacks.math.columbia.edu/tag/02KH}{Lemma 02KH}]{stacks-project} 
 applied to the generic fibre,
 \[
 \begin{tikzcd}
 V \arrow[r] \arrow[d] & \VVV \arrow[d,"\pi"] \\
 \Spec ( F  ) \arrow[r] & 	\CCC 
 \end{tikzcd}
 \]
 we have $( R^i \pi_! \OOO_{\VVV} )_\eta = H^i ( V , \OOO_V ) $ for all $i$ and in particular $H^1 ( V , \OOO_V ) = H^2 ( V , \OOO_V ) = 0$. 
 Conversely, if $H^1 ( V , \OOO_V )$ and $ H^2 ( V , \OOO_V ) $ are both trivial, then there exists a non-empty open subset of $\CCC$ above which $R^1 \pi_! \OOO_{\VVV}$ and $R^2 \pi_! \OOO_{\VVV}$ are both trivial. 
 This argument actually shows that the assumptions on the first and second cohomology groups of the structure sheaf of a Fano-like variety,
 in \cref{def:Fano-like}, 
 can be done indifferently with respect to $V$ or to a special fibre of $\VVV$.  
 
Moreover, by \cite[Proposition 5.19]{kleiman2005picard}, $\Pic_{\VVV _{\CCC '} / \CCC ' } $ is smooth over $\CCC ' \cap \CCC ''$ and by \cite[Corollary 5.13]{kleiman2005picard} each fibre $\Pic_{\VVV _{p } / \kp }$ above $p\in \CCC ' \cap \CCC '' $ is discrete, given by $H^2 ( \VVV_p , \ZZ ) $. 
From this point of view, $\mathbf{Pic}_{\VVV_{\CCC '} / \CCC ' }$ is a constructible sheaf on $\CCC$ and can be seen as a variation of mixed Hodge structure, see the proof of \cref{proposition:hodge-realisation-almost-fano-varieties} \cpageref{proposition:hodge-realisation-almost-fano-varieties}.  
\end{myremark}

\begin{mysetting}\label{setting:main-setting}
Let $\VVV \to \CCC$ be a proper model of a Fano-like $F$-variety $V$. 
We fix a finite set $L_1 , ... , L_r $ of invertible sheaves on $V$ whose linear classes form a basis of the torsion-free $\ZZ$-module $\Pic ( V)$, as well as invertible sheaves $\LLL_1 , ... , \LLL_r$  on $\VVV$ extending respectively $L_1 , ... , L_r$.
\end{mysetting}

\begin{mydef}[Multidegree]
\label{def:multidegree-relative}  Let $\VVV \to \CCC$ 
and $\underline{\LLL} = (\LLL_ 1 , ... , \LLL_r )$
be as in \cref{setting:main-setting}.

A section $\sigma : \CCC \to \VVV$ 
defines an element $\degg_{\underline{\LLL}} ( \sigma )
$ of the dual $\Pic ( V )^\vee$
by 
sending an effective invertible sheaf of the form
\[
\otimes_{i=1}^r L_i ^{\otimes \lambda_i } 
\]
to the linear combination of degree
\[ 
\sum_{i=1}^r \lambda_i \deg ( \sigma^* \LLL_i ). 
\]
This element $\degg_{\underline{\LLL}} ( \sigma )$ is 
called the \textit{multidegree of $\sigma$ with respect to the model $ \underline{\LLL} $}.
\end{mydef}

\subsection{Moduli spaces of curves} 
Again, let $\VVV \to \CCC$ and $\underline{\LLL} = (\LLL_ 1 , ... , \LLL_r )$ be as in \cref{setting:main-setting}. 
For every class $\mdeg \in \Pic ( V ) ^\vee $, 
we consider the functor
\[
\HHom_\CCC^\mdeg ( \CCC , \VVV ) 
\]
sending a $\kb$-scheme $T$ to the set of maps  $\sigma \in \Hom_{\mathrm{Sch}/T} ( \CCC \times_\kb T , \VVV \times_\kb T ) $ such that
\begin{align*}
& \pi_T \circ  \sigma = \mathrm{id}_{\CCC \times_\kb T} \\
& \text{for all $t\in T$, }  \degg_{\underline{\LLL}} ( \sigma_t ) = \mdeg . 
\end{align*}
If $U$ is a dense open subset of the generic fibre $V$, 
we define 
\[
\HHom_\CCC^\mdeg ( \CCC , \VVV )_U 
\]
to be the subfunctor of $\HHom_\CCC^\mdeg ( \CCC , \VVV ) $ sending a $\kb$-scheme $T$ to the $T$-families of maps sending the generic point of $\CCC$ into $U ( F )$.  

The moduli space of sections of a proper model $\VVV\to \CCC$ is well-defined:
\index{moduli space of sections}  
adapting the ideas of \cite[Lemme 4.1]{bourqui2009produit} and \cite[Proposition 2.2.2]{chambert2016motivic} we get the following general representability lemma. Here we assume that $\VVV$ is projective over the base field $k$. 
\begin{mylemma}\label{lemma:representability-Hom-functor}
For any non-empty open subset $U \subset V$ and any class $\mdeg \in \Pic ( V )^\vee$, 
the functor $\mathbf{Hom}_\CCC^\mdeg ( \CCC , \VVV )_{U}$ 
is representable by a quasi-projective scheme. 
\end{mylemma}
\begin{proof}
Let $\LLL$ be an ample invertible sheaf on $\VVV$. 
For every $d\geqslant 1$,
let 
\[
\HHom^d_ \kb ( \CCC , \VVV )
\]
be the functor of morphisms $ \varsigma : \CCC \to \VVV$ such that $\deg ( \varsigma ^* \LLL ) = d $
and
\[
\HHom^d_\CCC ( \CCC , \VVV ) 
\]
be
the functor of sections $ \sigma : \CCC \to \VVV$ such that $\deg ( \sigma ^* \LLL ) = d $. 

By the  existence theorems of Hilbert schemes \cite[4.c]{grothendieck1960techniques},
there exists a quasi-projective $\kb$-scheme $\Hom^d_\kb ( \CCC , \VVV ) $
representing $\HHom^d_ \kb ( \CCC , \VVV )$. 
The condition $\pi \circ \sigma = \mathrm{id}_\CCC$ is closed 
and thus 
$\HHom^d_\CCC ( \CCC , \VVV ) $ 
is a closed subfunctor of $\HHom^d_ \kb ( \CCC , \VVV ) $. 
Therefore it is represented by a quasi-projective scheme $\Hom^d_\CCC ( \CCC , \VVV )$.

Let $U$ be a non-empty open subset of $V$ and let 
$\Hom_\CCC^d ( \CCC , \VVV ) _U$ be the complement of the closed subscheme of $\Hom_\CCC^d ( \CCC , \VVV ) $ defined by the condition $\sigma ( \CCC ) \subset | \VVV \setminus U |$. This open subscheme $\Hom_\CCC^d ( \CCC , \VVV ) _U$ parametrises sections  
$\sigma :\CCC \to \VVV$ such that $f(\eta_\CCC ) \in U ( F )  $ and 
$\deg ( \sigma ^* ( \LLL ) )= d$. It is again a quasi-projective scheme, since $\Hom_\CCC^d ( \CCC , \VVV )$ is.

The restriction $L$ of $\LLL$ to $V$ is isomorphic to a linear combination 
\[
\otimes_{i=1}^r L_i^ {\otimes \lambda_i} 
\]
of the $L_i$'s. Let $\LLL '$ be the invertible sheaf
\[
\otimes_{i=1}^r \LLL_i^ {\otimes \lambda_i} 
\]
on $\VVV$.

Let $\mdeg \in \Pic ( V ) ^\vee $ be a multidegree
and $\sigma : \CCC \to \VVV $ 
a section 
such that $\degg_{\underline{\LLL}} ( \sigma ) = \mdeg$ and $\sigma ( \eta_\CCC) \in U ( F )$. 
Then 
\begin{align*}
\deg ( \sigma ^* \LLL  ) & = \deg ( \sigma^* \LLL ) - \deg ( \sigma ^* \LLL '  ) + \deg ( \sigma ^* \LLL '  ) \\
	 & = \deg ( \sigma^* \LLL ) - \deg ( \sigma ^* \LLL '  ) + \mdeg \scdot L  . \\
\end{align*}
Since the restriction of $\LLL \otimes  ( \LLL ' )^{-1} $ to the generic fibre is trivial,
there exist vertical divisors $E$ and $E' $ such that 
\[
\LLL \otimes  ( \LLL ' )^{-1}  \simeq \OOO_{\VVV} ( E ' - E ) .
\]
In particular, the difference $\deg ( \sigma^* \LLL ) - \deg ( \sigma ^* \LLL '  )$ only takes a finite number of values,
since $\sigma$ has intersection degree one with any fibre of $\pi$. Let $a$ be its maximal value.  
Moreover, by flatness, the $r$ conditions given by $\degg_{\underline{\LLL}} = \mdeg$ are open and closed in the Hilbert scheme of $\VVV$. 
Therefore $\HHom_\CCC^\mdeg ( \CCC , \VVV )_U $ can be identified with an open subfunctor of 
\[
\coprod_{0 \leqslant d \leqslant a + \mdeg \scdot L }  \HHom_\CCC^d ( \CCC , \VVV )_U 
\]
hence the lemma. 
\end{proof}

\subsection{Greenberg schemes and motivic integrals} 

A reference for this subsection is given by chapters 4 to 6 of \cite{chambert2018motivic}. 

If $R$ is a complete discrete valuation ring, with field of fractions $F$ and residue field $\kappa $, 
such that $F$ and $\kappa$ have equal characteristic, 
then the choice of an uniformizer $\pi$ of $R$ 
together with a section of $R\to \kappa$
provides a morphism of $\kappa$-algebras 
\[
\begin{array}{rll}
	\kappa [[ t ]] & \to & R \\
	t & \mapsto & \pi 
\end{array}
\]
which is an isomorphism by Theorem 2 of \cite[Chap. IX, \S 3]{bourbaki1983algebre}.

\begin{myexample}
If $p$ is a closed point of the smooth projective $\kb$-curve $\CCC$,
the previous result applies to the completed local ring $R_p = \widehat{\OOO_{\CCC , p}}$. 	
\end{myexample}

\begin{mydef}[Greenberg schemes] Let $R$ be a complete discrete valuation ring, with maximal ideal $\mathfrak m$ and residue field $\kappa = R / \mathfrak m$. 
Assume that $R$ has equal characteristic and fix a section of $R\to \kappa$. 

Let $\XXX$ be an $R$-variety. 
For any non-negative integer $m$, the Greenberg scheme of order $m$ of $\XXX$ is the $\kappa$-scheme $\Gr_m ( \XXX ) $ representing the functor 
\[
A \mapsto \Hom_R ( \Spec ( R_m \otimes_{ \kappa  } A  ) , \XXX ) 
\]
on the category of $\kappa$-algebras \cite[Chap. 4, \S 2.1]{chambert2018motivic},
where $R_m = R / \mathfrak m^{m+1}$ for all $m\geqslant 0$.
There are canonical affine projection morphisms 
\[ 
\theta^{m+1}_m :   \Gr_{m+1} ( \XXX ) 
 \rightarrow   \Gr_m ( \XXX ) 
,
\] 
given by truncation, and the Greenberg scheme is the $\kappa$-pro-scheme 
\[
\Gr_\infty ( \XXX ) = \underset{\longleftarrow}{\lim} \,  \Gr_m ( \XXX ) 
\]
(or more concisely $\Gr ( \XXX )$) which represents the functor 
\[
A \mapsto \Hom_\kappa  ( \Spec ( A \otimes_\kappa R ) , X ) 
\]
on the category of $\kappa $-algebras. 
This scheme carries a canonical projection
\[
\theta_m^\infty  : \Gr_\infty ( \XXX ) \to \mathscr\Gr_m ( \XXX )
\]
for every non-negative integer $m$, called the truncation of level $m$. 
\end{mydef}

\begin{myexample}
If $\VVV \to \CCC$ is a model of a projective variety $V$ over $F = \kb ( \CCC ) $, 
let $\VVV_{\Rp }$  be
the schematic fibre 
over the completed local ring $R_p = \widehat{\OOO_{\CCC,p}}$.
If $\mathfrak m_p$ is the maximal ideal of $R_p$, then the $\kp$-points respectively 
of 
$\Gr_m ( \VVV_{\Rp } ) $ 
and 
$\Gr (  \VVV_{\Rp } ) $
are in 
bijection 
respectively
with the sets of points 
$ \VVV_{\Rp } ( R_p / \mathfrak m_p ^{m+1}) $
and 
$\VVV_{\Rp }  ( R_p ) $.
Moreover, if $\VVV \to \CCC$ is proper, 
by the valuative criterion of properness 
the set $\VVV_{R_p}  ( R_p ) $ 
is in one-to-one correspondance 
with the set $V_{\Fp} ( \Fp )$, where $F_p$ is the completion of $F$ at $p$.  	
\end{myexample}

By \cite[Chap. 4, Lemma 4.2.2]{chambert2018motivic}, the constructible subsets of $\Gr (\XXX) $ are exactly the subsets $C$
of the form 
\[
C = \left ( \theta^\infty_m \right )^{-1} ( C_m ) 
\]
for a certain level $m$ and a constructible subset $C_m $ of $\Gr_m ( \XXX ) $. Moreover, if $C$ is Zariski-closed, respectively Zariski-open, then $C_m$ can be chosen to be closed, respectively open. 
Then, a map 
\[
f : C \to \ZZ \cup \{ \infty \} 
\]
on a constructible subset $C $ of $\Gr ( \XXX )$ is said to be constructible if $f^{-1} ( n ) $ is constructible for every $n \in \ZZ$. 

By \cite[Chap. 6, \S 2]{chambert2018motivic}, there is an additive motivic measure 
\[
\mu_\XXX : \mathrm{Cons}_{\Gr ( \XXX ) } \to \widehat{\MMM_{\XXX_0}} ^{\dim}
\]
(where $\XXX_0$ is the special fibre of $\XXX$) called the \textit{motivic volume} or \textit{motivic density}, which  extends to a countably additive motivic measure $\mu_\XXX^*$ on a class $\mathrm{Cons}_{\Gr ( \XXX ) }^*$ of measurable subsets of $\Gr ( \XXX )$, see \cite[Chap. 6, \S 3]{chambert2018motivic}. 
For example, if $\XXX$ is smooth of pure relative dimension $d$ over $R$, then 
\[
\mu_\XXX ( \Gr ( \XXX ) ) = [\XXX_0 ]\LL^{-d}. 
\]
If $ A $ is a measurable subset of $\Gr ( \XXX )$ and $f : A \to \ZZ \cup \{ \infty \} $
has measurable fibres (by this we mean that $f^{-1} ( n )$ is measurable for all $n\in \ZZ$) 
such that the series
\[
 \sum_{n\in \ZZ} \mu_\XXX^* ( f^{-1} (n)) \LL^{-n} 
\]
is convergent in $\widehat{\MMM_{\XXX_0}}^{\dim}$, then the motivic integral of $\LL^{-f}$ 
\[
\int_A \LL^{-f} \mathrm d \mu_\XXX^* = \sum_{n\in \ZZ} \mu_\XXX^* ( f^{-1} (n)) \LL^{-n} 
\]
is well-defined. 

\begin{myremark}
	By the quasi-compactness of the constructible topology, a constructible function $f$
	which does not reach infinity is bounded and thus 
	 only takes a finite number of values. 
In particular, if $f$ is bounded constructible, then $f$ is measurable and $\LL^{-f}$ is integrable, see Example 4.1.3 in \cite[Chap. 6]{chambert2018motivic}.
\end{myremark}

As explained in \cite[Chap.~4, (3.3.7)]{chambert2018motivic},
one can go from points on $\XXX$
with coordinates in extensions of $R$
of ramification index one
to points on the Greenberg schemes of $\XXX$. 
For this it is convenient to 
use the functors from $\kappa$-algebras to rings
given by 
\[
\mathscr R_m
:
A \mapsto  R_m \otimes_\kappa A 
\]
for any non-negative integer $m$, which form a system of functors whose limit is 
\[
\mathscr R_\infty 
:
A \mapsto  R \, \hat{\otimes}_\kappa A = \lim_{\leftarrow}  R_m \otimes_\kappa A .
\]
Such functors are 
represented respectively by $\mathbf A_\kb^m$, $m \in \NN$, and $\mathbf A_\kb^\NN$.
If $\kb '$ 
is a field extension of $\kb$,
then $R ' = \mathscr R_\infty ( k ' )$
is an extension of $R$ with ramification index one
(the only one up to unique isomorphism if $\kb ' / \kb $ is unramified),
by Proposition 2.3.2 of \cite[Chap.~4]{chambert2018motivic}.
Then by definition of $\Gr_\infty ( \XXX )$
there is a canonical bijection
\[
\Gr_\infty ( \XXX ) ( \kb ' )
\overset{\sim}{\longrightarrow} 
\XXX ( \mathscr R_\infty ( \kb ' )  ) 
\]
which is
compatible with the truncation morphisms.
Thus, if one wants to define functions at the level of Greenberg schemes, it is enough to define them on $R'$-points,
for any extension $R'$  of ramification index one over the complete valuation ring $R$.
This will allow one to use such functions on Greenberg schemes of weak Néron models later in \cref{subsection-weak-neron-models}, thanks to \cref{proposition:neron-smoothening-and-greenberg-schemes} stated a few pages below.

\begin{mydef}[{\S 3.1 of \cite[Chap. 5]{chambert2018motivic}}]
\label{def:jacobian-order-along-an-arc}
Let $f:\YYY \to \XXX$ be a morphism of flat 
$R$-schemes of finite type and pure relative dimension $d$. 
Let $R' $ be an extension of $R$ and $q \in \YYY ( R ' )$. 
Consider the morphism 
\[
\alpha ( q ) : \left ( ( f \circ q ) ^* \Omega^d_{\XXX / R '} \right ) / ( \mathrm{torsion} ) 
\to 
\left (  q  ^* \Omega^d_{\YYY / R '} \right ) / ( \mathrm{torsion} )
\]
of free $R'$-modules of finite rank
induced by $f$. 

The \textit{order of the Jacobian of $f$ along $q$}
is defined by 
\[
\ordjac_f ( q ) = \length_{R '} \coker ( \alpha ( q ) ) .
\]

This provides a function 
\[
\ordjac_f : \Gr ( \YYY ) \to \NN . 
\]
\end{mydef}

Note that by Proposition 3.1.4 of \cite[Chap. 5]{chambert2018motivic}, if $\YYY$ is smooth over $R$ and the restriction of $f$ to generic fibres is étale,
then $\ordjac_f$ is constructible and bounded. 

\begin{myptn}[Smooth change of variable]
\label{ptn:change-of-variable-smooth-case}
Let $\XXX$ and $\YYY$ be two $R$-schemes 
of finite type and pure relative dimension $d$,
with singular loci respectively $\XXX_{\text{sing}}$ and $\YYY_{\text{sing}}$.
Let ${f:\YYY \to \XXX}$ be a morphism of $R$-schemes.
Let $A$ and $B$ be constructible subsets respectively of $\Gr ( \XXX) - \Gr ( \XXX_{\text{sing}})  $ and $\Gr ( \YYY ) -  \Gr ( \YYY_{\text{sing}}) $. 
Assume that $f$ induces a bijection
\[
B ( \kappa ' ) \to A ( \kappa ' ) 
\]
for every extension $\kappa ' $ of $\kappa$,
and that $B \cap \ordjac^{-1}_f ( + \infty ) $ is empty. 

Let $\alpha : A \to \ZZ $ be a constructible function on $A$. 
Then the function
\[
\beta : y \in B \longmapsto ( \alpha \circ \Gr ( f ) ) ( y ) + \ordjac_f ( y) 
\]  
is constructible and 
\[
\int_A \LL^{-\alpha } \mathrm d \mu_\XXX = {f_0}_! \int_B \LL^{-\beta} \mathrm d \mu_\YYY
\]
in $\MMM_{\XXX_0}$. 
\end{myptn}

\begin{proof}
	See Theorem 1.2.5 in \cite[Chap. 6]{chambert2018motivic}.
\end{proof}

\begin{myremark}
	Here the symbol $\LL^{-\ordjac_f}$
	plays the role of the absolute value of the determinant 
	$\left | \det \left ( \left ( \frac{\partial x_i}{\partial y_j} \right )_{i,j} \right ) \right |$ in the usual change-of-variable formula. 
\end{myremark}

\subsection{Local intersection degrees}

Let $L$ be an invertible sheaf on $V$. 
A coherent sheaf on $\VVV$ whose restriction to $V$ is isomorphic to $L$ is called a \textit{model} of $L$. 
In this paragraph we define 
local intersection degrees on $L$ of a section $\CCC \to \VVV$ at a closed point $p \in \CCC$, with respect to a model of $L$,
and study the difference between the degrees given by two different models of $L$. 
This is a reformulation, in the framework of Greenberg schemes and motivic integration \cite{chambert2018motivic}, of the $p$-adic and adelic metrics of \cite[\S 1.2]{peyre2012points}. 
Let us first recall the local definition of such metrics. 

\begin{mydef}\label{def:local-order-with-respect-line-bundle}
\index{local intersection degree with respect to a model of a line bundle}
Let $R$ be a complete discrete valuation ring with fraction field $K$.
	Let $\XXX$ be an $R$-scheme of pure dimension, $X$ its generic fibre, $L$ an invertible sheaf on $X$ and $\LLL$ a model of $L$ on $\XXX$ (not necessarily invertible).
	
	For any extension $R'$ of $R$, of ramification index one over $R$, with uniformizer $\varpi$ and field of fractions $K'$, we define order functions on $\XXX ( R ' )$  as follows:
	let $\tilde{q} : \Spec ( R ' ) \to \XXX  $ be an $R'$-point of $\XXX$, 
	$q$ its restriction to the generic fibre and
	$y$ a point of the $K'$-vector space  $q^* L = (\tilde{q}^* \LLL ) \otimes_{R '} {K '}$. 
	Then we set, if $q^* y \neq 0 $,
	\[
	\ord_{\tilde{q}} ( y  ) = \max \{ n \in \ZZ \mid \varpi^{-n} y \in \tilde{q}^* \LLL / (\text{torsion}) \}
	\]
	(the unique integer $\ell$ such that $\varpi^{-\ell} y$ is a generator of the $R'$-lattice $ \tilde q^* \LLL / (\text{torsion})$),
	 and 
	 \[ 
	 \ord_{\tilde{q}} ( y ) = \infty 
	 \]
	 if $q^* y = 0 $. 
	 
	 In other words, $\ord_{\tilde{q}} ( y )$ is the valuation of $y$ with respect to the lattice $\tilde q^* \LLL / (\text{torsion})$: given a generator $y_0$ of $\tilde q^* \LLL / (\text{torsion})$, 
	 \[
	 \ord_{\tilde{q}} ( y ) = v_{R'} ( y / y_0 ) \in \ZZ \cup \{ \infty \}
	 \]
	 where $v_{R'}$ is the normalized discrete valuation extended to $K '$. 
	 In particular, this definition does not depend on the choice of the uniformizer $\varpi$ nor on the choice of the generator $y_0$. 
	 
	 If $s$ is a section of $L$ on an open set containing $q$, then 
	 \[
	 \ord \circ s  : \XXX ( R ' ) \to \ZZ \cup \{ \infty \}
	 \] 
	  is the function sending $\tilde q \in \XXX ( R ' ) $ to $\ord_{\tilde q }(s(q) )$. 
\end{mydef}


\begin{myremark}
We are going to apply the previous definition as follows. 
Let $\VVV$ be a proper model of a smooth $F$-variety $V$. We fix a closed point ${p\in \CCC}$ and work with the $R_p$-scheme $\VVV_\Rp$, identifying the sets $V(F_p)$ and $\VVV ( \Rp )$ by properness. 
We take $L$ to be an invertible sheaf on $V$ and $\LLL$ a model of $L$ on $\VVV$.

	If $q$ is a $F_p$-point of $V$, 
the function $\ord_{\tilde q } $ is well-defined on $L_q ( F_p )$.
Of course, any $F_p$-point of $L$ belongs to the fibre of a unique $F_p$-point of $V$, so that this function extends to a function
\[
\begin{array}{rll}
	\ord_p : L ( F_p ) &\to & \ZZ \cup \{  \infty \} \\
	y & \mapsto & \ord_{\tilde q} ( y ) \text{ whenever } y \in L_{q} .

\end{array}
\]

If $s \in \Gamma ( U , L ) $ is a section of $L$ above an open subset $U\subset V$, 
by composition one gets a map 
\[
\begin{array}{rll}
\ord_p \circ \, s  : U ( F_p ) & \to &  \ZZ \cup \{  \infty \} \\
x & \mapsto & \ord_{\tilde x}  ( s ( x ) )  . 
\end{array}
\]

If $\DDD$ is a Cartier divisor on $\VVV$ given by a rational section of $\LLL$ 
and $s$ is the restriction of this section to $V$, defined on an open subset $U\subset V$,
then $ \ord_p (s ( x )) $ coincides with the intersection number $( x , \DDD )_p = \deg ( \tilde x ^* \DDD )$ for all $x\in U ( F_p )$ not in $\DDD$.  
Besides, both take an infinite value whenever $x$ lies in $\DDD$. 
	
By the product formula, 
for all $x\in V(F)$
the (finite) sum on closed points 
	\[
	\sum_{p\in | \CCC | }\ord_p (s ( x ))\in \ZZ 
	\]
does not depend on the choice of a local section $s$ of $L$ such that $s(x) \neq 0$, 
and if 
$\sigma : \CCC \to \VVV $ is the section of the proper model $\VVV$ given by the point $x \in V ( F) $, 
then 
	\[
	\deg ( \sigma^* \LLL ) = \sum_{p\in | \CCC |} \ord_p (s ( x )) . 
	\]
\end{myremark}

We go back to the notations of \cref{def:local-order-with-respect-line-bundle}.

\begin{mylemma}
Let $\XXX$ be a model of $X$ and $\LLL $, $\LLL'$ be two models of $L$, 
and $\ord$, $\ord ' $ the corresponding order functions given by
\cref{def:local-order-with-respect-line-bundle}. 
For all $\tilde x \in \XXX ( R ' )$, 
the difference 
	\[
	y \mapsto \ord_{\tilde x }' ( y   )  - \ord_{ \tilde x } ( y   ) 
	\]
is constant on the stalks of $L$.
This constant value defines a function 
\[
\varepsilon_{ \LLL ' - \LLL } : \Gr_\infty ( \XXX ) \to \ZZ .
\] 	
\end{mylemma} 

\begin{proof}
The difference of the two corresponding valuations is 
\[
\ord_{ \tilde x } ' ( y )   - \ord_{ \tilde x  } ( y  ) 
 = v_{R'} ( y / y_0  ') - v_{R'} ( y / y_0  ) = v_{R'} ( y_0  / y_0 ' ) 
\]
for every $\tilde x  \in \XXX ( R' )$ and $y\in x^* L$,
where $y_0$ and $y_0'$ are generators respectively of
$\tilde x ^* \LLL$ and $\tilde{x} ^* \LLL '$ in $x^* L$.
Consequently, this difference induces a map 
\[
\tilde x \in \XXX ( R ' ) \mapsto v_{R'} ( y_0  / y_0 ' ) 
\]
which does not depend on the choices of the generators $y_0$ and $y_0 ' $ of $\tilde x ^* \LLL$ and $\tilde x ^* \LLL ' $, since the quotient of two generators has valuation zero.  
\end{proof}

\begin{myremark}
Note that if $\LLL $, $\LLL' $ and $\LLL ''$ are three models of $L$, we have the relation 
\[
\varepsilon_{ \LLL ''  - \LLL } = \varepsilon_{ \LLL '' - \LLL '  } + \varepsilon_{ \LLL ' - \LLL }. 
\]
\end{myremark}

\begin{myremark}
If $\mathscr I$ is a coherent sheaf of ideals on $\XXX$, 
an order function 
\[
\ord_{\mathscr I} : \Gr_\infty ( \XXX ) \to \NN \cup \{ + \infty \}
\]
can be obtained by taking 
	\[
	\ord_{\mathscr I} ( \tilde x ) = \inf_{f \in \mathscr I _{\tilde x}} v_{R'} ( f ( \tilde x )) 
	\]
for all points $\tilde x \in \XXX ( R  ' ) $,
where $v_{R'} ( f ( \tilde x )) = v_{R'} ( \tilde x^* f )$,
see (4.4.3) of \cite[Chap. 4]{chambert2018motivic}.
By Corollary 4.4.8 of \cite[Chap. 4]{chambert2018motivic} this defines a constructible function  $\Gr_\infty ( \XXX ) \to \NN \cup \{ + \infty \}$. 

The affine local description of this function \cite[Chap. 4, Example 4.4.4]{chambert2018motivic}
shows that $\ord_{\mathscr I} ( \tilde x )$ is given by the smallest $v_{R'} ( f ( \tilde x ))$ for $f$ belonging to a finite set of generators of the ideal corresponding to $\mathscr I$. 
In particular, if $\mathscr I$ and $\mathscr I ' $ are two coherent sheaves of ideals,
 with local generators respectively $y_0 \in \mathscr I_{\tilde x }$ and $y_0'  \in \mathscr I_{\tilde x }' $, and whose restrictions $I$ and $I'$ to $X$
are invertible and isomorphic,
then $\ord_{\mathscr I} ( \tilde x )  = v_{R'} ( \tilde x^* y_0 )$ and
for all $y \in I_x \simeq I_x '$
\begin{align*}
\ord_{ \tilde x }' ( y )  - \ord_{  \tilde x }  ( y ) 
 & = v_{R'} ( y / (1/ y_0 ') ) - v_{R '} ( y / (1/y_0)  ) = v_{R '} ( y_0 ' / y_0  ) \\
& = v_{R '} ( \tilde x ^*  y_0  '  ) - v_{R '} ( \tilde x ^* y_0   ) \\
& = \ord_{\mathscr I ' } ( \tilde x ) - \ord_{\mathscr I } ( \tilde x ) . 
\end{align*}
The $(1/y_0)$ here 
comes from the fact that the ideal sheaf associated to an effective Cartier divisor $D$ on $X$ is $\OOO_X ( - D )$. 
\end{myremark}

\begin{mylemma}\label{lemma:local-difference-is-constructible}
	The difference $\varepsilon_{\LLL ' - \LLL }$ is a constructible function on $\Gr_\infty ( \XXX )$.
	By the quasi-compactness of the constructible topology, it takes only finitely many values. 
\end{mylemma} 

\begin{proof} 
Let $n \in \ZZ$.
Assume first that $n\geqslant 0$.
Then $\varepsilon_{ \LLL ' - \LLL  } ( \tilde x ) > n$ if and only if $y_0/y_0 '$ belongs to the $(n+1)$-th power of the maximal ideal of $R$, 
if and only if its class in $R_n$ is zero. 

We claim that there exists $n_0 \in \ZZ$ such that $\varepsilon_{ \LLL ' - \LLL  } ( \tilde x ) \geqslant n_0$
for all $\tilde x \in \XXX ( R ' )$. 
We choose generators $y_0$ and $y_0 ' $ of $\tilde x ^* \LLL$ and $\tilde x ^* \LLL ' $. 
Then the rational  section ${y_0 \over y_0 '} \in K '$ of $ ( \LLL ')^\vee \otimes \LLL$ has a vertical divisor of poles $E$ which is the pull-back of a formal multiple of the closed point of $R '$, and $(\LLL ')^\vee \otimes \LLL \otimes \OOO_\XXX ( - E ) $ is effective. 
Let $z_0$ a generator of $E$, 
then $z_0 {y_0 \over y_0 '} \in R '$ and 
\begin{align*}
	\varepsilon_{ \LLL ' - \LLL  } ( \tilde x )  = v_{R'} ( y_0  / y_0 ' ) = v_{R'} \left ( \frac{y_0 z_0}{y_0 '} \cdot \frac{ 1}{ z_0 } \right )  & = v_{R'} \left ( \frac{y_0 z_0}{y_0 '}\right ) -  v_{R'} \left (  z_0 \right ) \\
& \geqslant -  v_{R'} \left (  z_0 \right ).
\end{align*}
We take $n_0 = -  v_{R'} \left (  z_0 \right )$ so that
\[
\varepsilon_{ \LLL ' - \LLL  } ( \tilde x ) = v_{R'} ( \varpi^{-n_0} y_0 '  / y_0 )  + n_0
\]
for all $\tilde x$. 
Then for a given $n\geqslant n_0$, one has 
$\varepsilon_{ \LLL ' - \LLL  } ( \tilde x ) > n$ if and only if $\varpi^{-n_0} y_0'/y_0 $ belongs to the $(n-n_0+1)$-th power of the maximal ideal of $R$, if and only if its class in $R_{n-n_0}$ is zero. 
Thus 
\[
\{ 
\xi \in \Gr ( \XXX ) \mid \varepsilon_{ \LLL ' - \LLL  } ( \xi ) > n 
\}
\]
is constructible of level $\leqslant (n-n_0)$. 

Moreover it follows from the definition that $\varepsilon_{ \LLL ' - \LLL  }$ does not reach infinity. 
Thus it takes only a finite number of values by the quasi-compactness of the 
constructible topology (see for example Theorem 1.2.4 in \cite[Appendix A]{chambert2018motivic}). 
\end{proof}

\begin{myremark}\label{remark:global-difference-is-bounded}
	Note that this difference is trivial if $\LLL$ and $\LLL ' $ are already isomorphic above $ R $. 
	In particular, if $\LLL$ and $\LLL' $ are two different models on $\VVV \to \CCC$ of the same $L $ on $V$,
	there exists a dense open subset of $\CCC$ above which they are isomorphic. Its complement $\mathsf S$ is a finite set of closed points of $\CCC$ and
	\[
	\deg_{\LLL '} - \deg_\LLL  = \sum_{p\in \mathsf S} \varepsilon_{\LLL_\Rp '  - \LLL_\Rp  } 
	\]
	is bounded. 
\end{myremark}

\begin{mydef}[Motivic density associated to a model of the anticanonical sheaf]
\label{def:motivic-measure-with-respect-model-line-bundle}
\index{motivic density associated to a model of the anticanonical sheaf}
	Let $\XXX$ be an $R$-scheme of pure relative dimension $n$. 
		
	Assume that the generic fibre $X$ is smooth over $K$ and take a model $\LLL_\XXX$ of the anticanonical sheaf $\omega_X^{-1}$ over $\XXX$.
	
	The sheaf $\Omega^1_{\XXX / R }$ of relative differentials of $\XXX$ over $R$ \cite[p. 175]{hartshorne1977algebraic} is a coherent sheaf of $\OOO_\XXX$-modules and the dual of its determinant $(\Lambda^{n } \Omega^1_{\XXX / R } ) ^\vee $ is a model of $\omega_X^{-1}$. 
	
	By \cref{lemma:local-difference-is-constructible} the local difference function $\varepsilon_{\LLL_\XXX - ( \Lambda^{n } \Omega^1_{\XXX / R })^\vee}$ 
	is a constructible function. 
	
	Its motivic integral over any measurable subset $A$ 
	avoiding the singular locus of $\XXX$
	will be written
	\[
	\mu_{\LLL_\XXX}^* ( A )  = 
	\int_{A } \LL^{- \varepsilon_{\LLL_\XXX - ( \Lambda^{n } \Omega^1_{\XXX / R })^\vee}} \mathrm d \mu_\XXX^* 
	\]
	and $\mu_{\LLL_\XXX}^* $ will be called \textit{motivic density associated to} $\LLL_\XXX$. 
\end{mydef}

\subsection{Weak Néron models and smoothening}
\label{subsection-weak-neron-models}

In order to prove our result about invariance by change of model (see \cref{thm:equidistribution-and-models} below), we need to collect a few additional definitions and results about weak Néron models and relations between them. 
References for this subsection
are the third chapter of the book of Bosch, Lütkebohmert and Raynaud \cite{bosch1990neron},
together with the reminder of \S 7.1 in \cite[Chap. 7]{chambert2018motivic}
as well as \S 3.4 in \cite[Chap. 4]{chambert2018motivic}.

\subsubsection{Local models}
We keep the notations of the previous paragraph, except that we do not need to assume that $R$ is complete.

\begin{mydef}
	Let $X$ be a $K$-variety. 
	
	A model for $X$ is a flat separated $R$-scheme of finite type $\XXX$ together with an isomorphism $\XXX_K \to X$.
	
	A \textit{weak Néron model}
	\index{weak Néron model}
	\footnote{We adopt the terminology used by Chambert-Loir, Nicaise and Sebag in \cite{chambert2018motivic}. For the comparison of this definition of weak Néron models with the one given by Bosch, Lütkebohmert and Raynaud in \cite{bosch1990neron}, see Remark 7.1.6 in \cite[Chap. 7]{chambert2018motivic}.
	The main difference is a properness assumption. 
	}
	of $X$ 
	is a model $\XXX$ of $X$ such that $\XXX$ is smooth over $R$
	and 
	every $K'$-point of $X$ extends to an $R'$-point of $\XXX$,
	for every unramified extension $R'$ of $R$ with fraction field $K'$. 
	Since by definition $\XXX$ is separated, such an $R'$-point is unique. 
	
	Let $\YYY$ be a flat separated $R$-scheme of finite type with smooth generic fibre.   
	A \textit{Néron smoothening} 
	\index{Néron smoothening}
	of $\YYY$ is a smooth $R$-scheme $\XXX$ of finite type together with an $R$-morphism $\XXX \to \YYY$
	inducing an isomorphism $\XXX_K \to \YYY_K$ 
	and such that 	
	$\XXX ( R ' ) \to \YYY ( R ' ) $ is bijective for every unramified extension $R'$ of $R$.
	
	Given $\XXX$ and $\XXX '$ two weak Néron models of $X$, a morphism of weak Néron models is a $R$-morphism $\XXX ' \to \XXX $ whose restriction to the generic fibre commutes with the isomorphisms with $X$. In that case we say that $\XXX ' $ dominates $\XXX$. 
\end{mydef}
Then a short reformulation of Theorem 3 and Corollary 4 of \cite[p. 61]{bosch1990neron} is the following fundamental fact.
\begin{mythm}\label{thm:existence-of-weak-neron-models}
	Every $R$-scheme of finite type whose generic fibre is smooth over $K$ admits a Néron smoothening,
	given by the $R$-smooth locus of a composition of admissible blow-ups. 
	\index{weak Néron model!local}
\end{mythm}  

Even if the result is formulated in the language of formal schemes, the proof of Proposition 3.4.7 in \cite[Chap. 3]{chambert2018motivic},
which is a variant of the Néron smoothening algorithm of \cite{bosch1990neron},
 gives the following useful proposition. 
\begin{myptn}\label{ptn:existence-wNm-dominating-2-wNm}
	Let $X$ be a smooth $K$-variety. 
	If $\XXX '$ and $\XXX '' $ are two models of $X$, 
	then 
	there exists another model $\XXX ''$ of $X$ above $\XXX $ and $\XXX '$
	whose $R$-smooth locus $( \XXX '')^\circ$ is a Néron smoothening of both models. 
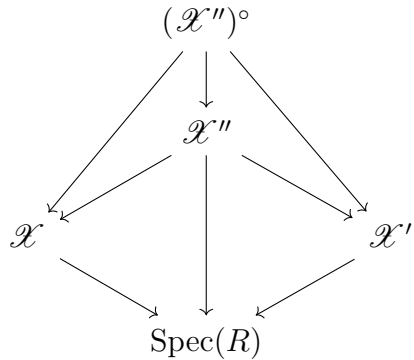
\[
\begin{tikzcd}
		& ( \XXX '')^\circ \arrow[ddr] \arrow[d] \arrow[ddl] & 	\\
		&    \XXX ''\arrow[dd] \arrow[dr] \arrow[dl] & \\
	\XXX \arrow[dr] & 		   & 	\XXX ' \arrow[dl]	\\
		& \Spec ( R ) & 		\\
\end{tikzcd}
\]
\end{myptn}

In the equal characteristic case, we have the following correspondance of points of Greenberg schemes, see Proposition 3.5.1 of \cite[Chap. 4]{chambert2018motivic}.
It allows one to apply the motivic change of variable formula, \cref{ptn:change-of-variable-smooth-case}, to Néron smoothenings. 
\begin{myptn}\label{proposition:neron-smoothening-and-greenberg-schemes}
	Let $\YYY \to \XXX$ be a morphism of separated flat $R$-schemes of finite type,
	restricting to an immersion on generic fibres and such that $\YYY$ is smooth.
	
	Then $\YYY \to \XXX$ is a Néron smoothening if and only if
	the induced map
	\[
	\Gr ( \YYY ) ( \kappa ' )  \to \Gr ( \XXX ) ( \kappa ' )
	\]
	is a bijection for every separable extension $\kappa ' $ of $\kappa $. 
\end{myptn}

\begin{myptn}
\label{ptn:ordjac-Neron-smoothening}
Let $\XXX$ be an $R$-model of a smooth $K$-variety $X$, of pure relative dimension $n$,
and $f : \YYY \to \XXX $ a Néron smoothening of $\XXX$. 

Then
\[
\varepsilon_{\left ( \Omega^n_{\YYY / R } \right ) ^\vee - f^* \left ( \Omega^n_{\XXX  / R } \right ) ^\vee }  = \ordjac_f 
\]
on $\Gr ( \YYY )$.
\end{myptn}
\begin{proof}
This is given by the argument of the chain rule (5.2.2) in \cite[Chap. 7]{chambert2018motivic}
and the fact that in our situation the function $\ordjac_f$
coincides with the order function of the Jacobian ideal of $f$. 
See Lemma 3.1.3 in \cite[Chap. 5]{chambert2018motivic} as well.
\end{proof}

\subsubsection{From local models to global ones} We will need the following gluing result, which is a variant of \cite[p.18, Proposition 1]{bosch1990neron}. 

\begin{myptn}\label{ptn:gluing-dominating-models}
Let $\VVV \to \CCC$ be a model of $V$ above $\CCC$. 
	Let $\CCC_0 $ be a dense open subset of $\CCC$ and 
	$\VVV_0 ', \VVV_1 ' , ... , \VVV_s ' $ 
	a finite number of models of $V$ respectively over $\CCC_0$ and over the local rings of the closed points $p_1 , ... , p_s$ not in $\CCC_0$. 
	Assume that these models dominate respectively the restriction of $\VVV$ to $\CCC_0$ and to these local rings. In particular, they induce isomorphisms on generic fibres. 
	
	Then there exists a model $\VVV ' \to \CCC $ of $V$ extending $\VVV_0 , ... , \VVV_s$
	as well as a $\CCC$-morphism $\VVV ' \to \VVV$ extending the local ones. If moreover the local models are smooth, then $\VVV ' \to \CCC $ is smooth as well. 
	\end{myptn}
\begin{proof}
	Let $R_i$ be the local ring of $\CCC$ at the point $p_i \in \CCC \setminus \CCC_0$ for $i=1 , ... , s$. 
	The morphism $\VVV '_i \to \VVV_{R_i}$ uniquely extends to a $\CCC_i$-morphism for a certain open neighbourhood $\CCC_i$ of $p_i$ by \cite[Théorème 8.8.2]{EGA-IV-3}. 
	Since $( \VVV '_i )_F \simeq ( \VVV_{R_i} )_F \simeq V $, such a model coincides with $\VVV_0 ' \to \CCC_0$ above a non-empty open subset $\CCC_0 ' \subset \CCC_0$. Then, up to removing a finite number of points of $\CCC '_i$ so that $\CCC_i \cap ( \CCC - \CCC_0 '  ) = \{ s_i \} $, we can assume that they coincide above $\CCC_i \cap \CCC_0 ' $ and glue them above each $\CCC_i \cap \CCC '_0$, obtaining a model extending the starting data and dominating $\VVV \to \CCC$.   
\end{proof}

\index{weak Néron model!global}
For us, a weak Néron model of $V$ above $\CCC $ will be a smooth $\CCC$-scheme $\VVV$ of finite type, together with an isomorphism $\VVV_F \to V$
and satisfying the following property concerning étale integral points: 
for any closed point $p$ and any étale local $\OOO_{\CCC , p }$-algebra $R'$ with field of fractions $F '$,
the canonical map $\VVV ( R ' ) \to \VVV_F ( F ' ) $ is surjective (see p.7, Definition 1, as well as the end of p.12, and p.60-61 in \cite{bosch1990neron}).  

\begin{mycor}\label{cor:existence-global-wNm-dominating-2-global-wNm}
	Let $\VVV$ and $\VVV ' $ be two models of $V$ above $\CCC $. Then there exists a third model $\VVV '' $ of $V$ above $\CCC$ whose $\CCC$-smooth locus is a Néron smoothening of both $\VVV $ and $\VVV ' $ above $\CCC$. 
	\[
	\begin{tikzcd}
				& \VVV '' \arrow[dd] \arrow[dr]\arrow[dl]  & \\
		\VVV \arrow[dr]	& 									& 	\VVV ' \arrow[dl] \\
				& 			\CCC 							& 
	\end{tikzcd}
	\]
	Moreover, if $\VVV$ and $\VVV ' $ are proper, then $\VVV '' $ can be taken proper as well. 
\end{mycor}
\begin{proof}
	By spreading-out \cite[Théorème 8.8.2]{EGA-IV-3} applied to the generic point of $\CCC$, 
we know the existence of a non-empty open subset $\CCC '   \subset \CCC$ 
such that the restrictions of $\VVV $ and $\VVV ' $ above $\CCC' $ 
are isomorphic as $\CCC ' $-schemes.  
\[
\begin{tikzcd}
\VVV_{\CCC ' } \arrow[d] \arrow[dr] \arrow[rr,"\sim"] &       & \VVV_{\CCC ' } ' \arrow[d] \arrow[dl] \\
\VVV \arrow[dr] & \CCC ' \arrow[d] & \VVV ' \arrow[dl] \\
 		& \CCC & 
\end{tikzcd}
\]
Then, for each closed point $p$ in the complement of $\CCC ' $, by \cref{ptn:existence-wNm-dominating-2-wNm} we can find a weak Néron model which dominates both restrictions of $\VVV $ and $\VVV ' $ to $\Spec ( \OOO_{\CCC , p} )$. One can use \cref{ptn:gluing-dominating-models} to glue $\VVV_{\CCC ' }$ together with these models and get the desired new weak Néron model. 
These operations preserve properness, by the valuative criterion. 
\end{proof}

\subsection{Piecewise trivial fibrations}
In this subsection $S$ is a Noetherian scheme.
We recall definitions and properties of (classes in $\KVar S$) of piecewise trivial fibrations, following \cite[Chap. 2, \S 2.3]{chambert2018motivic}.

\begin{mydef}Let $F$ be an $S$-variety.
A \textit{piecewise trivial fibration with fibre $F$}
is a morphism of schemes $f : X \to Y$ between two $S$-varieties 
such that there exists a finite partition $(Y_i )_{i\in I}$ 
of $Y$ into locally closed subsets 
together with an isomorphism of $(Y_i)_{\text{red}} $-schemes between $(X\times_Y Y_i)_{\text{red}} $ and $( F \times_S Y_i )_{\text{red}}$ for all $i\in I$.  
\end{mydef}

\begin{myptn}\label{ptn:classes-of-piecewise-trivial-fibrations}
Let $F$ be an $S$-variety and $f: X \to Y$ a piecewise trivial fibration with fibre $F$. 
Then 
\[ 
[ X ] = [ F ] [ Y ] 
\]
in $\KVar S$.
\end{myptn}
\begin{proof}
See Corollary 1.4.9 and Proposition 2.3.3 of \cite[Chap. 2]{chambert2018motivic}.
\end{proof}

We will make extensive use of the following criterion.
\begin{myptn}\label{ptn:point-wise-criterion-for-piecewise-trivial-fibrations}
Let $F$ be an $S$-variety and $f: X \to Y$ a morphism of $S$-varieties.
Then $f$ is a piecewise trivial fibration with fibre $F$ if and only if, for every point $y \in Y$, the $\kappa ( y ) $-schemes
$f^{-1} ( y )_{\text{red}}$ and $( F \otimes_k \kappa ( y ) ) _{\text{red}}$ are isomorphic. 
\end{myptn}
\begin{proof}
See Proposition 2.3.4 of \cite[Chap. 2]{chambert2018motivic}
(one proceeds by Noetherian induction and applies \cite[8.10.5]{EGA-IV-3}). 
\end{proof}

\bigskip


\section{Motivic Euler products}

\subsection{The weight filtration} \label{subsection:weight-topology}
\label{section:weight-filtration}
In this paragraph $S$ is a variety over a subfield $\kb$ of the field $\CC$ of complex numbers. We fix once and for all an embedding of $\kb$ in $\CC$ and consider that $S$ is actually defined over $\CC$ by extension of scalars. We briefly recall the construction of a weight filtration on the Grothendieck ring of varieties over $S$.
We use \cite{peters2008mixed} as a general reference for Mixed Hodge Modules, as well as the summaries of \cite[Chapter 4]{bilu2018motivic} and \cite[Chap. 2, \S 3.1-3.3]{chambert2018motivic}. 

\subsubsection{Mixed Hodge modules}
\index{category $\MHM_S$ of mixed Hodge modules over $S$}
The category $\MHM_S$ of mixed Hodge modules over $S$ was introduced by Saito in \cite{saito1988modules,saito1990mixed}. It is an abelian category which provides a cohomological realization of the Grothendieck group $\KVar S$ of $S$-varieties. Its derived category is endowed with a six-functors formalism \textit{à la Grothendieck}. 
In case $S=\Spec ( \CC ) $ is a point, mixed Hodge modules over $S$ coincide with polarizable Hodge structures as defined by Deligne \cite{deligne1971theorie}, see \cite[Lemma 14.8]{peters2008mixed}.   

The Grothendieck group $K_0 ( \MHM_S ) $ of mixed Hodge modules over $S$ is the quotient of the free abelian group of isomorphism classes of mixed Hodge modules over $S$ by the relations 
\[
[E] - [F] + [G]
\]
whenever there is a split exact sequence 
\[ 
0 \to E \to F \to G \to 0 
\]
for $E$, $F$ and $G$ objects of $\MHM_S$.
There is a notion of weight of a mixed Hodge module, morphisms are strict for the weight filtration and the Grothendieck group $K_0 ( \MHM_S ) $ is generated by the classes of pure Hodge modules.

The tensor product operation in the bounded derived category of $\MHM_S$ provides a multiplicative structure on $K_0 ( \MHM_S ) $ as follows. 
The Grothendieck group $K_0^{\text{tri} } ( D^b ( \MHM_S))$ is defined similarly to $K_0 ( \MHM_S )$ by taking distinguished triangles 
\[
E^\bullet \to F^\bullet  \to G^\bullet \to E^\bullet [1] 
\]
of complexes as relations, in the place of exact sequences \cite[Exposé VIII]{SGA-V}. 
By the theorem of decomposition of mixed Hodge modules \cite[Cor. 14.4]{peters2008mixed},
there is an isomorphism of groups
\[
K_0 ( \MHM_S ) \to K_0^{\text{tri}} ( D^b ( \MHM_S )) 
\]
sending the class of a mixed Hodge module $M$ to the class of the complex with $M$ in degree zero. Indeed, the inverse is given by the morphism
\[
[M^{\bullet} ] \mapsto \sum_{i\in \ZZ} ( - 1 )^i [ \mathcal H^i ( M^\bullet ) ]
\]
from $K_0^{\text{tri} } ( D^b ( \MHM_S)) $ to $K_0 ( \MHM_S ) $. The tensor product on $D^b ( \MHM_S ) $ induces a ring structure on $K_0^{\text{tri}} ( D^b ( \MHM_S ) ) $ and on $K_0 ( \MHM_S ) $ through the previous isomorphism. 

The faithful and exact functor
\[
\rat_S : \MHM_S \to \Perv_S 
\]
to perverse sheaves on $S$
sends the six functors formalism of $\MHM_S$ to the one of the bounded derived category of constructible sheaves $D^b_c ( S ) $. 
In order to prove an isomorphism between two mixed Hodge modules, it will be enough to check it at the level of perverse sheaves.
Indeed, 
a mixed Hodge module $M$ is given by the data of filtrations 
on a $\mathcal D$-module isomorphic to $\CC \otimes_\QQ \rat_S ( M ) $ by the Riemann-Hilbert correspondence (comparison isomorphism, \cite[\S 14.1]{peters2008mixed}).  
Moreover, the Verdier duality functor $\mathbf D_S$ on $D^b_c ( S ) $ lifts to $\MHM_S$ so that $\rat_S \circ \mathbf D_S = \mathbf D_S \circ \rat_S$. 

\subsubsection{The Hodge realisation of $\KVar S$}For every integer $d$ 
we denote by $\QQ_S^\Hdg (-d)$ the complex of mixed Hodge modules 
obtained by pulling back to $S$ 
the Hodge structure $\QQ_{\text{pt}}^\Hdg ( -d )$ of type $(d,d)$ 
through the structure morphism $S\to \Spec ( \CC ) $.  
If $p : X\to S$ is an $S$-variety, 
let $\QQ_X^\Hdg$ be the complex $p^* \QQ_S^\Hdg$ of mixed Hodge modules.
\begin{mydef}
The \textit{Hodge realisation}
\index{Hodge realisation}
\[
\chi_S^\mathrm{Hdg} : \KVar S \to K_0 \MHM _S,
\]
sometimes called the \textit{motivic Hodge-Grothendieck characteristic},
sends a class 
$\left [X \overset{p}{\to} S \right ]$ 
to 
\[
\left [p_! \QQ_X^\Hdg \right ]
=
\sum_{i\in \mathbf Z} ( - 1 )^i \left [ \mathcal H^i \left ( p_{!} \QQ_X^\text{Hdg} \right ) \right ]
\]
where the equality comes from the isomorphism 
\[
K_0 ( \MHM_S) \overset{\sim}{\to} K_0^{\text{tri}} ( D^b ( \MHM_S )) 
\]
described in the previous paragraph. 
\end{mydef}
The perverse realisation $\KVar S \to K_0 ( D^b ( \Perv_S ) )$ factors through this morphism 
as $\rat_S \circ \chi_S^\Hdg$ 
\cite[Chap. 2, Proposition 3.3.7]{chambert2018motivic}. 
If $S= \Spec ( \CC ) $ this is the class,
in the Grothendieck group of mixed Hodge structures, 
of the cohomology with compact support of $X$ and rational coefficients, together with its Hodge structure. 
This homomorphism is well-defined (see \cite[Lemma 16.61]{peters2008mixed} or Proposition 3.3.7 in \cite[Chap. 2]{chambert2018motivic} for a proof).
Since $\chi_S^\Hdg$ 
sends $\LL_S$ 
to $\left [ \QQ_S^\Hdg ( -1 ) \right ]$,
which is invertible in $K_0 ( \MHM_S ) $, 
it is a morphism of rings compatible with the localisation $\KVar S \to \mathscr M_S$. 

\subsubsection{The weight filtration on $\KVar S$} In this paragraph we collect definitions and properties from \cite[\S 4.6]{bilu2018motivic} and develop a few useful examples about weights. 

\begin{mydef}
The weight function $w_S : \MMM_S \to \ZZ $ is given by the composition of $\chi^\Hdg_S$ together with the weight function on mixed Hodge modules. 	
\end{mydef}

\begin{myptn}\label{proposition:properties-of-the-weight}
 Let $S$ be a complex variety. The weight function $w_S : \MMM_S \to \mathbf Z $ satisfies the following properties.
\begin{enumerate}
\item $w_S ( 0 ) = - \infty  $.
\item $w_S ( \mathfrak a + \mathfrak a '  ) \leqslant \max ( w_S ( \mathfrak a ), w_S (   \mathfrak a ' ) )$ with equality if  $w_S ( \mathfrak a ) \neq w_S (   \mathfrak a ' ) $, for any $\mathfrak{a}, \mathfrak{a}' \in \MMM_S$. 
\item If $\YYY \to S$ is a variety over $S$ then 
\[
w_S  ( \YYY ) = 2 \dim_S ( \YYY ) + \dim ( S ) . 
\] 
\end{enumerate}
\end{myptn}
For proofs of these properties, see  Lemmas 4.5.1.3, 4.6.2.1 and 4.6.3.1 of \cite[Chap. 4]{bilu2018motivic}. 
This weight function induces a filtration $(W_{\leqslant n} \MMM_S)_{n\in \mathbf Z}$ on $\MMM_S$ given by 
\[
W_{\leqslant n} \MMM_S = \{ \mathfrak a \in \MMM_S \mid w_S ( \mathfrak a ) \leqslant n \}
\]
for all $n\in \mathbf Z$.  
\begin{mydef}
\index{completion of $\MMM_S$!w.r.t the weight}
The \textit{completion of $\MMM_S$ with respect to the weight topology} is the projective limit 
\[
\widehat{\MMM_S}^w = \underset{\longleftarrow}{\lim} \left ( \MMM_S / W_{\leqslant n} \MMM_S \right ). 
\]
\end{mydef}

\bigskip 

\subsubsection{Useful examples and vanishing properties} Let $n$ be the dimension of the complex variety $S\overset{a_S}{\to}\Spec ( \CC ) $.
If $S$ is smooth, the 
complex 
$\QQ_S^\Hdg$
of mixed Hodge modules 
is concentrated in degree $n$
and 
$\mathcal H^n \QQ_S^\Hdg$ 
is given by the pure Hodge module of weight $n$ associated to the constant one dimensional variation of Hodge structure on $S$ of weight zero.  
Furthermore, we have the relation $a_S^! \cong a_S^* ( n ) [2n]  $, in particular ${\QQ_S^\Hdg ( n ) [2n] \cong a^!_S \QQ_{\Spec ( \CC ) }^\Hdg}$. 

The class of $\mathbf A^d_S = \mathbf A^d_{\CC} \times_\CC S$ is sent by $\chi_S^{\mathrm{Hdg}}$ to 
\[
\chi_S^\Hdg \left ( \LL_S^d \right ) 
= 
\left ( \chi_S^\Hdg  ( \LL_S ) \right )^{\otimes d} 
= 
\left ( \pr_! \QQ_{\mathbf A^1_S}^\Hdg  \right )^{\otimes d} 
= 
\QQ_S^\Hdg ( - d ). 
\]
More generally, we have the following proposition on top-graded parts. 
\begin{myptn}[{\cite[Lemma 4.6.3.4]{bilu2018motivic}}] \label{ptn:weight-of-the-difference-smooth-irreducible-fibres}
Let $S$ be a smooth and connected complex variety of dimension $n$. 
Let $p : \YYY \to S$ and $\ZZZ \to S $ be two smooth $S$-varieties with irreducible fibres of dimension $d\geqslant 0$. 
Then 
\[
w_S ( [\YYY ] - [\ZZZ ] ) \leqslant 2 d + n - 1 .
\]
\end{myptn}
\begin{proof}
Since $p$ is smooth, one has $p^! \simeq p^* (d) [2d] $ and there is a morphism of mixed Hodge modules
\[
\mathcal H^{2d+n} \left ( p_! \QQ_\YYY^\Hdg \right ) 
\to 
\mathcal H^{2d+n} \left ( \QQ^\Hdg_S (-d) [-2d] \right )
\]
which induces an isomorphism on the $(2d+n)$-th graded parts \cite[Remark 4.1.5.5]{bilu2018motivic}. This means that if $\ZZZ \to S$ is another smooth $S$-variety with irreducible fibres of dimension $d$, the corresponding top-weight graded parts cancel out and the weight of $\chi_S^\Hdg ( [\YYY] - [\ZZZ] )  $ is at most $2d + n - 1$. 
\end{proof}

\begin{myptn}\label{proposition:hodge-realisation-almost-fano-varieties}
Let $S$ be a smooth and connected complex variety of dimension $n$. 
Let ${p : \YYY \to S}$ be a proper smooth morphism 
whose fibres are smooth projective varieties of dimension $\dim_S ( \YYY) = d$.  
Assume that 
\[
\mathcal R^1 p_! \OOO_\YYY  = \mathcal R^2 p_! \OOO_\YYY = 0 
\]
and that $p$ has local sections. 

Then there exists an open subset  $S ' \subset S$
above which the relative Picard scheme exists, is smooth with discrete fibres of rank $r$,
and such that the class 
\begin{align*}
& \chi_{S'}^\Hdg ( \YYY ) -  \chi_{S'}^\Hdg ( \LL_{S'}^d )- \chi_{S'}^\Hdg ( r \LL_{S ' }^{d-1} ) \\
& = \left [ p_! \QQ_{\YYY_{| S' } }^\Hdg \right ] -\left  [ \QQ_{S'}^\Hdg (-d ) [-2d]\right  ] -  \left [  \QQ_{S'}^\Hdg ( -(d-1) )[-2(d-1)]^{\oplus r } \right ]
\end{align*}
has $S '$-weight at most $2d + n - 3 $.  
\end{myptn} 
\begin{proof}
Recall that if a complex $M$ is concentrated in degree $n$, then its shifting 
\[
{M[k]^\bullet = M^{\bullet + k}}
\] is concentrated in degree $n'$ such that $n' + k = n$, that is to say in degree $ n - k$
(in what follows, we will use this with $k=-2d$).

Since $S$ is smooth and connected, the complex $\QQ_S^\Hdg $ is concentrated in degree ${n = \dim (S)}$ and 
\[
\mathcal  H^{2d+n} ( \QQ^\Hdg_S (-d) [-2d] )
\]
is a pure Hodge module of weight $2d+n$. 
As complexes, $\QQ_\YYY^\Hdg $ and $\ZZ_\YYY^\Hdg$ are both concentrated in degree $d + n $, which will explain the shift in what follows. 
The morphism $p$ induces functors $p_! : D^b ( \YYY ) \to D^b ( S )$ and $p_*$ between the bounded derived categories of sheaves respectively over $\YYY$ and $S$, 
compatible with the ones
on mixed Hodge modules.  
Since $p$ is proper, $p_*$ and $p_!$ coincide. 
By \cite[Lemma 4.1.4.2]{bilu2018motivic}, 
the functor $p_! : D^b ( \MHM_\YYY ) \to D^b ( \MHM_S) $
has cohomological amplitude $\leqslant d$,
which means that 
the complex $p_! \QQ_\YYY^\Hdg $ 
only has cohomology in degree at most $(d+n)+d = 2d + n$. 
We deduce that only terms of weight $\leqslant 2d+n$
occur 
in $\chi_{S'}^\Hdg ( \YYY )$. Hence, to prove the proposition, we have to compute the terms of weights $2d+n$, $2d+n-1$ and $2d+n-2$. 
By \cref{ptn:weight-of-the-difference-smooth-irreducible-fibres} and its proof,
we already know what is the part of weight $2d+n$.

The exponential exact sequence of sheaves of abelian groups over $\YYY$
\[
0 \to \ZZ_{\YYY} ( 1 ) \to \mathbf G_{a , \YYY}  \to  \mathbf G_{m , \YYY} \to 0 
\]
gives rise to an exact sequence of cohomology sheaves over $S$
\begin{equation}\label{exact-seq:exp-cohomology-sheaf} 
\begin{tikzcd}
  \cdots \arrow[r] &  \mathcal H^n ( p_! \ZZ_\YYY (1) ) \rar & \mathcal H ^n ( p_! \OOO_\YYY [-n] ) \rar & \mathcal H^n ( p_! \OOO_\YYY^\times [-n] ) \ar[out=-30, in=150]{dll} \\ 
 & \mathcal H^{n+1} ( p_! \ZZ_\YYY (1)  ) \rar & \mathcal H^{n+1} ( p_! \OOO_{\YYY} [-n]) \rar & \mathcal H^{n+1} ( p_! \OOO_{\YYY}^\times [-n] ) \ar[out=-30, in=150]{dll} \\
 & \mathcal H^{n+2} ( p_! \ZZ_\YYY (1) ) \rar & \mathcal H^{n+2}( p_! \OOO_{\YYY} [-n] )  \rar & \cdots   &\\
\end{tikzcd}
\end{equation}
which on stalks specialises to the well-known exact sequence of abelian groups
\[
	\begin{tikzcd}
  0 \rar &  \ZZ ( 1 ) \rar & \CC \rar & \CC^* \ar[out=-30, in=150]{dll} \\ 
 & H^1 ( \YYY_s ( \CC ) , \ZZ (1) ) \rar & H^1 ( \YYY_s , \OOO_{\YYY_s}) \rar & H^1 ( \YYY_s , \OOO_{\YYY_s}^\times ) \ar[out=-30, in=150]{dll} \\
 & H^2 ( \YYY_s ( \CC ) , \ZZ (1 ) ) \rar &  H^2( \YYY_s , \OOO_{\YYY_s})  \rar & \cdots   &\\
\end{tikzcd}
\]
where $\ZZ ( 1 ) $ is the $\ZZ$-Hodge structure with underlying $\ZZ$-module $2i\pi \ZZ$ and Hodge type $(-1,-1)$. 
Since $\mathcal H^{ n + i} ( p_! \mathscr O_\YYY [-n] )  = \mathcal R^i p_! \mathscr O_\YYY = 0  $ for $i=1,2$, the map 
\[
\Pic_{\YYY / S }  = \mathcal H^{n+1} ( p_!  \OOO_\YYY^\times [-n] ) \to \mathcal H^{n+2}  (p_! \ZZ_\YYY  (  1 ) )
\]
is an isomorphism. 
In particular, the map it induces on stalks
\[ 
\Pic ( \YYY_s ) = H^1 ( \YYY_s , \OOO_{\YYY_s} ^\times ) \to H^2 ( \YYY_s ( \CC ) , \ZZ (1) ) 
\]
is an isomorphism. 
Since we assumed that $S$ and $\YYY \to S $ are smooth,
this means that 
\begin{equation}\label{eq:picard-scheme-Y->S}
	\Pic_{\YYY / S  } \otimes \QQ_S^\Hdg \simeq \mathcal H^{n+2} ( p_! \QQ_\YYY^\Hdg ( 1 )  ) 
\end{equation} is a variation of Hodge structure, of rank $r$, above $S$. 

By surjectivity of the exponential map,
the arrow $\CC^* \to H^1 ( \YYY_s ( \CC ) , \mathbf Z ( 1 ) )$ is trivial.
Thus $H^1 ( \YYY_s ( \CC ) , \mathbf Z ( 1 ) )$ injects into $H^1 ( \YYY_s , \OOO_{\YYY_s} ) $ 
which is trivial by assumption, 
thus $H^1 ( \YYY_s ( \CC ) , \mathbf Z ( 1 ) )$ is trivial as well for all $s \in S$, which means that $\mathcal H^{n+1} ( p_! \ZZ_\YYY ( 1 ) ) $ is trivial. 

We use the exact involutive dual functor 
\[
\mathbf D : \MHM  \to \MHM^{\mathrm{opp}}
\]
of Verdier duality on mixed Hodge modules,
which extends to the derived bounded category $D^b_c ( \MHM )$, 
by the formula
\[
\mathbf D M^\bullet = \HOM ( M^\bullet , \mathbf D \QQ^\Hdg ) 
\]
in $D^b_c ( \MHM )$, where $\mathbf D \QQ^\Hdg$ is the dualizing complex. 
Here since $S$ is smooth, 
\[
\mathbf D_S \QQ^\Hdg_S \simeq \QQ_S^\Hdg ( n ) [2n] ,\]
see for example \cite[Appendix A]{saito2016young}. 
It sends mixed Hodge modules of weight $w$ to mixed Hodge modules of weight $-w$, and interchanges $p_*$ and $p_!$. In our situation, it gives
\[
\mathbf D_S ( p_! \QQ_\YYY^\Hdg ) = p_*  (\mathbf D_\YYY \QQ_\YYY^\Hdg) = p_!  (\mathbf D_\YYY \QQ_\YYY^\Hdg) = p_! ( \QQ_\YYY^\Hdg ( n + d) [ 2 ( n+ d )] )
\]
thus 
\begin{align}
\label{equation-MHM-shift-D-operator}
\Gr_i^W \mathcal H^j ( p_! \mathbf D_\YYY \QQ_\YYY^\Hdg )  
& = 
\Gr_i^W \mathcal H^j ( p_! ( \QQ_\YYY^\Hdg ( n + d) [ 2 ( n+ d )] ) ) \notag \\
& =
\Gr_{i+2(n+d)}^W \mathcal H^{j+2(n+d)} ( p_! \QQ_\YYY^\Hdg )  
\end{align}
for all integers $i$ and $j$ 
(recall that the $(n+d)$th Tate twist translates into a double shift $2(n+d)$ of the weight).
On the other hand the decomposition theorem \cite[Corollary 14.4]{peters2008mixed}
\[
p_! \QQ_\YYY^\Hdg \simeq \oplus_{k} \mathcal H^k ( p_! \QQ_\YYY^\Hdg ) [ - k ] 
\]
in $D^b ( \MHM_S)$, gives
\[
\mathbf D_S ( p_! \QQ_\YYY^\Hdg ) \simeq \oplus_{k} ( \mathbf D_S \mathcal H^k ( p_! \QQ_\YYY^\Hdg ) ) [ k ] .
\]
We apply \cite[Prop. 2.6]{saito1990mixed}, which says that  
\[
\DD_S \Gr_i^W M = \Gr_{-i}^W \DD_S M 
\] 
for all $M\in \MHM_S$, which gives
\begin{align*}
\Gr_{i+2(n+d)}^W \mathcal H^{j+2(n+d)} ( p_! \QQ_\YYY^\Hdg )  & = \Gr_{i}^W  \mathcal H^j ( \mathbf D_S ( p_! \QQ_\YYY^\Hdg ) ) \tag{by \eqref{equation-MHM-shift-D-operator}} \\
& \simeq \Gr_{i} ^W  \mathbf D_S ( \mathcal H^{-j} p_! \QQ_\YYY^\Hdg ) \\
& \simeq \mathbf D_S \Gr_{-i}^W \mathcal H^{-j} ( p_! \QQ_\YYY^\Hdg ) 	
\end{align*}
which, in turn, for $ j \in \{ - ( n +1 ), - ( n+2 ) \}$ and $i = j $, specializes to
\begin{align*}
	\Gr_{n+2d - 1 }^W \mathcal H^{n+2d - 1} ( p_! \QQ_\YYY^\Hdg ) & \simeq \mathbf D_S\Gr_{n+1} ^W   \mathcal H^{n+1}  ( p_! \QQ_\YYY^\Hdg ) \\
\Gr_{n+2(d - 1) }^W \mathcal H^{n+2(d - 1)} ( p_! \QQ_\YYY^\Hdg ) & \simeq \mathbf D_S \Gr_{n+2} ^W  \mathcal H^{n+2} ( p_! \QQ_\YYY^\Hdg )  .
\end{align*}
We previously showed (in the paragraph right after \eqref{eq:picard-scheme-Y->S}) that $\mathcal H^{n+1} ( p_! \ZZ_\YYY^\Hdg (1) )$ is trivial, we have in particular
\[
\mathcal H^{n+1} ( p_!  \QQ_{\YYY}^\Hdg  ) = 0  
\]
and
the first line is zero. 
Finally, since by \eqref{eq:picard-scheme-Y->S}
\[
\mathcal H^{n+2} ( p_! \QQ_\YYY^\Hdg ) \simeq \Pic_{\YYY / S  } \otimes \QQ_S^\Hdg ( - 1 ) 
\]
is pure, the local rank of 
\begin{align*}
	\Gr_{n+2(d - 1) }^W \mathcal H^{n+2(d - 1)} ( p_! \QQ_\YYY^\Hdg ) & \simeq \mathbf D_S \Gr_{n+2} ^W  \mathcal H^{n+2} ( p_! \QQ_\YYY^\Hdg )  \\
	& \simeq \mathbf D_S \Gr_{n+2} ^W \left ( \Pic_{\YYY / S  } \otimes \QQ_S^\Hdg ( - 1 ) \right )
\end{align*}
is given by the one of $\Pic_{ \YYY / S' } \otimes \QQ_{S'}^\Hdg $ above an open subset $S' \subset S$, hence the result. 
\end{proof}

\subsection{Motivic Euler product} 
Formal motivic Euler products have been introduced by Margaret Bilu \cite{bilu2018motivic}, 
as a \textit{notation} generalizing the Kapranov zeta function and \textit{behaving like a product}. 
For our purpose we will only need a particular case of this construction, but we will state a few useful properties of this object in a general framework. 
We mostly follow the exposition one can find in the third and sixth section of \cite{bilu2021motivic}.

\subsubsection{Symmetric products and configuration spaces}
\label{subsection:symmetric-products}

Let $S$ be a $\kb$-variety and $X$ an $S$-variety.
The $m$-th symmetric product of $X$ relatively to $S$ is by definition the quotient 
\[
\Sym_{/S}^m ( X ) = ( \underbrace{ X \times_S ... \times_S X}_{m \text{ times}} ) / \mathfrak S_m. 
\]
 Let ${\XXX = (X_i)_{i  \in I}}$ be a family of quasi-projective varieties above $X$, where $I$ is an arbitrary set. Let  $\mu = (m_i)_{i\in I} \in \NN^{(I)} $ be a family of non-negative integers with finite support, which we call a \textit{partition} (if $I=\mathbf N^*$, then a partition of a non-negative integer $n$ is a family $(m_i)_{i\geqslant 1}$ such that $\sum_{i\geqslant 1} im_i = n$).\footnote{Note that such partitions admit holes and that this set $\NN^{(I)}$ of generalised partitions is denoted by $\mathcal P ( I )$
in \cite{vakil2015discriminants,howe2019motivic,bilu2020zeta,bilu2021motivic}, while the set of partitions with no hole is written $\mathcal Q ( I )$. We will adopt the notations $\mathcal P$ and $\mathcal Q$ only for partitions of integers, that is to say elements of $\NN^{(\NN^*)}$.} 
For such a partition, we define
\[
\Sym^\mu_{/S} ( X )  = \prod_{i\in I} \Sym^{m_i}_{/S} ( X ) 
\]
as well as 
\[
\Sym^\mu_{X/S} ( \XXX )  = \prod_{i\in I} \Sym^{m_i}_{X/S} ( X_i ) 
\]
which is a variety over $\Sym^\mu_{/S} ( X )$. 
These constructions extend to elements of $\KVar X$,
using Cauchy products; for details, see for example \cite[\S 6.1.1]{bilu2021motivic}. 

Given a partition $\mu \in \NN^{(I)}$, one can construct the restricted $\mu$-th symmetric product 
\[
\Sym^\mu_{X/S} ( \XXX )_* 
\]
as follows. 
If we write $\left (\prod_{i\in I} X^{m_i}\right )_{\ast , X /S }$ for the complement of the diagonal (points having at least two equal coordinates) in $\prod_{i\in I} X^{m_i} $,
then the restricted symmetric product
\[
\Sym^\mu_{/S} ( X )_*
\]
sometimes abbreviated 
\[
\Sy^\mu_{/S} ( X )_* 
\]
is by definition the image of $\left (\prod_{i\in I} X^{m_i}\right )_{\ast , X /S }$
in $\Sym^\mu_{/S} X$.
Furthermore, there is a Cartesian diagram
\begin{center}
\begin{tikzcd}
\left (\prod_{i\in I} X_i^{m_i}\right )_{\ast , X /S } \arrow[r , hook] \arrow[d] & \prod_{i\in I} X_i^{m_i}  \arrow[d] \\
\left (\prod_{i\in I} X^{m_i}\right )_{\ast , X /S } \arrow[r , hook]  &  \prod_{i\in I} X^{m_i} 
\end{tikzcd}
\end{center}
defining an open subset $\left (\prod_{i\in I} X_i^{m_i}\right )_{\ast , X /S }$ of points of $\prod_{i\in I} X_i^{m_i}  $ mapping to points of $\prod_{i\in I} X^{m_i} $ with pairwise distinct coordinates. 
Then one defines 
\footnote{Denoted by $S^\mu  ( \XXX / S )$ in \cite{bilu2018motivic}
and $C^\mu_{X/S} ( \XXX )$ or $\left ( \prod_{i\in I} \Sym^{m_i} X_i \right )_{\ast , X/ S}$ in \cite{bilu2021motivic,bilu2020zeta}.} 
\[
\Sy^\mu_{X/S} ( \XXX )_* 
= 
\Sym^\mu_{X/S} ( \XXX )_* = \left (\prod_{i\in I} X_i^{m_i}\right )_{\ast , X /S } \left / \prod_{i\in I} \mathfrak{S}_{m_i} \right.
\]
that is to say, the image of $\left (\prod_{i\in I} X_i^{m_i}\right )_{\ast , X /S }$ in $\Sym^\mu_{X/S} ( \XXX )$.

\begin{myexample}
In the case where $I$ is a singleton, and $\XXX = ( Y \to X )$, then any partition  $ \mu $  is given by a non-negative integer $  n  $ 
and $\Sym^\mu_{X / S} ( \XXX )_* = \Sym^n_{X / S} ( Y )_*$ 
is the scheme parametrizing étale zero-cycles of degree $n$ above $X$, with labels in $Y $.
\end{myexample}

\begin{myexample}
If $I = \mathbf N^r \setminus \{ \mathbf 0 \} $, 
then 
for any $\nn \in \NN^{r} \setminus \{ \mathbf 0 \}$
the disjoint union 
\[
\Sym^\nn_{X / S} ( \XXX )_* = \coprod_{\substack{\mu = ( n_\mm ) \in \NN^{( ( \NN^r)^*)} \\ \sum_\mm n_\mm \mm = \nn }} \Sym^\mu_{X / S} ( \XXX )_*
\]
parametrizes $r$-tuples of zero-cycles of degree $\nn$ with labels in $\XXX$.  
\end{myexample}

As well, this construction extends to families of elements of $\KVar X$ and $\mathscr M_X$ \cite[Definition 6.1.7]{bilu2021motivic}:
if $\AAA = ( \mathfrak a_i )_{i\in I}$ is such a family, then 
\[
\Sym^\mu_{X/S} ( \AAA ) = \boxtimes_{i\in I} \Sym^{m_i}_{X/S} ( \mathfrak a_i ) \in \KVar {\Sym^\mu_{/S} X }
\] 
and
\[
\Sy^\mu_{X/S} ( \AAA )_* 
= 
\Sym^\mu_{X/S} ( \AAA )_* \in \KVar {\Sym^\mu_{/S} ( X )_* }
\]
is the restriction to
$\Sy^\mu_{/S} ( X )_* = \Sym^\mu_{/S} ( X )_* \subset \Sym^\mu_{/S} ( X )$ of $\Sym^\mu_{X/S} ( \AAA ) $. 
More generally,
if $K$ is a class in $\KVar {\Sym^\mu_{/S} X}$,
we will denote by $K_\ast$ its image in $\KVar {\Sym^\mu_{/S} (X)_* }$
by the restriction morphism. 

\subsubsection{Formal and effective motivic Euler products} 

\begin{mynotation}[Formal motivic Euler product]
\label{def:abstract-euler-product} 
\index{motivic Euler product!formal}
Let  $X$ be a variety over $S$
and 
$\XXX = ( X_i )_{ i \in I }  $ be a family of elements of $\KVar X $ or $\MMM_X$ indexed by a set $I$.
Let $(t_i)_{i \in I}$ be a family of indeterminates.  
Then the product
\[
\prod_{x \in X / S } \left ( 1 + \sum_{i \in I} X_{i,x}  t_i \right ) 
\]
is defined as a \textit{notation} for the formal series
\[
\sum_{ \mu \in \NN^{(I)}} [  \Sym^\mu_{X / S} ( \XXX )_* ] \TT^\mu 
\]
where $\TT^\mu = \prod_{i\in I}t_i^{m_i} $ whenever $\mu = ( m_i )_{i\in I} \in \NN^{(I)}$. 
\end{mynotation}

\begin{myexample}[Formal motivic Euler product with one indeterminate over a curve]\label{def:abstract-euler-product-over-curves} 
\index{motivic Euler product!over a curve}
The simplest kind of motivic Euler products we are going to use is the following. Assume that $I$ is made of a single element. Then a family $\XXX$ is given by a single class $Y$ and $\mu = [ \underbrace{1 , ... , 1}_{m \text{ times}}]$ is the only relevant partition type for a given positive integer $m$. In this setting,  
\[
\prod_{p \in \CCC } \left ( 1 + Y t  \right ) 
\]
is the formal series
\[
 \sum_{m \in \NN} [ \Sym^m_{\CCC } ( Y )_* ] t^m 
 \]
 where $ \Sym^m_{\CCC } ( Y )_* = \Sym^{ [ 1 , ... , 1 ] }_{\CCC } ( Y )_*  $ parametrises étale zero cycles of $\CCC$, of degree $m$ with labels in $Y$ (whenever $Y$ is a variety). 
\end{myexample}  

 \begin{myptn}[{\cite[\S 3.8.1]{bilu2018motivic}}]\label{ptn:motivic-Euler-product-compatible-cut-past}
 \index{motivic Euler product!multiplicativity}
 The Euler product notation is compatible with the cut-and-paste relations: if $X=U\cup Y $ with $Y$ a closed subscheme of $X$ and $U$ its complement, then for any family $\XXX = ( X_i )_{i\in I}$ of elements of $\KVar X$ or $\MMM_X$
\[
\prod_{x \in X/S} \left ( 1 + \sum_{i\in I} X_{i,x}  t_i \right )  = \left ( \prod_{u \in U/S} \left ( 1 + \sum_{i\in I} X_{i,u}  t_i \right ) \right )\left ( \prod_{y\in Y/S} \left ( 1 + \sum_{i\in I} X_{i,y}  t_i \right ) \right )
\]
when considering the motivic Euler products of the restrictions 
\[
\mathscr Y =  (X_i \times_X Y )_{i\in I}
\] 
and 
\[
\UUU =  {( X_i \times_X U )_{i\in I}}.
\] 
 \end{myptn}
 
We will need the following generalisation of \cite[Proposition 3.9.2.4]{bilu2018motivic}. 
\begin{myptn} \label{proposition:euler-product:multiplicativity}
 \index{motivic Euler product!of a formal product}
We assume that $I$ is of the form $I_0 \setminus \{ 0 \}$ where $I_0$ is a commutative monoid, and that $(t_i)_{i\in I}$ is a collection of indeterminates, such that\footnote{We refer the reader to the notion of \textit{algèbre large d'un monoïde} in \cite[Chap. III, p. 27]{bourbaki2007algebre}.}
\begin{enumerate}
	\item $I_0$ is endowed with a total order $<$ such that $p+q=i$, where $q\neq 0$, implies $p<i$;
	\item for all $i \in I $, the set $\{ p \in I \mid p < i \}$ is finite;
	\item the collection of indeterminates $(t_i)_{i\in I}$ satisfies $t_p t_q = t_{p+q}$. 
\end{enumerate}

Let $S$ be a variety, 
$X$ a variety over $S$,
${\AAA = ( A_i )_{i\in I}}$ 
and ${\BBB = ( B_i )_{i\in I}}$
\textit{any} pair of families of elements of $\KVar X$ or $\MMM_X$.
Then, 
under the above hypotheses on $I$
and $(t_i)_{i\in I}$,
\begin{align*}
& \prod_{x \in X/S} \left ( 
\left ( 
1 + \sum_{i\in I} A_{i,x} t_i
\right ) 
\left ( 
1 + \sum_{i\in I} B_{i,x} t_i
\right ) \right ) \\
& =
\prod_{x \in X/S}
\left ( 
1 + \sum_{i\in I} A_{i,x} t_i
\right ) 
\prod_{x \in X/S}
\left ( 
1 + \sum_{i\in I} B_{i,x} t_i
\right ) .
\end{align*}
\end{myptn} 

\begin{myremark}
	In \cite[Proposition 3.9.2.4]{bilu2018motivic}, the family $\AAA$ is assumed to be made of \emph{effective} elements.
	In order to check that the motivic Tamagawa number of a product of two Fano-like varieties is the product of the two motivic Tamagawa numbers, or, more generally, to compute the motivic Tamagawa number of a fibration, we need to drop the effectiveness assumption.
\end{myremark}

\begin{proof}[Proof of \cref{proposition:euler-product:multiplicativity}]
	I thank Margaret Bilu for having pointed out that it is actually a direct application
	of \cite[Corollary 3.9.2.5]{bilu2018motivic}. One proceeds exactly as in \cite[p. 89-90]{bilu2018motivic}, where \cref{proposition:euler-product:multiplicativity} is proved for families of varieties. 
	
	Indeed, \cite[Corollary 3.9.2.5]{bilu2018motivic} tells us that if $X'$ is a variety over $X$, and $\XXX '= ( X_i ' )_{i\in I}$ is a family of classes in $\KVar {X'}$, then 
\[
\prod_{x\in X / S } \left ( \prod_{x ' \in X ' /X} \left (  1 + \sum_{i\in I} X_{i,x'} ' t_i \right ) \right )_x
=
\prod_{x ' \in X ' / X } \left ( 1  + \sum_{i\in I} X_{i,u} ' t_i \right )
\]
(see \cite[\S 3.9.1]{bilu2018motivic} for the definition of this double-product notation).
Taking $X'$ to be the disjoint union of two copies of $X$,
and $\AAA$, $\BBB$ to be the restrictions of $\XXX$ respectively to the first and to the second copy, 
we get the expected identity. 	
\end{proof}

\begin{myexample}[{\cite[Proposition 3.7]{vakil2015discriminants}}]\label{example-kapranov-zeta}
The Kapranov zeta function of $X/S$ is defined as
	\[
	Z_{X/S}^{\text{Kapr}} ( t ) = \sum_{m \in \NN} \left [ \Sym_{/ S}^m X \right ] t^m .
	\]
If the characteristic of the base field is zero,
it can be rewritten 
	\[
	Z_{X/S}  ^{\text{Kapr}}  ( t )  = \prod_{x \in X/S } 
	\left ( 
	( 1 - t )^{-1} 
	\right )
	\]
and then by \cref{proposition:euler-product:multiplicativity}
\[
Z_{X/S}  ^{\text{Kapr}}  ( t ) ^{-1} = \prod_{x \in X / S  } 
	\left ( 
	 1 - t 
	\right ) .
\]
In positive characteristic, this equality only holds in the modified Grothendieck ring $\KVar S^{\text{uh}}$, see \cite[Example 6.1.11]{bilu2021motivic}.
\end{myexample}

\begin{mynotation}[Effective motivic Euler product]\label{def:euler-product-over-curves} 
\index{motivic Euler product!effective}
Let $Y$ be an element of $\KVar X $ or $\MMM_X$.
The motivic Euler product 
\[
\prod_{x \in X / S } \left  ( 1 + Y_x \right )
\]
is by definition the series
\[
\sum_{m\geqslant 0 } 
\left [ 
\Sym^m_{ X / S } ( Y )_* 
\right ] .
\]
In case this series converges in a convenient completion of $\KVar X $ or $\MMM_X$, 
its sum is written $\prod_{x \in X / S } \left  ( 1 + Y_x \right )$ as well. 
\end{mynotation}
 
\begin{myremark}
Since abstract motivic Euler products are compatible with changes of variable of the form $t' = \LL^a t$
\cite[\S 3.6.4]{bilu2018motivic},  
this notation can be seen as the specialization of 
\[
\prod_{x \in X / S} \left  ( 1 + Y_x \LL_x^a t \right )  = \sum_{m \geqslant 0}  [ \Sym^m_{X / S} ( Y \times_X \mathbf A^a_X )_* ] t^m  = \sum_{m \geqslant 0}  [ \Sym^m_{X / S} ( Y )_* ] \LL_S^{ma} t^m 
\]
at $t= \LL_S^{-a} $ for any non-negative integer $a$. 
 \end{myremark}

\subsubsection{Specialising products}
The previous multiplicative property specialises. 

\begin{myptn}\label{ptn:multiplicative-property-for-effective-Euler-products}
 \index{motivic Euler product!effective}
Assume that $	\prod_{x \in X/S} ( 1  + A_x ) 
$ and $	\prod_{x \in X/S} ( 1 + B_x )$ converge. 
Then 
\[
\prod_{x \in X/S} ( 1 + A_x ) ( 1 + B_x )
\]
converges and
	\[
	\prod_{x \in X/S} ( 1  + A_x ) 
	\prod_{x \in X/S} ( 1 + B_x )
	=
	\prod_{x \in X/S} ( 1 + A_x ) ( 1 + B_x ) . 
	\]
\end{myptn}

\begin{proof}Let us consider the two formal motivic Euler products 
\[
	P ( t ) = \prod_{x \in X/S} 
	\left (
	1 + ( A_x + B_x + A_x B_x ) t^2 
	\right ) 
\]
and
\[
P_{1,2} (t_1 , t_2 ) = \prod_{x \in X/S} ( 1 + t_1 A_x ) ( 1 + t_2 B_x ) 
= \prod_{x \in X/S} 
\left ( 
1 + t_1  A_x + t_2 B_x + t_1 t_2 A_x B_x 
\right ).
\]
We have to check that 
\[
P ( 1 ) = P_{1,2} ( 1 , 1 ). 
\]
We introduce the following intermediate Euler product:
\begin{align*}
P_{0,1,2} ( t_0 , t_1 , t_2 ) 
& = \prod_{x \in X/S} 
\left ( 
1 + t_0 t_1 A_x + t_0 t_2 B_x + t_1 t_2 A_x B_x 
\right )   \\
& = \prod_{x \in X/S} 
\left ( 	1 + t_0 (t_1 A_x + t_2 B_x ) + t_1 t_2 A_x B_x  
\right ) . 
\end{align*}
Then by \cref{proposition:euler-product:multiplicativity},
\[
P_{1,2} ( t_1 , t_2 ) =
\prod_{x \in X/S} ( 1  + t_1 A_x ) 
\prod_{x \in X/S} ( 1 + t_2 B_x )
\]
and by \cite[Proposition 6.4.5]{bilu2021motivic} 
\[
P_{0,1,2} \left ( 1 , t_1 , t_2 \right ) 
=
\prod_{x \in X/S} 
\left ( 
1 + t_1 A_x + t_2 B_x + t_1 t_2 A_x B_x 
\right )
= 
P_{1,2} ( t_1 , t_2 ).
\]
Moreover, by \cite[Lemma 6.5.1]{bilu2021motivic}
\[
P_{0,1,2} ( t , t , t ) = P ( t ) 
\]
Taking $t = 1$ everywhere one gets
\[
P \left ( 1 \right ) 
= 
P_{0,1,2} \left ( 1 , 1 , 1 \right ) 
= 
P_{1,2} \left ( 1 , 1 \right ) 
\]
as expected.
\end{proof}
 
\subsubsection{Convergence criterion with respect to the weight filtration}
In case $\kb$ is a subfield of $\CC$, we have a convergence criterion for motivic Euler products of families over a curve $\mathscr C$. It is a particular case of \cite[Proposition 2.6]{faisant2022geometric} which itself is a multi-variable variant of \cite[Proposition 4.7.2.1]{bilu2018motivic}. 

\begin{myptn}\label{proposition:linear-conv-criterion}
Fix an integer $r\geqslant 1$ and an $r$-tuble $\rho \in  \left ( \mathbf N^* \right ) ^r$. 
For any tuple $\mm$ of non-negative integers, we write $ \langle \rho , \mm \rangle = \sum_{i=1}^r m_i \rho_i$.

Assume that $\XXX = ( X_{\mm} )_{\mm \in ( \NN^r )^*}$ is a family of elements of $\MMM_\CCC$ such that 
there exist an integer $M\geqslant 0$ 
and real numbers $\alpha < 1 $ and $\beta $ such that 
\begin{itemize}
\item $w_{\mathscr C} (  X_\mm ) \leqslant 2   \langle \rho , \mm \rangle - 2 $ whenever $1 \leqslant \langle \rho , \mm \rangle \leqslant  M$;
\item $w_\mathscr C ( X_{\mm}) \leqslant 2  \alpha \langle \rho , \mm \rangle + \beta - 1$	whenever $\langle \rho , \mm \rangle > M$.
\end{itemize}

Then there exists $\delta \in \left ]0,1\right ]$  and $\delta ' > 0 $ such that 
\[
w_\CC \left ( 
 \Sym_{\CCC / \kb}^\pi ( \XXX )_*
 \cdot 
 \mathfrak a_1^{\rho_1 m_1}  \cdots \mathfrak a_r^{\rho_r m_r}  
 \right ) \leqslant - \delta ' \langle \rho , \mm \rangle 
\] 
for all
$\mathbf{m}\in ( \NN^r )^* $,
all 
partitions $\pi $ of $\mm$ 
and
all 
elements
$\mathfrak a_1 , ... , \mathfrak a_r \in \MMM_\CC$ such that $w( \mathfrak a_i ) <  - 2 + \delta - \frac{\beta }{M+1} $ for all $1\leqslant i \leqslant r$.  
\end{myptn}
If we specialize the previous proposition to the polynomial $F ( T) = 1 + Y \LL T$ and consider the convergence at $T=\LL^{-1}$ we get the following criterion.

\begin{myptn}\label{proposition:convergence-criterion-motivic-Euler-product-over-curve}
Assume that $Y \in \MMM_\CCC$ is such that $w_\CCC ( Y ) \leqslant -2 $. 
Then the series $\sum_{m\geqslant 0 } \left [ \Sym^m_{\CCC} ( Y )_* \right ] $ converges in $\widehat{\mathscr M_\kb}^w$ and the Euler product 
\[
\prod_{p\in \CCC} \left  ( 1 + Y_p \right ) 
\]
is well-defined. 
\end{myptn}

\begin{myexample}
Recall that  
$w_\CCC\left  ( \LL_\CCC^{-2} \right  )  = -  2 \dim_\CCC (  \mathbf A^2_\CCC ) + \dim ( \CCC ) = - 3 $ thus 
\[
\prod_{p\in \CCC } \left ( 1 - \LL^{-2}_{p } \right ) 
\]
converges in $\widehat{\MMM_\kb}^w$. 
Moreover, one can show that this convergence actually holds in $\widehat{\MMM_\kb}^{\dim}$ for any $\kb$. 

Let $p : Y \to \CCC$ be a smooth $\CCC$-variety with irreducible fibres of dimension $d\geqslant 0$. 
Then,
\[
\prod_{p\in \CCC } \left ( 1 + \left ( [ Y_p ] - \LL^{d}_p \right )\LL^{-(d+1)}_{p} \right ) 
\]
converges in $\widehat{\MMM_\kb}^w$, since 
\[ 
w_\CCC  \left ( [ Y ] - \LL^d_\CCC \right )  \leqslant 2d 
\]
by \cref{ptn:weight-of-the-difference-smooth-irreducible-fibres}, thus $w_\CCC \left ( \left ( [ Y ] - \LL^d_\CCC \right )\LL^{-(d+1)}_\CCC  \right )\leqslant -2$. 

\end{myexample}

The following little lemma will help us to explicitly control error terms when studying the case of toric varieties. One can replace the dimension by the weight or any other filtration compatible with finite sums. 
\begin{mylemma}\label{lemma:controlling-error-terms-convolution}
Let $(\mathfrak c_\mm)_{\mm \in \NN^r}$ be a family of elements of $\mathscr M_\kb$.
	Assume that there exists a constant $a > 0 $ such that 
	\[
	\dim ( \mathfrak c_\mm ) \leqslant - a | \mm | 
	\]
	for every $\mm \in \NN^r $, 
	where $| \mm | = \sum_{i=1}^r m_i$. 
	
	Then, for all non-empty subset $A \subset \{ 1 , ... , r \} $, and non-negative integer $b$,
	\[
	\dim \left ( \sum_{\mm ' \leqslant \mm - \mathbf b  } \mathfrak c_{\mm ' } \LL^{\mm ' _A - \mm _A } \right  ) 
	\leqslant - \frac{\min ( 1 , a )}{2}  \underset{\alpha \in A}{\min}( m _\alpha ) - \frac{b}{2} \min \left ( 1 , 2| A |  - a  \right )
	\]
	for all $\mm \in \NN^r_{\geqslant b}$,
	where $\mm_A $ is the restriction of $\mm$ to $A$ and $\mathbf b = ( b , ... , b ) $. 
\end{mylemma}

\begin{proof}
	If $\mm_A ' \nleqslant \frac{1}{2} (  \mm_A - \mathbf b ) $ then the coarse upper bound
	\[
	\dim ( \mathfrak c_{\mm ' } \LL^{\mm ' _A - \mm _A } ) \leqslant - a | \mm '  |  - b | A |
	\]
	for $\mm ' \leqslant \mm - b$
	gives 
	\[
	\dim ( \mathfrak c_{\mm ' } \LL^{\mm ' _A - \mm _A } ) < - \frac{a}{2}  \underset{\alpha \in A}{\min}(  m _\alpha - b )  - b | A | = - \frac{a}{2}  \underset{\alpha \in A}{\min}(  m _\alpha ) - \frac{b}{2} ( 2| A | - a  ) 
	\]
	while if $\mm_A ' \leqslant \frac{1}{2} ( \mm_A - \mathbf b ) $ then $\mm_A ' - \mm_A \leqslant - \frac{1}{2} ( \mm_A  + \mathbf b_A ) $ and 
	\[
	\dim ( \mathfrak c_{\mm ' } \LL^{\mm ' _A - \mm _A } )  \leqslant - a | \mm ' | - \frac{1}{2} | \mm_A | - \frac{b}{2} \leqslant - \frac{1}{2} \underset{\alpha \in A}{\min}( m _\alpha )- \frac{b}{2} . 
	\]
\end{proof}

\bigskip

 
\section{Batyrev-Manin-Peyre principle for curves} 
\label{section:BMP-general}

 \subsection{A convergence lemma in characteristic zero} 
 In this paragraph we assume that $\kb$ is a subfield of $\CC$ and choose once and for all an embedding $\kb \hookrightarrow \CC$, as in \cref{section:weight-filtration}.
 
\begin{mylemma}\label{lemma:convergence-formal-motivic-Tamagawa-number} Let $\VVV \to \CCC$ be a proper model of a Fano-like variety $V$. 
Then for any dense open subset $\CCC ' \subset \CCC$,
the motivic Euler product
\[
\prod_{p\in \CCC ' } \left ( \frac{\left [\VVV_p  \right ]}{\LL_p^{\dim  ( V ) }} \left ( 1 - \LL_p^{-1} \right )^{\rk ( \Pic ( V ) )} \right ) 
\]
converges in the completion of $\mathscr M_\kb$ with respect to the weight filtration. 
\end{mylemma}
\begin{proof}
By the multiplicative property of motivic Euler products \cref{ptn:motivic-Euler-product-compatible-cut-past},
it is enough to find some dense open subset $\CCC ' \subset \CCC$ such that the product above converges.  
Let $r$ be the Picard rank of $V$.
In $\mathscr M_\CCC$, we have \begin{align*}
	& [\VVV ]\LL _\CCC ^{-\dim_\CCC \VVV } ( 1 - \LL_\CCC^{-1})^r \\
	& = \left ( 1 + \left ( [\VVV ] - \LL_\CCC^{\dim_\CCC ( \VVV ) } \right )\LL^{-\dim_\CCC (\VVV )}_\CCC  \right )  \left ( 1 + \sum_{k=1}^r  {  r \choose k } (-1)^k \LL^{-k} _\CCC \right ) \\
& = 1 + \frac{[\VVV ] - \LL^{\dim_\CCC (\VVV) } _\CCC - r \LL^{\dim_\CCC ( \VVV ) - 1 } _\CCC}{\LL^{\dim_\CCC ( \VVV ) } _\CCC} + \mathscr R
\end{align*}
where 
\[
\mathscr R = -r \frac{[\VVV] - \LL _\CCC ^{\dim_\CCC ( \VVV ) }}{ \LL^{\dim_\CCC ( \VVV ) +1 } _\CCC}   + \frac{[\VVV]}{\LL^{\dim_\CCC (\VVV) } _\CCC}\sum_{k=2}^r {r \choose k} ( -1 )^k \LL^{-k} _\CCC . 
\]
By \cref{def:euler-product-over-curves} of the motivic Euler product, we are interested in the series 
\[
\sum_{m\geqslant 0 } [ S^m_\ast ( \YYY ) ] 
\]
where 
\[
\YYY = \frac{[\VVV ] - \LL _\CCC ^{\dim_\CCC (\VVV) } - r \LL^{\dim_\CCC ( \VVV ) - 1 } _\CCC}{\LL^{\dim_\CCC ( \VVV ) } _\CCC} + \RRR \in \mathscr M_\CCC 
\]  
and in order to prove convergence, 
it is enough to check that $w_\CCC ( \YYY ) \leqslant -2 $. 
Since the motivic Euler product is compatible with finite products, 
up to replacing $\CCC$ by a non-empty open subset 
we can assume that $\VVV \to \CCC$ is smooth with irreducible fibres. 
Then by \cref{ptn:weight-of-the-difference-smooth-irreducible-fibres} 
we have a first bound on the weight 
\[ 
w_\CCC  \left ( \left ( [\VVV ] - \LL^{\dim_\CCC ( \VVV ) } _\CCC \right ) \LL^{-\dim_\CCC ( \VVV) - 1  } _\CCC \right ) \leqslant -2 
\]
and the expression of $\RRR$ shows that 
\[
w_\CCC ( \RRR ) \leqslant -2
\] 
as well. 
In order to show that $w_\CCC ( \YYY - \RRR ) \leqslant -2 $ 
 we use the fact that $V$ is Fano-like. 
By \cref{remark:local-triviality-of-derived-fuctor}, we are in the situation of \cref{proposition:hodge-realisation-almost-fano-varieties}: up to restricting to an open subset of $\CCC$, we have a decomposition  
\[
\chi_\CCC^\Hdg ( [\VVV ] ) = (\chi_\CCC^\Hdg ( \LL_\CCC ))^{\dim_\CCC ( \VVV ) } + r (\chi_\CCC^\Hdg ( \LL_\CCC ) )^{\dim_\CCC ( \VVV )-1} + \WWW \in K_0 ( \MHM_\CCC ) 
\]
with $w_\CCC ( \WWW ) \leqslant 2 (\dim_\CCC ( \VVV ) - 1 )$. 
Hence $w_\CCC ( \YYY - \RRR ) \leqslant -2$ and by the property of the weight one has ${w_\CCC ( \YYY ) \leqslant -2} $. 
By the convergence criterion of \cref{proposition:convergence-criterion-motivic-Euler-product-over-curve}, 
this shows that the motivic Euler product we are considering converges in $\widehat{\mathscr M_\kb}^w$.  
\end{proof}

\subsection{Motivic Tamagawa number of a model}
Let $V$ be a Fano-like $F$-variety of dimension $n$
and $\VVV$ a proper model of $V$.   
We are now able to give a precise meaning to the motivic analogue of the Tamagawa number. 

Up to replacing $\VVV$ by a dominating model,
we can always assume that it is a \textit{good model},
that is to say a proper model of $V$ 
whose smooth locus $\VVV^\circ $ is a weak Néron model of $V$,
by \cref{thm:existence-of-weak-neron-models} and \cref{ptn:gluing-dominating-models}.

Furthermore, we assume that $\CCC$ admits a divisor of degree one, which is the case for example if $k = \mathbf F_q$ by the Lang-Weil estimate, or more trivially if $\kb = \mathbf \CC$ or if $\CCC$ is the projective line.

\begin{mydef}
\label{def:motivic-Tamagawa-number-of-a-model}
\index{motivic Tamagawa number!of a model}
The motivic constant $\tau_{\LLL} ( \VVV  ) $ of the proper model ${\VVV \to \CCC}$ 
with respect to a model $ \LLL $ of $\omega_V^{-1}$
is 
the effective element of $\widehat{\mathscr M_\kb}^{\dim}$ or $\widehat{\mathscr M_\kb}^w$
given by the motivic Euler product
\[
\LL_\kb^{( 1 - g ) \dim ( V )}  
(\res_{ t = \LL_\kb^{-1} }  Z^\text{Kapr}_\CCC ( t ) )^{\rg ( \Pic ( V ) )}
\prod_{p\in \CCC} 
\left  (1-\LL_p^{-1} \right )^{\rg (\Pic (V))} \mu_{\LLL_{ | \VVV_\Rp } } ^ * ( \Gr ( \VVV_{\Rp}^\circ )  ) 
\]
where the motivic density \[
\mu_{\LLL_{ | \VVV_\Rp } } ^ *
\] is given by \cref{def:motivic-measure-with-respect-model-line-bundle}
and 
\[\res_{t=\LL_\kb^{-1}} ( Z^\text{Kapr}_\CCC ( t ) )  = \left (  ( 1 - \LL_\kb t )  Z^\text{Kapr}_\CCC ( t )  \right )_{t= \LL^{-1}_\kb }
\]
with  $Z^\text{Kapr}_\CCC ( t )  $ being the Kapranov zeta function of $\CCC$ as defined in \cref{example-kapranov-zeta}.

In case $\VVV  $ is a constant model of the form $\VVV= V \times_\kb \CCC$, 
with $V$ a nice variety defined over $\kb$,
$\LLL = \pr_1^* ( \omega_V^{-1} )$ 
and $\pi = \pr_2$, 
then 
the constant $\tau_\LLL ( \VVV )  $ 
will be written $\tau ( V / \CCC ) $. If moreover $\CCC$ is clear from the context, it will be simply written $\tau ( V  ) $.  
\end{mydef}

\begin{myremark}
	The use of the motivic Euler product notation in this definition is licit. Indeed, 
	there exists a dense open subset $\CCC ' $ of $\CCC$ on which $\pi$ is smooth, 
	$\left ( \Omega^n_{\VVV / \CCC } \right )^\vee $ is invertible and isomorphic to $\LLL$. 
	For points $p$ in $\CCC ' $,  
	\[
	\mu_\LLL^* ( \Gr ( \VVV_{\Rp}^\circ )  ) =  \vol_{\VVV_\Rp} ( \Gr ( \VVV_{\Rp})) = [ \VVV_\kp ]\LL^{-\dim ( V )} .
	\]
	Then the motivic Euler product 
	\[
	\prod_{p\in \CCC ' } 
 \left  (1-\LL_p^{-1} \right )^{\rg (\Pic (V))} \vol_{\VVV_\Rp} ( \Gr ( \VVV_{\Rp} ) )
	\]
	is obtained by applying \cref{def:euler-product-over-curves} to the family 
	\[
	\left (1 - \LL^{-1}_{\CCC '} \right )^{\rk ( \Pic ( V ))} [ \VVV_{\CCC ' } ]\LL_{\CCC '}^{-\dim (V ) } - 1 \in \MMM_{\CCC '} .
	\]
\end{myremark}

\begin{myremark}
	By \cref{lemma:convergence-formal-motivic-Tamagawa-number}, when $\kb$ is a subfield of $\CC$, this class is well-defined in $\widehat{\MMM_\kb}^w$. 
	When $V$ is a smooth split projective toric variety, the convergence holds in $\widehat{\MMM_\kb}^{\dim}$, without any assumption on $\kb$ (see \cref{thm:motivic-BMP-toric}). 
\end{myremark} 

\begin{myremark}
	Since we assume that $\CCC$ admits a divisor of degree one,
	we know
	that the series
	\[
	(1 - t ) ( 1 - \LL_\kb t ) Z^\text{Kapr}_\CCC ( t )
	\]
	is actually 
	a polynomial $P_\CCC ( t ) $ of degree $2 g $ 
	such that 
	\[
	P_\CCC ( \LL_\kb^{-1} ) = [ \Pic^0 ( \CCC ) ] \LL_\kb^{-g}
	,
	\]
	see e.g. \S 1.3 in \cite[Chap. 6]{chambert2016motivic}.
	Hence in that case 
	\[
	\res_{ t = \LL_\kb^{-1} }  Z^\text{Kapr}_\CCC ( t )
	=
	\left ( 
	\frac{P_\CCC ( t ) }{1 - t} 
	\right )_{t= \LL^{-1}_\kb }
	=
	\frac{[ \Pic^0 ( \CCC ) ] \LL_\kb^{-g}}{1-\LL^{-1} } . 
	\]
	In general, 
	Daniel Litt showed in 
	\cite{litt2015zeta}
	that if 
	$d$ is the minimum degree of a divisor on $\CCC$,
	then $(1 - t^d ) ( 1 - \LL_\kb^d t^d ) Z^\text{Kapr}_\CCC ( t )$ 
	is a polynomial. 
	Note that if $d>1$, then in general 
	$(1 - t ) ( 1 - \LL_\kb t ) Z^\text{Kapr}_\CCC ( t )$ 
	is \emph{not} a polynomial,
	see e.g. Remark 1.3.4 of \cite[Chap.~7]{chambert2016motivic}.
	Thus one should modify the previous definitions according to that fact. 
	This more general situation is beyond the scope of this paper. 
\end{myremark}

It is possible to define variants of the motivic constant $\tau_{\LLL} ( \VVV  ) $, related to components of the moduli space.
\begin{mydef}
Let $\beta $ be a choice of vertical components $E_\beta$ of multiplicity one, that is to say, over a finite number of closed points $p$, the choice of an irreducible component of $\VVV_p $ of multiplicity one. 

If $E_{\beta_p}^\circ = E_{\beta_p} \cap \VVV^\circ $ for all $ p $, then
\[
\tau_\LLL ( \VVV  ) ^\beta
\]
is the motivic Tamagawa number 
\[
\LL_\kb^{( 1 - g ) \dim ( V )  } 
\left (
\frac{[ \Pic^0 ( \CCC ) ] \LL_\kb^{-g}}{1-\LL^{-1} } 
\right )
^{\rg ( \Pic ( V ) ) }
 \prod_{p\in \CCC} 
 \left  (1-\LL_p^{-1} \right )^{\rg (\Pic (V))} \mu_{\LLL_{ | \VVV_\Rp } } ^ * (  \Gr ( E_{\beta_p}^\circ ) ) . 
\]
\end{mydef}
Note that $\tau_\LLL ( \VVV ) $ equals the finite sum of the $\tau_\LLL ( \VVV  ) ^\beta $'s, for $\beta $ running over the finite set of choices of vertical components $E_\beta$ above each closed point of $\CCC$. 

\begin{mydef}
If $\CCC ' $ is a non-empty open subset of $\CCC$, we will call restriction of $\tau_\LLL ( \VVV )^\beta $ to $\CCC ' $, written $\tau_\LLL ( \VVV )^\beta_{|\CCC ' } $, the element 
\[
\LL_\kb^{( 1 - g ) \dim ( V )}  ( \res_{t=\LL_\kb^{-1}} Z_{\CCC '} ( t ) )^{\rg ( \Pic ( V ))} 
 \prod_{p\in \CCC ' } 
 \left  (1-\LL_p^{-1} \right )^{\rg (\Pic (V))} \mu_{\LLL_{ | \VVV_\Rp } } ^ *  (  \Gr ( E_{\beta_p}^\circ ) )   .
\]
\end{mydef}

\subsection{Motivic Batyrev-Manin-Peyre principle for curves} In order to deal with different good models of a projective variety $V$ over $F$, we need a refined version of {\cref{BMP-motivic-weak}}. The previous tools make the adaptation straightforward. 

We use the notations introduced in \cref{setting:main-setting}: we fix a finite set $L_1 , ... , L_r $ of invertible sheaves on $V$ whose linear classes form a $\ZZ$-basis of $\Pic ( V) $, as well as invertible sheaves $\LLL_1 , ... , \LLL_r$  on $\VVV$ extending respectively $L_1 , ... , L_r$.

There exists a unique $r$-tuple of integers $( \lambda_1 , ... , \lambda _r ) $ such that
\[
\omega_V^{-1} 
\simeq 
\otimes_{i=1}^r L_i^{\otimes \lambda_i} .
\]
Consequently, we have a model of $\omega_V^{-1}$ given by
\begin{equation}\label{equation:def-of-LLL-V-model-of-omega-V-dual}
	\LLL_\VVV = \otimes_{i=1}^r \LLL_i^{\otimes \lambda_i} .
\end{equation}

For any choice $\beta$ of irreducible vertical components of multiplicity one of $\VVV$,
and for every non-empty open subset $U$ of $V$,
the space 
\[
\Hom_\CCC^{ \degg_{\underline{\LLL}} = \mdeg  } ( \VVV , \CCC )_U ^\beta 
\]
parametrizing curves of multidegree $\degg_{\underline{\LLL}} = \mdeg$ 
intersecting the components given by $\beta$,
exists as a quasi-projective scheme, and the space 
\[
\Hom_\CCC^{ \degg_{\underline{\LLL}} = \mdeg  } ( \VVV , \CCC )_U \tag{\ref{lemma:representability-Hom-functor}}
\]
is then the finite disjoint union over $\beta $ of these subspaces.

Recall that $\CEff ( V )_\ZZ^\vee$ is the intersection of $\CEff ( V )^\vee$
and $\Pic ( V )^\vee$ in $\Pic ( V )^\vee \otimes_\ZZ \QQ $.
\begin{myquest}[Relative Geometric Batyrev-Manin-Peyre]
\index{motivic Batyrev-Manin-Peyre principle}
\label{conj:BMP-motivic-strong-weight-top} Let $\VVV$ and $\underline{\LLL}$ be as in \cref{setting:main-setting} \cpageref{setting:main-setting}.
 Does the symbol 
\[
\left [ \Hom_\CCC^{ \degg_{\underline{\LLL}} = \mdeg  } ( \VVV , \CCC )_U  \right ] \LL_\kb ^{-   \, \mdeg  \scdot  \omega_V^{-1}  }  \in \MMM_\kb
\]
converge to $\tau_{\LLL_\VVV} (  \mathscr V  ) $ in $\widehat{\MMM_\kb}^w$ (or even more optimistic, in $\widehat{\MMM_\kb}^{\dim} $),
as $\mdeg \in  \CEff (V) ^\vee_\ZZ $ 
goes arbitrarily far away from the boundaries of the dual cone $ \CEff (V) ^\vee $ ? 
\end{myquest}

\begin{myremark}\label{remark:BMP-motivic-strong-weight-top-with-components}
We can refine the previous question as follows: given a choice $\beta$ of vertical components of multiplicity one: 
does the symbol 
\[
\left [ \Hom_\CCC^{ \degg_{\underline{\LLL}} = \mdeg  } ( \VVV , \CCC )_U ^\beta 
  \right ] \LL_\kb ^{-   \, \mdeg  \scdot  \omega_V^{-1}  }  \in \MMM_\kb
\]
converge to $\tau (  \mathscr V    ) ^\beta$ in  $\widehat{\MMM_\kb}^w$, or even $\widehat{\MMM_\kb}^{\dim} $,
as $\mdeg \in  \CEff (V) ^\vee_\ZZ $ 
goes arbitrarily far away from the boundaries of the dual cone $ \CEff (V) ^\vee $ ? 
\end{myremark}

\begin{myexample}\label{example:compact-ev}
Starting from previous works by Chambert-Loir-Loeser \cite{chambert2016motivic} and Bilu \cite{bilu2018motivic},
we show in \cite{faisant2022geometric} that the conjecture of \cref{remark:BMP-motivic-strong-weight-top-with-components} is true when $V$ is an equivariant compactification of a vector space 
and $\kb$ algebraically closed with characteristic zero. 

\end{myexample}

\begin{myexample}
Bilu and Browning show in \cite{bilu2023circle}	
that the answer to \cref{conj:BMP-motivic-strong-weight-top} 
is positive 
whenever 
$\CCC = \PP^1_\CC$,
$V\subset \PP^{n-1}_\CC$ is a hypersurface of degree $d\geqslant 3$
such that $n> 2^d ( d - 1 ) $, and $\VVV = \PP^1_\CC \times_\CC V$.
\end{myexample}


\subsection{Products of Fano-like varieties}

\index{motivic Tamagawa number!of a product}
\begin{myptn}
	Let $V_1$ and $V_2$ be two Fano-like varities over $F$. 
Let $\VVV_1$ (respectively $\VVV_2$)
be a model of $V_1$ above $\CCC$
and $\LLL_1$
be a model of $\omega_{V_1}^{-1}$ (resp. $V_2$, $\LLL_2$ and $\omega_{V_2}^{-1}$). 

Then $\VVV_1 \times_\CCC \VVV_2$ 
is a model of $V_1 \times_F V_2 $ 
above $\CCC$,
$( \pr_1^* \LLL_1 ) \otimes ( \pr_2^* \LLL_2 )$
is a model of $( \pr_1^* \omega_{V_1}^{-1} ) \otimes ( \pr_2^* \omega_{V_2}^{-1} ) $ 
and
\[
\tau_{ (\pr_1^* \LLL_1 ) \otimes  ( \pr_2^* \LLL_2 ) } ( \VVV_1 \times_\CCC \VVV_2 ) = \tau_{\LLL_1 } ( \VVV_1 ) \tau_{\LLL_2 } ( \VVV_2 ) .
\] 
\end{myptn}
 
 \begin{proof}
 	In order to apply \cref{ptn:multiplicative-property-for-effective-Euler-products},
 	we have to check that for all closed points $p\in \CCC$,
 	\[
 \mu_{( \pr_1^* \LLL_1 ) \otimes ( \pr_2^* \LLL_2 )_{ | ( \VVV_ 1 \times_\CCC \VVV_2 )_\Rp } } ^ * ( \Gr ( ( \VVV_ 1 \times_\CCC \VVV_2 )_\Rp^\circ )  )  = \mu_{{\LLL_1}_{ | ( \VVV_1 )_\Rp } } ^ * ( \Gr ( \VVV_{\Rp}^\circ )  ) \mu_{{\LLL_2}_{ | (\VVV_2 )_\Rp } } ^ * ( \Gr ( \VVV_{\Rp}^\circ )  ). 
 	\]
 	Going back to \cref{def:motivic-measure-with-respect-model-line-bundle}, 
 	and by functoriality of Greenberg schemes,
 	it is enough to check that on $\Gr ( ( \VVV_ 1 \times_\CCC \VVV_2 )_\Rp^\circ )  $ one has  
 	\begin{equation}
 		\label{eq:product-varieties-difference-of-degree}
\varepsilon_{( \pr_1^* \LLL_1 ) \otimes ( \pr_2^* \LLL_2 ) - ( \Lambda^{n_1+n_2 } \Omega^1_{\VVV_1\times \VVV_2 / \Rp})^\vee}
= 
\varepsilon_{\LLL_1 - ( \Lambda^{n_1 } \Omega^1_{\VVV_1 / \Rp })^\vee}
+
\varepsilon_{\LLL_2 - ( \Lambda^{n_2 } \Omega^1_{\VVV_2 / \Rp })^\vee} .
 	\end{equation}
We can assume that $\VVV_1$ and $\VVV_2$ are $\Rp$-smooth. 	
Let $R_p'$ be an unramified extension of $R_p$,
 $\tilde x = (  \tilde x_1 , \tilde x_2 )  \in ( \VVV_1 \times_{R_p} \VVV_2 ) ( R_p ') \simeq  \VVV_1 ( R_p ') \times_{\kappa ( p )'} \VVV_2  ( R_p ')	$
and take $y_1$, $y_2$, $\omega_1$ and $\omega_2$ to be
generators respectively of 
$\tilde x_1 ^* \LLL_1$, $\tilde x_2^* \LLL_2$, $\tilde x_1^* (\Lambda^{n_1 } \Omega^1_{\VVV_1 / \Rp })^\vee$
and $\tilde x_2^* (\Lambda^{n_2 } \Omega^1_{\VVV_2 / \Rp })^\vee$.
Then 
$y_1  y_2$ is a generator of $\tilde x ^* ( \pr_1^* \LLL_1 ) \otimes ( \pr_2^* \LLL_2 )$ 
and $\omega_1 \omega_2 $ a generator of $\tilde x^* ( \Lambda^{n_1+n_2 } \Omega^1_{\VVV_1\times \VVV_2 / \Rp})^\vee$. 
Now \eqref{eq:product-varieties-difference-of-degree} applied to $\tilde x$
is the identity $v_{R_p'} \left ( { \omega_1 \omega_2  \over y_1 y_2 } \right ) = v_{R_p'} \left ( {\omega_1 \over y_1} \right ) + v_{R_p'} \left ( { \omega_2 \over y_2 } \right )$. 
 \end{proof}

\bigskip 


\section{Equidistribution of curves}
\label{section:equidistribution}

The goals of this section are
\begin{itemize}
	\item 
to introduce and provide a definition of the equidistribution principle,
which will be done with \cref{def:equidistribution-products-constructible-subsets},
\item and then
to prove \cref{thm-intro:change-of-model},
which will be restated as \cref{thm:equidistribution-and-models},
telling that this principle does not depend on the choice of models me made in \cref{setting:main-setting}.
\end{itemize} 

\smallskip 

The intuition behind this result is that the equidistribution hypothesis encodes (among many other things)
the information one needs to switch from a multidegree coming from a given model to another one. 
Indeed, while the Batyrev-Manin-Peyre principle \cref{conj:BMP-motivic-strong-weight-top} describes the expected asymptotic behaviour of the motivic measure of a certain moduli space of sections,
the equidistribution principle describes the asymptotic behaviour of the sequence of motivic measures itself. 
In particular, it measures the motivic distribution of sections for which two different models give two different multidegrees. 

\smallskip 

In this section we assume that 
\begin{itemize}
	\item $\VVV \to \CCC$ is a proper model over $\CCC$ of a Fano-like variety $V$ together with a model $\underline{\LLL} = ( \LLL_i )$ of a family of invertible sheaves $(L_i)$ on $V$ whose classes form a $\ZZ$-basis of $\Pic ( V )$ (see \cref{def:Fano-like} and \cref{setting:main-setting});
	\item $U$ is a dense open subset of $V$;
	\item the motivic Tamagawa number $\tau_{\LLL_\VVV} ( \VVV ) $, see \cref{def:motivic-Tamagawa-number-of-a-model} and \eqref{equation:def-of-LLL-V-model-of-omega-V-dual}, is well-defined in either $\widehat{\MMM_\kb}^{\dim}$ or $\widehat{\MMM_\kb}^w$. The following discussion will not depend on the choice of the filtration. 
\end{itemize}

\subsection{First approach}

Let $\SSS $ be a zero-dimensional subscheme of the smooth projective curve $\CCC$, $| \SSS |$ its set of closed points and $\CCC '$ the complement of $|\SSS|$.
This subscheme $\SSS$ is given by a disjoint union of spectra of the form
\[
\Spec \left  ( \mathscr O_{\CCC , p} / (\mathfrak m_p^{m_p+1} ) \right  )
\simeq  
\Spec ( \kappa (p ) [[t]] / t^{m_p+1} ) 
\]
for $p\in  | \SSS |$.
Its length is
\[
\ell ( \SSS ) = \sum_{p\in | \SSS |} ( m_p + 1 ) [ \kappa ( p ) : \kb ] .
\]
Then for every $\CCC$-morphism $\varphi : \SSS \to \VVV$ and every $\mdeg \in \Pic ( V ) ^\vee $ we define 
\[
\Hom_\CCC^{ \degg_{\underline{\LLL}} \, = \,  \mdeg} ( \CCC , \VVV \mid \varphi )_U
\]
as being the schematic fibre above $\varphi $ of 
the restriction morphism 
\[
\res_\SSS^{\VVV}
: 
\Hom_\CCC^{ \degg_{\underline{\LLL}} \, = \,  \mdeg} ( \CCC , \VVV )_U \longrightarrow \Hom_\CCC ( \SSS , \VVV ). 
\]

We assume temporally that $\LLL_\VVV$ is isomorphic to $( \Lambda^n \Omega^1_{\VVV / \CCC })^\vee$
and that $\VVV \to \CCC$ is smooth above an open subset containing the closed points of $\SSS$.
Then we say that there is weak equidistribution for $\SSS $ 
if for every $\CCC$-morphism $\varphi :  \SSS \to \VVV $, the normalized class 
\[
 \left [ \Hom ^{ \degg_{\underline{\LLL}} \, = \,  \mdeg}_\CCC ( \CCC , \VVV  \mid \varphi )_U \right  ] 
 \LL_\kb ^{-\, \mdeg  \scdot  \omega_V^{-1} }
 \in 
 \mathscr M_\kb 
\]
converges to 
\[
\tau ( \VVV  )_{| \CCC '} 
\times 
\prod_{p\in | \SSS |} \LL_p  ^{- ( m_p + 1 ) \dim  ( V ) } \in \widehat{\mathscr M_\kb} 
\] 
when the multidegree $\mdeg$ tends to infinity
-- again, by this we will always mean $\mdeg \in \CEff ( V )_\ZZ^\vee$
and $d ( \mdeg , \partial \CEff ( V )^\vee ) \to \infty $. 
This definition may be seen as a first extension of Peyre's definition \cite[$5.8$]{peyre2021beyond} to non-constant families $\VVV \to \CCC$. 

\begin{myremark} 
 $\Hom_\SSS ( \SSS , \VVV_\SSS )  $ can be interpreted as the product of spaces of jets 
 \[ 
 \prod_{p\in | \SSS |} \Gr_{ m_p  } ( \VVV_{\Rp} ) .
 \] 
Since 
\[
\Gr_{ m_i  } ( \VVV_{\Rp} ) \to \Gr_0 ( \VVV_\Rp ) \simeq  \VVV_p
\]
is a Zariski locally trivial fibration over $\VVV_p$ with fibre an affine space of dimension $ m_p  \dim ( V )$,
the class of the space of $m_p$-jets of $\VVV_\Rp$ in $\KVar {\VVV_p } $ is $\LL^{m_p \dim ( V )}_p [ \VVV_p ] $. 
Finally the class
 \[
 \left  [ \Hom_\CCC ( \SSS , \VVV ) \right  ] \in \KVar {\prod_{p\in | \SSS |} \VVV_p }
 \]
is sent to the finite product
 \[ \prod_{p\in | \SSS |}   \frac{[\VVV_p]}{\LL_p^{ \dim ( V) } } \LL_p^{ ( m_p +1 ) \dim ( V ) } \in \KVar \kb .
 \]
 Thus weak equidistribution for $\SSS$ implies that
 \[
 \LL_\kb ^{-\, \mdeg  \scdot  \omega_V^{-1} } \left [ \Hom ^{ \degg_{\underline{\LLL}} \, = \,  \mdeg}_\CCC ( \CCC , \VVV  \mid \varphi )_U \right  ]  \left  [ \Hom_\CCC ( \SSS , \VVV ) \right  ] 
 \]
 tends to $\tau ( \VVV )$ when $\delta \to \infty$.
 \end{myremark}

\subsection{Equidistribution and arcs} 
Actually, the equidistribution hypothesis 
can be reformulated 
in terms of constructible sets of arcs. This reformulation is natural since we already interpreted the local factors of the motivic Tamagawa number as motivic densities of spaces of arcs. In this paragraph $ \mathsf S $ is any finite set of closed points of $\CCC$. We drop as well the previous assumption on $\LLL_\VVV$.

The restriction to $\Spec ( \widehat{\OOO_{\CCC,p}}) $  provides a morphism 
\[
\res_{\mathsf S}^\VVV : \Hom ^{ \degg_{\underline{\LLL}} \, = \,  \mdeg}_\CCC ( \CCC , \VVV ) \to \prod_{p\in \mathsf S} \Gr_\infty ( \VVV _\Rp )
\]
for every multidegree $\mdeg \in \CEff ( V )^\vee_\ZZ $.
If ${ \varphi = ( \varphi_p)_{p\in \mathsf S}}$ is a finite collection of jets such that ${\varphi_p \in \Gr_{m_p} ( \VVV_{\Rp}  ) }$ for every $p$ in $\mathsf S$, 
the schematic fibre of 
\[
 \prod_{p\in \mathsf S}  \theta_{m_p}^\infty  \circ \res_{\mathsf S}^\VVV 
\]
above $\varphi $ is written  
\[
\Hom ^{ \degg_{\underline{\LLL}} \, = \,  \mdeg}_\CCC ( \CCC , \VVV \mid  \varphi )_U .
\]
\begin{mydef}
We say that there is weak equidistribution above $\mathsf S$ at level $(m_p)_{p\in \mathsf S}$ if for every collection $\varphi =  ( \varphi_p )_{p\in \mathsf S} \in \prod_{p\in \mathsf S} \Gr_{m_p} ( \VVV_{\Rp}  )  $ of jets above $\mathsf S$, the class 
\[
 \left [ \Hom ^{  \degg_{\underline{\LLL}} \, = \,  \mdeg }_\CCC ( \CCC , \VVV \mid  \varphi )_U \right  ] \LL ^{- \, \mdeg  \scdot  \omega_V^{-1} }_\kb
 \]
tends to the non-zero effective element
\[
\tau_{\LLL_\VVV } ( \VVV \mid \varphi  )
=  
\tau_{\LLL_\VVV } ( \VVV )_{| \CCC \setminus \mathsf S }  \times \prod_{p\in \mathsf S} \mu_{\LLL_{ | \VVV_\Rp } } ^ * ( ( \theta_{m_p}^\infty )^{-1} (  \varphi_p )    \cap   \Gr ( \VVV_{\Rp}^\circ )  )   
\] 
of $\widehat{ \mathscr M_\kb} $,
when $\mdeg$ becomes arbitrarily large. 
\end{mydef}

This definition is consistent with the previous one since we have the factorisation 
\[
\res_\SSS^\VVV : \Hom^\dd _\CCC ( \CCC , \VVV ) 
\overset{\res_{\mathsf S}^\VVV}{\longrightarrow} 
\prod_{p\in \mathsf S} \Gr_{\infty} ( \VVV_\Rp  ) 
\longrightarrow 
\prod_{p\in \mathsf S}\Gr_{m_p} ( \VVV_\Rp) \simeq \Hom_\CCC \left ( \SSS  , \VVV \right ) 
\]
for every $\mathsf S $-tuple $ (m_p) \in \NN^{\mathsf S} $ and corresponding zero-dimensional subscheme $\SSS \subset \CCC $ with support $|\SSS | = \mathsf S$. 

\smallskip 

More generally, if $W $ is a product $\prod_{p\in \mathsf S} W_p$ of constructible subsets $W_p$ of $\Gr_\infty ( \VVV_\Rp ) $, 
\[
\Hom ^{ \degg_{\underline{\LLL}} \, = \,  \mdeg}_\CCC ( \CCC , \VVV \mid  W )_U
\]
is defined as the schematic fibre of $\res_{\mathsf S}^\VVV$ over $W$. Recall that each constructible set $W_p$ of arcs is nothing else but the preimage by a projection morphism of a certain constructible subset of jets. 
\begin{mydef}\label{def:equidistribution-products-constructible-subsets}
	We will say that there is 
	\textit{equidistribution with respect to $W=\prod_{p\in \mathsf S} W_p$ and the multidegree $\degg_{\underline{\LLL}}$},
	where each $W_p$ is a constructible subset of arcs, if 
\[
\left [ \Hom ^{ \degg_{\underline{\LLL}} \, = \,  \mdeg}_\CCC ( \CCC , \VVV \mid  W )_U \right  ] 
\LL ^{- \, \mdeg  \scdot  \omega_V^{-1} }_\kb
\]
tends to 
\[
\tau_{\LLL_\VVV } ( \VVV \mid W  ) = 
\tau_{\LLL_\VVV } ( \VVV )_{ | \CCC \setminus S } 
\times 
\prod_{p\in \mathsf S} 
\mu_{\LLL_{ | \VVV_\Rp } } ^ * ( W_p \cap \Gr ( \VVV_{\Rp}^\circ ) )
\]
when the multidegree becomes arbitrarily large. 

We will say that there is \textit{equidistribution of curves, with respect to the multidegree $\degg_{\underline{\LLL}}$} if the previous statement holds for every such $W$. 
\index{equidistribution of curves}
\index{motivic Tamagawa number!of a model!associated to constructible constraints}
\end{mydef}

\begin{myremark}
Note that the notion of equidistribution of curves is stronger that the motivic Batyrev-Manin-Peyre principle for curves we formulate in \cref{conj:BMP-motivic-strong-weight-top}. 
\end{myremark}

\begin{myremark} 
	In \cref{def:equidistribution-products-constructible-subsets} 
	one may ask if it would be possible to replace \textit{contructible} by \textit{measurable} 
	to get a more general notion of equidistribution, 
	or consider constructible subsets which are not products over $S$ of constructible sets, as in \cref{thm-equidistribution-toric}, but this higher level of generality would be mostly useless in the present work.  
\end{myremark}

\subsection{Checking equidistribution pointwise}\label{subsection:check-equidistribution-pointwise}
\label{subsection:checking-equidistribution-pointwise}

Let $\SSS $ be a zero-dimensional subscheme of $\CCC$. 
Assume that for every $\mdeg \in \Pic ( V )^\vee$,
there exists a $\kb$-scheme $F_\mdeg $ (which depends on $\SSS$) such that 
\[
 \Hom^{ \degg_{\underline{\LLL}} \, = \,  \mdeg }_\CCC  ( \CCC  , \VVV \mid \varphi )_U  \simeq ( F_\mdeg \otimes \kappa ( x ) )_{\text{red}}
\]
for every point $x \in \Hom_\CCC ( \SSS , \VVV )$ corresponding to a $\CCC$-morphism $\varphi : \SSS \to \VVV$.  
Then by \cref{ptn:point-wise-criterion-for-piecewise-trivial-fibrations},
the reduction map 
\[
 \Hom^{ \degg_{\underline{\LLL}} \, = \,  \mdeg}_\CCC  ( \CCC  , \VVV \mid \varphi )_U   \to \Hom_\CCC ( \SSS , \VVV )
\]
is a piecewise trivial fibration with fibre $F_\mdeg$ and by \cref{ptn:classes-of-piecewise-trivial-fibrations},
for every constructible subset $W$ of $ \Hom_\CCC ( \SSS , \VVV )$, 
\[
\left [ \Hom^{ \degg_{\underline{\LLL}} \, = \,  \mdeg}_\CCC  ( \CCC  , \VVV \mid W )_U \right ] = [ F_\mdeg ] [ W ] 
\]
in $\KVar \kb$.
Hence, if one wishes to show that there is equidistribution of curves on $\VVV$ above $\SSS$,
a strategy is to prove the existence of such an $F_\mdeg$
and then study the convergence of the normalised class 
$[ F_\mdeg ]  \LL_\kb^{- \, \mdeg  \scdot  \omega_V^{-1} } $. 
In general the situation is not as simple,
but we use a similar argument in   \cref{section:equidistribution-toric} in order to prove \cref{thm-intro:BMP-toric-varieties}.

\subsection{Equidistribution and models} 
Our goal for the end of this section is to prove the following main result, which does not depend on the choice of the filtration (dimensional or by the weight) on $\MMM_\CCC$.
\begin{mythm}[Change of model]
\label{thm:equidistribution-and-models} 
Let $\VVV $ and $\VVV ' $ be two proper models over $\CCC$ 
of the same Fano-like $F$-variety $V$, together with models
$\underline{\LLL} = ( \LLL_i)$ and $\underline{\LLL ' } = ( \LLL_i ' )$,
respectively on $\VVV$ and $\VVV '$,
of a family $(L_i)$ of invertible sheaves forming a $\ZZ$-basis of $\Pic ( V) $,
as in
\cref{setting:main-setting}. 


Then there is equidistribution of curves for $(\VVV , \underline{\LLL} )$, in the sense of \cref{def:equidistribution-products-constructible-subsets},
if and only if there is equidistribution of curves for $( \VVV ' , \underline{\LLL ' } )$. 

\end{mythm}

The remainder of this section is devoted to the proof of \cref{thm:equidistribution-and-models}.
We take $\VVV$, $\underline \LLL$, $\VVV ' $ and $\underline \LLL '$ as in \cref{setting:main-setting}. 
As before, 
we know the existence of a non-empty open subset $\CCC '  \subset \CCC$ 
above which we have an isomorphism 
of $\CCC ' $-schemes.  
By \cref{cor:existence-global-wNm-dominating-2-global-wNm}, we can find a 
proper model $\widetilde{\VVV } \to \CCC$ of $V$ whose $\CCC$-smooth locus 
is a Néron smoothening of both $\VVV ' $ and $\VVV$. Above $\CCC ' $, the three models are isomorphic. 
\[
\begin{tikzcd}
				& \widetilde{\VVV} \arrow[dl,"f",swap]\arrow[dr,"f'"]\arrow[dd,"{\widetilde \pi }"] & 							\\
\VVV \arrow[dr,"\pi",swap] &   & \VVV ' \arrow[dl,"{\pi '}"] \\
 		& \CCC & 
\end{tikzcd}
\]
This diagram induces morphisms
\[
\begin{tikzcd}
	& \Hom_\CCC \left ( \CCC , \widetilde{\VVV } \right )_U \arrow[dl,"{f_*}",swap] \arrow[dr,"{f_* '}"]& \\
	 \Hom_\CCC ( \CCC , \VVV )_U &    						& \Hom_\CCC \left  ( \CCC , \VVV ' \right )_U
\end{tikzcd}
\]
between moduli spaces of sections. 

Let 
$\widetilde{\LLL_i}$
and 
$\widetilde{\LLL_i '} $ be respectively the pull-backs of the sheaves $\LLL_i$ and $\LLL_i ' $ to $\widetilde{\VVV }$
for all $i$.  
Then both 
$\widetilde{\LLL_i}$
and 
$\widetilde{\LLL_i' }$
are models of
 $L_i$ on $\widetilde{\VVV}$. 
Up to shrinking $\CCC ' $, we can assume that they are isomorphic above $\CCC ' $. 
If $\widetilde \sigma  $ is a section of $\widetilde{\VVV } $ and $\sigma = f \circ \widetilde \sigma  $, one has the relation 
\[
\deg \left  ( ( \widetilde \sigma  )^* \widetilde{\LLL_i}
 \right ) =  \deg  \left ( ( \widetilde \sigma  )^*  f^* \LLL_i \right )  =  \deg ( ( f \circ \widetilde \sigma  )^* \LLL_i ) = \deg ( \sigma^* \LLL_i ) 
\]
for all $i \in \{ 1 , ... , r \}$,
so that $f_*$ bijectively sends points of 
\[
\Hom_\CCC^{\degg_{\underline{\widetilde{\LLL }}} \, = \,  \mdeg } \left ( \CCC , \widetilde{\VVV } \right )_U
\]
to points of 
\[
\Hom_\CCC^{\degg_{\underline{\LLL }} \, = \,  \mdeg } ( \CCC , \VVV )_U ,
\] 
and similarly for $f'_*$ when considering the multidegrees given by $\underline{\widetilde{\LLL '}}$ and $\underline{\LLL '}$.
By \cref{ptn:point-wise-criterion-for-piecewise-trivial-fibrations} this implies equality of the corresponding classes in $\KVar \kb$.

We are going to compare the multidegrees $\degg_{\underline{\widetilde{\LLL}}}$ and $\degg_{\underline{\widetilde{\LLL '}}} $. 

\subsubsection{Lifting equidistribution}
As an application of the change-of-variable formula \cref{ptn:change-of-variable-smooth-case},
we show that equidistribution of curves holds 
for $(\VVV , \underline{\LLL} )$
if and only if it holds for $( \widetilde{\VVV } , \underline{\widetilde{\LLL}} ) $.

Let $\mathsf S$ be the complement of $\CCC ' $ in $\CCC$. 
\begin{mylemma}\label{lemma:lifting-equidistribution-and-convergence-speed}
	Let $\widetilde W$ be a finite product of constructible subsets  $\widetilde{W}_p \subset \Gr ( \widetilde{\VVV}_\Rp ) $ for  $p\in \mathsf S$, and let $W$ be its image by $f$.
Then 
\[
\tau_{ \widetilde{\LLL_{\VVV } } } ( \widetilde{\VVV } \mid \widetilde{W} ) = \tau_{ \LLL_{\VVV }  } ( \VVV  \mid W )
\]
and 
\[ 
\left [ \Hom_\CCC^{\degg_{\underline{\widetilde{\LLL }}} \, = \,  \mdeg } \left ( \CCC , \widetilde \VVV \mid \widetilde{W} \right )_U \right ]
= 
\left [ \Hom_\CCC^{\degg_{\underline{\LLL }} \, = \,  \mdeg }  ( \CCC , \VVV \mid  W )_U  \right ] 
\]
for every $\mdeg \in \Pic ( V ) ^\vee$. 

In particular, for every $m\in \ZZ$,
\[
\tau_{ \widetilde{\LLL_{\VVV } } } \left ( \widetilde{\VVV } \mid \widetilde{W} \right )  - 	\left [ \Hom_\CCC^{\degg_{\underline{\widetilde{\LLL }}} \, = \,  \mdeg } \left ( \CCC , \widetilde \VVV \mid \widetilde{W} \right ) _U \right ]\LL^{- \, \mdeg \scdot \omega_V^{-1}} \in \mathcal F^m \MMM_\kb 
\]
if and only if 
\[
\tau_{ \LLL_{\VVV }  } ( \VVV  \mid W )  - 	\left [ \Hom_\CCC^{\degg_{\underline{\LLL }} \, = \,  \mdeg }  ( \CCC , \VVV \mid  W )_U \right ]\LL^{- \, \mdeg \scdot \omega_V^{-1}} \in \mathcal F^m \MMM_\kb . 
\]
\end{mylemma}

\begin{proof}
	Up to shrinking $\CCC ' $ and adding trivial conditions, one can assume that $\mathsf S$ is contained in the complement of $\CCC' $. 
	By Theorem 3.2.2 of \cite[Chap. 5]{chambert2018motivic},
	the image of $\widetilde{W}_p$ 
	in $\Gr ( \VVV_\Rp )$ is a constructible subset $W_p  $. 
Then $f_*$ 
bijectively sends points of 
$\Hom_\CCC^{\degg_{\underline{\widetilde{\LLL }}} \, = \,  \mdeg } \left ( \CCC , \widetilde \VVV \mid \widetilde{W} \right )_U$ 
to points of $\Hom_\CCC^{\degg_{\underline{\LLL }} \, = \,  \mdeg }  ( \CCC , \VVV \mid  W )_U $ and 
 by \cref{ptn:point-wise-criterion-for-piecewise-trivial-fibrations},	
	\[
	\left [ \Hom_\CCC^{\degg_{\underline{\widetilde{\LLL }}} \, = \,  \mdeg } \left ( \CCC , \widetilde \VVV \mid \widetilde{W} \right ) _U \right ] = \left [ \Hom_\CCC^{\degg_{\underline{\LLL }} \, = \,  \mdeg }  ( \CCC , \VVV \mid  W )_U \right ]
	\]
	so that the only thing to show is the equality of the motivic Tamagawa numbers. 
	
	Up to shrinking $\CCC ' $ again, we only have to show the equality of local factors 
	\[
	\mu_{\LLL_{ | \VVV_\Rp } } ^ * 
	\left ( W_p \cap \Gr \left ( \VVV_{\Rp}^\circ \right ) \right )
	=
 \mu_{\widetilde{\LLL}_{ | \widetilde{\VVV}_\Rp } } ^ * 
 \left ( \widetilde{W}_p \cap \Gr \left ( \widetilde{\VVV}_{\Rp}^\circ \right ) \right )
	\]
	above closed points $p\in \mathsf S$. 
	By assumption $V$ is smooth, both $\widetilde{\VVV } $ and $\VVV  $ are models of $V$, thus by Corollary 3.2.4 of \cite[Chap. 5]{chambert2018motivic}
	$\ordjac_{f_{\Rp}}$ only takes a finite number of values. 
	By the change of variable formula, \cref{ptn:change-of-variable-smooth-case}, applied to the constructible function 
	\[
\varepsilon_{\LLL_{\VVV} - \left ( \Lambda^{n } \Omega^1_{\VVV   / \Rp } \right )^\vee }  
	\tag{see \cref{def:motivic-measure-with-respect-model-line-bundle}}
	\] 
	one has the following relation between local factors 
	\begin{align*}
	& \mu_{\LLL_{ | \VVV_\Rp } } ^ * \left ( W_p \cap \Gr \left ( \VVV_{\Rp}^\circ \right ) \right )\\
	& =\int_{ W_p \cap \Gr \left ( \VVV_\Rp^\circ \right )}  \LL^{- \varepsilon_{\LLL_{\VVV} - ( \Lambda^{n } \Omega^1_{\VVV   / \Rp } )^\vee }} \mathrm d \mu_{\VVV_\Rp }\\
	& = 
		\int_{ W_p \cap \Gr \left ( \VVV_\Rp^\circ \right )} 	\LL^{- \ord_{\LLL_{\VVV }} + \ord_{( \Lambda^{n } \Omega^1_{\VVV  / \Rp })^\vee}
} \mathrm d \mu_{\VVV_\Rp } \\
	& =  \int_{\widetilde{W}_p  \cap \Gr \left ( \widetilde{\VVV}_\Rp^\circ \right )} 	\LL^{ - f^* \ord_{\LLL_{\VVV }} + f^* \ord_{( \Lambda^{n } \Omega^1_{\VVV  / \Rp })^\vee} - \ordjac_{f} }  \mathrm d \mu_{\widetilde{\VVV}_\Rp }   \\
	& =  \int_{ \widetilde{W}_p  \cap \Gr \left ( \widetilde{\VVV}_\Rp^\circ \right )} \LL^{- \ord_{ \widetilde{\LLL_{\VVV }}} + \ord_{( \Lambda^{n } \Omega^1_{\widetilde{\VVV}  / \Rp })^\vee} } \mathrm d \mu_{\widetilde{\VVV}_\Rp } \\
	& =  \mu_{\widetilde{\LLL}_{ | \widetilde{\VVV}_\Rp } } ^ * \left ( \widetilde{W}_p \cap \Gr \left ( \widetilde{\VVV}_{\Rp}^\circ \right ) \right )
	\end{align*}
in $\MMM_{\VVV_\Rp }$,
where we used the relations
\[
\ord_{\widetilde{\LLL_\VVV }} - f^* \ord_{\LLL_\VVV} = 0
\]
and
\[
f^* \ord_{\left ( \Lambda^{n } \Omega^1_{\VVV  / \Rp } \right )^\vee} - \ord_{\left ( \Lambda^{n } \Omega^1_{\widetilde{\VVV} / \Rp }\right )^\vee} = \ordjac_{f}
\tag{by \cref{ptn:ordjac-Neron-smoothening}}
\]
above the smooth $R_p$-locus.
Taking the product over $\mathsf S$, one gets
	\[
	\tau_{\LLL_{\VVV  }}  ( \VVV \mid W  ) = \tau_{\widetilde{\LLL_{\VVV  }}} ( \widetilde{\VVV} \mid \widetilde{W}  ) 
	\]
hence the lemma.
\end{proof}

\subsubsection{Switching the degree}

The difference of degrees on $\Gr \left ( \widetilde{\VVV}_\Rp \right )$ is given by the map
\[
\varepsilon_{ \underline{\widetilde{\LLL '}} - \underline{\widetilde{\LLL }}  } : \left \{ 
\begin{array}{rll}
	 \Gr \left ( \widetilde{\VVV}_\Rp  \right )  &  \longrightarrow &  \Pic ( V ) ^ \vee  \\
	x & \mapsto & \displaystyle \left ( \otimes_i L_i^{\otimes d_i} \longmapsto \sum_{i=1}^r d_i  \varepsilon_{\widetilde{\LLL_i '} - \widetilde{\LLL_i}  } ( x )  \right )
\end{array}
\right. 
\]
which is trivial for all $p\notin \mathsf S$. 
For any $\varepsilon_p \in \Pic (V)^\vee $,
let 
\[
\widetilde{\WWW}_p  ( \varepsilon_p  ) 
=
\varepsilon_{ \underline{\widetilde{\LLL '}} - \underline{\widetilde{\LLL }}  }^{-1} ( \{ \varepsilon_p \} ).
\]
As a direct consequence of \cref{lemma:local-difference-is-constructible}, we have the following.

\begin{mylemma}
\label{lemma:difference-of-degrees-takes-finite-nb-of-values}
The map $\varepsilon_{ \underline{\widetilde{\LLL '}} - \underline{\widetilde{\LLL }}  } $
is constructible and 
there is only a finite number of values of $\varepsilon_p  \in \Pic ( V )^\vee $ 
for which 
$\widetilde{\WWW}_p  ( \varepsilon_p  ) $
is non-empty. 
\end{mylemma}
Now, for every $\varepsilon = ( \varepsilon_{s} ) \in \left ( \Pic ( V ) ^\vee \right )^{\mathsf S}$,
let 
\[
\widetilde{\WWW}  ( \varepsilon ) 
= 
\prod_{s \in \mathsf S}  \widetilde{\WWW}_s  ( \varepsilon_s  ) 
\subset 
\prod_{s \in \mathsf S} \Gr_\infty \left ( \widetilde{\VVV}_{R_s} \right )  
\]
and let $\widetilde W $ be any finite product $\prod_{s\in \mathsf S} \widetilde W_s $ of constructible subsets ${\widetilde{W}_s \subset \Gr \left ( \widetilde{\VVV}_{R_s} \right )}$. Let $W$, $W'$, $W_s$ and $W_s '$ be the corresponding images by $f$ and $f'$.
By the previous lifting \cref{lemma:lifting-equidistribution-and-convergence-speed},
\[
 \left [ \Hom^{\degg_{\underline{\widetilde{\LLL }}} \, = \,  \mdeg }_\CCC  (\CCC , \widetilde{\VVV} \mid  \widetilde{W} \cap \widetilde{\WWW}  ( \varepsilon )  )_U \right ] 
 \LL^{- \, \mdeg   \scdot \omega_V^{-1} }
 \longrightarrow 
 \tau_{\widetilde{\LLL_\VVV}} \left ( \widetilde \VVV \mid  \widetilde{W} \cap  \widetilde{\WWW}  ( \varepsilon ) \right )
\]
when $\mdeg \in  \CEff ( V )^\vee_\ZZ $ becomes arbitrarily large.

Let $\WWW ' ( \varepsilon )$ 
be the image of $\widetilde \WWW ( \varepsilon ) $ in $\prod_{s\in \mathsf S} \Gr ( \VVV_{R_s}' )$.
We decompose our classes as follows:
\begin{align*}
\left [\Hom^{ \degg_{\underline{\LLL ' }} \, = \,  \mdeg ' }_\CCC \left ( \CCC , \VVV ' \mid W ' \right ) \right ] & = \sum_{ \varepsilon \in \left ( \Pic ( V ) ^\vee \right )^{\mathsf S} } 
\left [
\Hom^{\degg_{\underline{\LLL ' }} \, = \, \mdeg '}_\CCC \left ( \CCC , \VVV ' \mid W ' \cap \mathscr W ' ( \varepsilon ) \right ) 
\right ] \\
& = \sum_{ \varepsilon \in \left ( \Pic ( V ) ^\vee \right )^{\mathsf S} } 
\left [
\Hom^{\degg_{\underline{\widetilde{\LLL '}}} \, = \,  \mdeg ' }_\CCC 
\left ( \CCC , \widetilde{\VVV} \mid  \widetilde{W} \cap \widetilde{\WWW}  ( \varepsilon ) \right ) 
\right ]
\end{align*}
these sums being finite by \cref{lemma:difference-of-degrees-takes-finite-nb-of-values}. Normalising, we get 
\begin{align*}
& \left [
\Hom^{ \degg_{\underline{\LLL ' }} \, = \,  \mdeg ' }_\CCC \left ( \CCC , \VVV ' \mid W ' \right ) 
\right ]  \LL^{-\, \mdeg ' \scdot  \omega_V^{-1}  } \\
& = \sum_{ \varepsilon \in \left ( \Pic ( V ) ^\vee \right )^{\mathsf S} } 
\left [
\Hom^{\degg_{\underline{\widetilde{\LLL '}}} \, = \,  \mdeg ' }_\CCC \left ( \CCC , \widetilde{\VVV} \mid \widetilde{W} \cap \widetilde{\WWW}  ( \varepsilon ) \right ) 
\right ] \LL^{-\, \mdeg  ' \scdot  \omega_V^{-1}  } \\
& = \sum_{\varepsilon \in \left ( \Pic ( V ) ^\vee \right )^{\mathsf S} } 
\left [\Hom^{\degg_{\underline{\widetilde{\LLL }}} \, = \,  \mdeg ' - | \varepsilon | }_\CCC \left ( \CCC , \widetilde{\VVV} \mid \widetilde{W} \cap  \widetilde{\WWW}  ( \varepsilon ) \right ) \right ] 
\LL^{-\, ( \mdeg ' - | \varepsilon | ) \scdot \omega_V^{-1}} \LL^{ - |\varepsilon | \scdot \omega_V^{-1}} \\
\end{align*}
where $|\varepsilon |$ stands for $\sum_{s\in \mathsf S} \varepsilon_{s} \in \Pic ( V )^\vee$. 
Then 
\begin{align*}
& \left [\Hom^{ \degg_{\underline{\LLL ' }} \, = \,  \mdeg ' }_\CCC \left ( \CCC , \VVV ' \mid  W ' \right ) \right ]  \LL^{-\, \mdeg ' \scdot  \omega_V^{-1}  }
\\
& \underset{d \left ( \mdeg ' , \partial \CEff ( V ) ^\vee  \right ) \to \infty }{\longrightarrow} \sum_{\varepsilon \in \left ( \Pic ( V ) ^\vee \right )^{\mathsf S} } \tau_{\widetilde{\LLL_\VVV}} \left ( \widetilde \VVV \mid \widetilde{W} \cap  \widetilde{\WWW}  ( \varepsilon ) \right )
 \LL^{ - | \varepsilon | \scdot \omega_V^{-1}} = \tau_{\widetilde{\LLL_\VVV '}} \left (   \widetilde \VVV \mid \widetilde{W} \right ) = \tau_{\LLL_\VVV '} \left ( \VVV '  \mid W ' \right )  
\end{align*}
and \cref{thm:equidistribution-and-models} is proved.

\bigskip 


\section{Rational curves on smooth split projective toric varieties}

\label{section-BMP-toric}
In this section, 
we prove equidistribution of rational curves on smooth split projective toric varieties over any base field.

As a warm-up, we start with
proving that the motivic Batyrev-Manin-Peyre principle holds for rational curves on this class of varieties, 
in line with the works of Bourqui \cite{bourqui2003fonctions,bourqui2009produit}, Bilu \cite{bilu2018motivic} and Bilu-Das-Howe \cite{bilu2020zeta}, see \cref{thm:motivic-BMP-toric}. 

Then we generalize this result, by proving equidistribution of rational curves on smooth split projective toric varieties, see \cref{thm-equidistribution-toric}. 

\subsection{Geometric setting}
General references for toric varieties are \cite{oda1988convex,fulton1993introduction,cox2011toric}.
Let $U$ be  a split torus of dimension $n$ over $\kb$.
Let 
\[
\mathcal X^* (  U )  = \Hom (  U , \mathbf G_m ) 
\]
be its group of characters and $\mathcal X_* (  U ) = \Hom_\ZZ ( \mathcal X^* ( U ) , \ZZ )  $ its dual as a $\ZZ$-module. 
Let $\Sigma$ be a complete and regular fan of $\mathcal X_* (  U ) $, which defines a smooth projective toric variety $V_\Sigma$ over $\kb$, with open orbit isomorphic to $U$. 
Let 
\[
r  = \rg \left ( \Pic \left ( V_\Sigma \right ) \right )
\]
be its Picard number,
$\Sigma ( 1 ) $ the set of rays of the fan $\Sigma $ 
and $(D_\alpha )_{\alpha \in \Sigma ( 1 ) }$ 
the set of its $U$-invariant divisors.
Each ray $\alpha \in \Sigma ( 1 ) $ admits a minimal generator $\rho_\alpha \in \mathcal X_* ( U ) $ 
and
the map sending a character $\chi \in \mathcal X^* ( U ) $ to the divisor 
\[
\sum_{\alpha \in \Sigma ( 1 ) } \langle \chi , \rho_\alpha \rangle D_\alpha 
\]
is part of the exact sequence \cite[Theorem 4.1.3]{cox2011toric}
\begin{equation}\label{exact-sequence:torus}
	0 \to \mathcal X^* ( U ) \to \oplus_{\alpha \in \Sigma ( 1 ) } \ZZ D_\alpha \to \Pic \left ( V_\Sigma \right ) \to  0 
\end{equation}
which provides, in particular, the equality 
\[ 
| \Sigma ( 1 ) | = n + r .
\]
If $\sigma$ is an element of the fan $\Sigma$, let $\sigma ( 1 ) \subset \Sigma ( 1 ) $ be the subset of rays which are faces of $\sigma$. 

\subsection{Möbius functions}
Let $B_\Sigma \subset \{ 0 , 1 \}^{\Sigma ( 1 )}$ be the complement of the image of
\[
\begin{array}{rll}
	\Sigma & \longrightarrow & \{ 0 , 1 \}^{\Sigma ( 1 )}\\
	\sigma & \longmapsto & ( \mathbf 1_{\sigma ( 1 )} ( \alpha ) )_{\alpha \in \Sigma ( 1 )} . \\
\end{array}
\]
In \cite[\S 3.5]{bourqui2009produit}, this set $B_\Sigma$ is described explicitly as
\[
B_\Sigma = \{ \nn \in \{ 0 , 1 \}^{\Sigma ( 1 ) } \mid \forall \, \sigma \in \Sigma , \, \exists \, \alpha \in \Sigma ( 1 ), \, \alpha \notin \sigma ( 1), \text{ and } n_\alpha = 1 \} .
\]
It has a geometric interpretation in terms of the effective divisors $D_\alpha$: it corresponds to the subsets $I\subset \Sigma ( 1 ) $ such that 
\[
\cap_{\alpha \in I} D_\alpha = \varnothing . 
\]
Then, the universal torsor of $V_\Sigma$
admits an explicit description
which goes back to Salberger \cite{salberger1998tamagawa}: 
\[
\mathcal T_\Sigma 
=
\mathbf A^{\Sigma ( 1 )}
\setminus 
\left ( 
\cup_{J\in B_\Sigma}
\cap_{\alpha \in J} \{ x_\alpha = 0 \}
\right ) .
\]
\index{universal torsor of a toric variety}

\subsubsection{Local Möbius function}
Bourqui inductively defines a local Möbius function 
\[
 \mu_{B_\Sigma}^0 : \{ 0 , 1 \}^{\Sigma ( 1 ) } \to \ZZ
 \]
through the relation 
\[
\mathbf 1_{ \{ 0 , 1 \}^{\Sigma ( 1 )} \setminus B_\Sigma } ( \nn ) = \sum_{0\leqslant \nn ' \leqslant \nn } \mu_{B_\Sigma}^0 ( \nn '  ) 
\]
for every $\nn \in \{ 0 , 1 \}^{\Sigma ( 1 ) }$. It comes with a generating polynomial 
\[
P_{B_\Sigma} ( \TT )  = \sum_{\nn \in \{ 0 , 1 \}^{\Sigma ( 1 )}} \mu_{B_\Sigma}^0 ( \nn ) \TT^\nn 
\]
and a series
\[
Q_{B_\Sigma} ( \TT )  = \frac{P_{B_\Sigma} ( \TT ) }{\prod_{\alpha \in \Sigma ( 1  ) }  ( 1 - t_\alpha ) } . 
\]
Let 
\[
A \left ( B_\Sigma \right ) \subset \NN^{\Sigma ( 1 ) } 
\label{definition-ABsigma}
\] be the set of tuples $\nn\in \NN^{\Sigma ( 1 ) }$ 
such that there is no $\nn ' \in B_{\Sigma} $ with $\nn \geqslant \nn ' $. In particular, ${\mathbf 0 \in A \left ( B_\Sigma \right )}$. 
It is the set of elements of $\NN^{\Sigma ( 1 ) }$ \textit{not lying above} $B_\Sigma $ in the sense of \cite[\S 4.4]{bilu2020zeta}.
Let 
\[
\mu_{B_\Sigma } : \NN^{\Sigma ( 1 ) } \to \ZZ 
\] be the local Möbius function defined by the relation 
\begin{equation}
\label{def-eq:toric-local-mobius-function}
	\mathbf 1_{A ( B_\Sigma )  } ( \nn ) = \sum_{0\leqslant \nn ' \leqslant \nn} \mu_{ B_\Sigma } ( \nn ' ) .
\end{equation}

As Bilu-Das-Howe point out in \cite[\S 5.2]{bilu2020zeta},  $\mu_{B_\Sigma } $ coincides with $\mu_{B_\Sigma}^0 $ on $\{ 0 , 1 \}^{\Sigma ( 1 ) }$ 
and is zero outside of this set.
Hence
\[
Q_{ B_\Sigma } ( \TT ) =  \sum_{\nn \in A \left ( B_\Sigma \right )}  \TT^\nn  .
\]

\subsubsection{Global motivic Möbius function}
For any $\ee \in \NN^{\Sigma ( 1 ) }$, let $( \PP^1_\kb )_{B_\Sigma}^\ee$ be the open subset of 
\[
\Sym^\ee_{/ \kb} ( \PP^1_\kb ) 
\] parametrizing 
$\Sigma ( 1  ) $-tuples of effective zero-cycles of degree $e_\alpha$ having disjoint supports with respect to $B_\Sigma $. 
More precisely,
\[
( \PP^1_\kb )_{B_\Sigma}^\ee = \left  \{ ( C_\alpha ) \in \Sym^\ee_{/ \kb} ( \PP^1_\kb ) 
 \; \left | \; \forall \, J \in B_\Sigma \quad \bigcap_{\alpha \in J } \Supp ( C_\alpha )    = \varnothing \right \} \right. . 
\]
Up to a tuple of multiplicative factors, this set corresponds to $\Sigma ( 1 )$-tuples of homogeneous polynomials $P(T_0,T_1)$ of degree $e_\alpha$ with coefficients in $\kb$ such that 
for all 
$J \in B_\Sigma$
the polynomials $ (P_\alpha )_{\alpha \in J} $ 
have no common root in any finite extension of $\kb$,
see \cite[Lemme 5.10]{bourqui2009produit}. 

Then, applying the definition of the motivic Euler product in Bilu's sense \cite{bilu2018motivic}, 
we get 
\[
\prod_{p\in \PP^1_\kb } Q_{B_\Sigma } ( \TT  ) = \sum_{\dd \in \NN^{\Sigma ( 1 )}} \left [ \left ( \PP^1_\kb \right )^\dd_{B_\Sigma} \right ] \TT^\dd .
\]
The formalism of pattern-avoiding zero cycles 
allows Bilu, Das and Howe to work with Bilu's motivic Euler product 
and to give a positive answer to a technical question of Bourqui 
\cite[Question 5]{bourqui2009produit}, 
which in turn provides a lift of Bourqui's main theorem \cite[Théorème 1.1]{bourqui2009produit} 
from the localised Grothendieck ring of Chow motives $\mathscr M_\kb^\chi$
to the localised Grothendieck ring of varieties $\mathscr M_\kb$, see \cite[Lemma 4.5.4 \& Remark 4.5.7]{bilu2020zeta}. 
Indeed, since 
\[
Q_{B_\Sigma } ( \TT )  = P_{B_\Sigma } ( \TT )  \prod_{\alpha \in \Sigma ( 1 ) } ( 1 - t_\alpha )^{-1} 
\]
one obtains, by taking motivic Euler products (in Bilu's sense)
\begin{equation}\label{equation:mobius-inversion-P1-in-terms-of-series}
\sum_{\dd \in \NN^{\Sigma ( 1 )}} \left [ \left ( \PP^1_\kb \right )^\dd_{B_\Sigma} \right ] \TT^\dd 
 = 
 \prod_{p\in \PP^1_\kb} Q_{B_\Sigma} ( \TT )  = \prod_{p\in \PP^1_\kb} P_{B_\Sigma } ( \TT ) \times  \prod_{\alpha \in \Sigma ( 1 )} Z_{\PP^1_\kb} ( t_\alpha ) 
\end{equation}
where 
$Z_{\PP^1_\kb} ( t ) $ is Kapranov's zeta function of the projective line 
\[
Z_{\PP^1_\kb} ( t )  = \sum_{e \geqslant 0} \left [ \Sym_{/ \kb}^e  ( \PP^1_\kb  ) \right ] t^e .
\]
Bourqui's construction \cite[Section 3.3]{bourqui2009produit}, applied to the projective line over $\kb$ and the set $B_\Sigma$, 
provides a motivic global Möbius function 
\[
\mu_\Sigma : \NN^{\Sigma ( 1  )} \to \mathscr M_\kb
\]
given by the relation
 \[
 \left [ ( \PP^1_\kb )_{B_\Sigma}^\ee \right ] 
 = 
 \sum_{
 \substack{\ee \in \NN^{\Sigma ( 1 )}
 \\
 \ee ' \leqslant \ee}
 }
 \mu_{\Sigma } ( \ee  '  ) \left [ \Sym_{/ \kb}^{\ee - \ee '} ( \PP^1_\kb  )\right ]  
 \]
for all $\ee \in \NN^{\Sigma ( 1 )}$,
which is nothing else than what one obtains by considering the coefficient of multidegree $\ee$ in expression \eqref{equation:mobius-inversion-P1-in-terms-of-series}. 

It follows from the definitions that the motivic global Möbius function is linked to the local one by the relation 
\[
\prod_{p\in \PP^1_\kb } \left ( \sum_{\mm \in \NN^{\Sigma ( 1 )}} \mu_{B_\Sigma } ( \mm ) \TT^\mm \right ) = \sum_{\ee \in \NN^{\Sigma ( 1  ) }} \mu_\Sigma  ( \ee ) \TT^\ee . 
\]

\subsection{Motivic Tamagawa number}
By the following Proposition and Remark, the constant $\tau ( V_\Sigma )$ is well-defined in $\widehat{\mathscr M_\kb}^{\dim}$. 
\begin{myptn}\label{ptn:convergence-Euler-product-toric-varieties}
The motivic Euler product 
\[
 \left ( \prod_{p\in \PP^1_\kb } P_{B_\Sigma } ( \TT  ) \right ) \left ( \LL^{-1}_\kb \right ) = 
\prod_{p\in \PP^1_\kb } \left ( \sum_{\mm \in \NN^{\Sigma (1) } } \mu_{B_\Sigma} ( \mm ) \LL_p^{ - | \mm |} \right ) = \sum_{\ee \in\NN^{\Sigma ( 1 )}} \mu_\Sigma ( \ee ) \LL_\kb^{- | \ee |}
\]
is well-defined in $\widehat{\mathscr M_\kb}^{\dim}$.
\end{myptn}
\begin{proof}
By \cite[Lemme 3.8]{bourqui2009produit}, the valuation of $P_{B_\Sigma } ( T ) -1 $ is at least equal to $2$. Thus by \cite[Lemma 4.2.5]{bilu2020zeta} the formal motivic Euler product $\prod_{p\in \PP^1_\kb } P_{B_\Sigma } ( \TT  )$ converges at $\TT=\LL^{-1}$ for the dimensional filtration.  
\end{proof}

\begin{myremark}\label{remark:local-factor-Euler-product-toric-varieties}
The local factor of the motivic Euler product above is actually 
\[
\left ( P_{B_\Sigma} ( \TT ) \right ) _{| t_\alpha = \LL_p^{-1} } 
= 
\sum_{\mm \in \NN^{\Sigma (1) } } \mu_{B_\Sigma} ( \mm ) \LL_p^{-| \mm |} =  \frac{[V_\Sigma \otimes \kappa ( p )] }{\LL_p^{\dim ( V_\Sigma ) }} \left ( 1 - \LL_p^{-1} \right )^{\rg  \left ( \Pic ( V_\Sigma ) \right  )} = {[ \mathcal T_\Sigma \otimes \kappa ( p ) ] \over \LL_p^{| \Sigma ( 1 ) |}}. 
\]
Indeed, 
we can interpret $[ \mathcal T_\Sigma \otimes \kappa ( p )] \LL_p^{-| \Sigma ( 1 ) |}$
as the motivic density of $\kappa ( p )$-arcs of $\mathbf A_\kb^{\Sigma ( 1 )}$
with origin in $\mathcal T_\Sigma$,
and $\LL_p^{- | \mm |}$
as the motivic density of the subspace $V_\mm$ 
of $\kappa ( p )$-arcs with $\Sigma ( 1 )$-tuple of valuations greater than $\mm$:
\[
V_\mm  = \{ x \in \mathbf A_\kb^{\Sigma ( 1 )} ( \kappa ( p ) [[t]] ) \mid \forall \, \alpha \in \Sigma ( 1 ) , \; x_\alpha \in t^{m_\alpha} \kappa ( p ) [[t ]] \} .
\]
We consider as well the subspace of arcs with given valuation:
\begin{align*}
V_\mm^\circ  & = \{ x \in \mathbf A_\kb^{\Sigma ( 1 )} ( \kappa ( p ) [[t]] ) \mid \forall \, \alpha \in \Sigma ( 1 ) , \; x_\alpha \in t^{m_\alpha} \kappa ( p ) [[t ]] 
\text{ and } x_\alpha \notin t^{m_\alpha +1 } \kappa ( p ) [[t ]] 
\} \\
& = V_{\mm } \setminus \bigcup_{ \alpha \in  \Sigma ( 1 )} V_{\mm  + \mathbf 1_\alpha } .
\end{align*}
The arc-space of $\mathcal T_\Sigma$
is the space of arcs whose $\Sigma ( 1 )$-tuple of valuations 
does not lie above $B_\Sigma$:
\[
\Gr_\infty \left ( \mathcal T_\Sigma \otimes \kappa ( p ) \right ) = \bigsqcup_{\mm  \in A ( B_\Sigma )} V_{\mm }^\circ 
=
\bigsqcup_{\mm  \in A ( B_\Sigma )} 
\left ( 
V_{\mm } \setminus \bigcup_{ \alpha \in  \Sigma ( 1 )} V_{\mm  + \mathbf 1_\alpha } 
\right )
\]
and thus 
\begin{align*}
	\left [ 
\mathcal T_\Sigma \otimes \kappa ( p )
\right ]
\LL_p^{-|\Sigma ( 1 )|} 
& =
\sum_{\mm  \in A ( B_\Sigma )} \mu_V \left ( 
V_{\mm } \setminus \bigcup_{\alpha \in  \Sigma ( 1 )} V_{\mm  + \mathbf 1_\alpha } 
\right ) \\
& = \sum_{\mm  \in A ( B_\Sigma )} 
\sum_{J\subset \Sigma ( 1 )} ( - 1 )^{|J|}
\mu_V \left ( 
\bigcap_{\alpha \in J } V_{\mm  + \mathbf 1_\alpha } 
\right ) \\
& = \sum_{\mm  \in A ( B_\Sigma )}
\LL_p^{-|\mm |} \prod_{\alpha \in \Sigma ( 1 )} \left ( 1 - \LL_p^{-1} \right ) \\
& = \left ( \left ( Q_{B_\Sigma} ( \TT ) \right ) \prod_{\alpha \in \Sigma ( 1 )} ( 1 - t_\alpha ) \right )_{| t_\alpha = \LL_p^{-1} } \\
& = \left ( P_{B_\Sigma} ( \TT ) \right ) _{| t_\alpha = \LL_p^{-1} } 
\end{align*}
as expected.

By compatibility of formal motivic Euler products with changes of variables of the form $t'_\alpha = \LL_\kb^a t _\alpha$
with $a$ an integer \cite[\S 3.6.4]{bilu2018motivic}, together with the compatibility with partial specialisation \cite[Lemma 6.5.1]{bilu2021motivic}, 
what we get by taking the corresponding motivic Euler product is exactly the motivic Tamagawa number of $V_\Sigma$ given by \cref{def:motivic-Tamagawa-number-of-a-model} and \cref{def:euler-product-over-curves}:
\[
\left ( \prod_{p\in \PP^1_\kb }  P_{B_\Sigma} ( \TT )  \right ) _{| t_\alpha = \LL_\kb^{-1} } 
= 
\left ( 
\prod_{p\in \PP^1_\kb } \left ( 1 + \left ( {[ \mathcal T_\Sigma \otimes \kappa ( p ) ] \over \LL_p^{| \Sigma ( 1 ) |}} - 1  \right ) t \right ) 
\right )_{| t = 1 } 
= 
\tau ( V_\Sigma ) .
\]
\end{myremark}

\subsection{Non-constrained rational curves}

Let 
\[
\Hom_\kb^\mdeg ( \PP^1_\kb , V_\Sigma )_U 
\]
be the quasi-projective scheme 
parametrizing morphisms $\PP^1_\kb \to V_\Sigma $ of multidegree ${\mdeg \in \Pic ( X )^\vee}$ 
intersecting the dense open subset $U \subset V_\Sigma$ (see \cref{def:multidegree-relative} and \cref{lemma:representability-Hom-functor}). 
It is empty whenever $\mdeg \notin \CEff ( V_\Sigma ) ^\vee _\ZZ$
(see the remark after Notation 5.4 in 
\cite{bourqui2009produit}),
so we will always assume that $\mdeg \in \CEff ( V_\Sigma ) ^\vee _\ZZ$ in the remaining of this section.

Through the injection $\Pic ( V_\Sigma ) ^\vee \hookrightarrow \ZZ^{\Sigma ( 1 )}$ given by the exact sequence \eqref{exact-sequence:torus},
$\CEff ( V_\Sigma ) ^\vee _\ZZ$ can be see as the submonoid of tuples $(d_\alpha )_{\alpha \in \Sigma ( 1 )} \in \NN^{\Sigma ( 1 )}$ such that 
\[
\sum_{\alpha \in \Sigma ( 1 ) }d_\alpha  \langle \chi , \rho_\alpha \rangle  = 0
\]
for all $ \chi \in \mathcal X^* ( U ) $. Note that this submonoid is denoted by $\NN^{\Sigma ( 1 ) }_{ ( * ) }$ in Bourqui's work \cite[Notation 5.3]{bourqui2009produit}. 

For every $\dd \in \CEff ( V_\Sigma )_\ZZ ^\vee$, 
let 
\[
\widetilde{\left ( \PP^1_\kb \right )_{B_\Sigma }^\dd }  
\]
be the inverse image of the open subset 
\[
\left ( \PP^1_\kb \right )_{B_\Sigma }^\dd \subset \Sym_{/\kb}^{\dd} \left ( \PP^{1}_\kb \right )  \simeq \prod_{\alpha \in \Sigma ( 1 )} \PP^{d_\alpha}_\kb 
\]
through the $\mathbf G_m ^{\Sigma ( 1 ) }$-torsor 
\[
\prod_{\alpha \in \Sigma ( 1 )} \left (  \mathbf A^{d_\alpha +1 }_\kb \setminus \{ 0 \}  \right )
\longrightarrow
\prod_{\alpha \in \Sigma ( 1 )} \PP^{d_\alpha}_\kb . 
\]
One of Bourqui's key results is the following proposition.
\begin{myptn}[{\cite[Proposition 5.14]{bourqui2009produit}}]\label{ptn:repres-moduli-space-toric}
For every $\dd \in \CEff ( V_\Sigma )_\ZZ ^\vee $, 
\[
\widetilde{\left ( \PP^1_\kb \right )_{B_\Sigma }^\dd } / T_{\NS} 
\]	
represents the functor $\mathbf{Hom}_\kb^{\dd} ( \PP^1 , V_\Sigma )_U $. 
\end{myptn}
\begin{mythm}
\label{thm:motivic-BMP-toric}
The normalised class 
\[
\left [
\Hom_\kb^\mdeg ( \PP^1_\kb , V_\Sigma )_U 
\right ] 
\LL_\kb^{- \, \mdeg \scdot \omega_{V_\Sigma }^{-1} }  
\]
tends to the non-zero effective element 
\[
\tau ( V_\Sigma ) =  \frac{\LL_\kb^{\dim ( V_\Sigma )}}{\left ( 1 - \LL_\kb^{-1} \right )^{\rg ( \Pic ( V_\Sigma ) )} } 
\prod_{p\in \PP^1_\kb } 
\frac{ \left [ V_\Sigma  \otimes_\kb \kappa ( p ) \right ] }{\LL_p^{\dim ( V_\Sigma ) }} \left ( 1 - \LL_p^{-1} \right )^{\rg ( \Pic ( V_\Sigma ))}
\in \widehat{\MMM_\kb}^{\dim}
\]
when $\mdeg \in \CEff ( V_\Sigma )^\vee_\ZZ  $ goes arbitrarily far away from the boundary of the dual of the effective cone of $V_\Sigma$.

Moreover the error term 
\[
\tau ( V_\Sigma ) 
- 
\left [
\Hom_\kb^\mdeg ( \PP^1_\kb , V_\Sigma )_U 
\right ] 
\LL_\kb^{- \, \mdeg \scdot \omega_{V_\Sigma }^{-1} }  
\]
has virtual dimension at most 
\[
-  \frac{1}{4} \min_{\alpha \in \Sigma ( 1 )} ( \mdeg_\alpha ) + \dim ( V_\Sigma ) .
\]
\end{mythm}

\begin{proof}
First note that since $\sum_{\alpha \in \Sigma ( 1 )} D_\alpha$ is an anticanonical divisor of $V_\Sigma$, 
the anticanonical degree of a curve of multidegree $\dd \in \CEff ( V_\Sigma )_\ZZ ^\vee$
is $| \dd | = \sum_{\alpha \in \Sigma ( 1 )} d_\alpha$. 
By \cref{ptn:repres-moduli-space-toric}, we have the relation 
\[
[ \Hom_\kb^\dd  ( \PP^1_\kb , V_\Sigma )_U ]  
= ( \LL_\kb - 1 )^{\dim ( V_\Sigma ) }  \left [ \left ( \PP^1_\kb \right )_{B_\Sigma }^\dd   \right ] 
\]
for every $\dd \in \CEff ( V_\Sigma )_\ZZ ^\vee$. 
Therefore studying the asymptotic behaviour of $[ \Hom_\kb^\dd  ( \PP^1_\kb , V_\Sigma )_U ] \LL_\kb^{-|\dd |} $ 
when $\min_\alpha ( d_\alpha ) \to \infty$
goes back to studying the one of $\left [ \left ( \PP^1_\kb \right )_{B_\Sigma }^\dd   \right ]  \LL_\kb^{-|\dd |}$. Moreover it is convenient to drop the assumption $\dd \in \CEff ( V_\Sigma )_\ZZ ^\vee$. 

Then \eqref{equation:mobius-inversion-P1-in-terms-of-series} gives
\[
\sum_{\dd \in \NN^{\Sigma ( 1 )}} \left [ \left ( \PP^1_\kb \right )^\dd_{B_\Sigma} \right ] \TT^\dd 
 = \prod_{p\in \PP^1_\kb} P_{B_\Sigma } ( \TT ) \times  \prod_{\alpha \in \Sigma ( 1 )} Z_{\PP^1_\kb}^{\text{Kapr}} ( t_\alpha ) 
\]
and we proceed as in \cite[\S 4.1]{faisant2022geometric}. 
First, recall that 
\[
Z^{\text{Kapr}}_{\PP^1_\kb} ( t ) = ( 1 - t )^{-1} \left ( 1 - \LL_\kb t \right )^{-1} .
\]
For any $A\subset \Sigma ( 1 ) $,
set
\[
Z_A ( \TT ) = \prod_{\alpha \in A } ( 1 - t_\alpha )^{-1} \prod_{\alpha \notin A} ( 1  - \LL_\kb t_\alpha )^{-1} 
\]
so that 
\begin{equation}\label{equation:decomposition-product-Kapr-P1}
\prod_{\alpha \in \Sigma ( 1 )} Z^{\text{Kapr}}_{\PP^1_\kb} ( t_\alpha ) 
= 
\prod_{\alpha \in \Sigma ( 1 )} \frac{1}{1-\LL_\kb}  \left ( \frac{1}{1- t_\alpha } - \frac{\LL_\kb}{1-\LL_\kb t_\alpha } \right )
=
\sum_{A\subset \Sigma ( 1 )} \frac{( - \LL_\kb )^{| \Sigma ( 1 ) | - | A |}}{ ( 1 - \LL_\kb )^{| \Sigma ( 1 ) |}} Z_A ( \TT ) .
\end{equation}
The coefficient of order $\dd \in \NN^{\Sigma ( 1 )} $ of 
\[
Z_A ( \TT ) 
\times 
\prod_{p\in \PP^1_\kb} P_{B_\Sigma } ( \TT ) 
\]
for a given $A\subset \Sigma ( 1 )$ is 
the finite sum 
\[
\sum_{\substack{ \ee \in \NN^{\Sigma ( 1 ) } \\ \ee \leqslant \dd }} \mu_\Sigma ( \ee ) \LL_\kb^{ \sum_{\alpha \in \Sigma ( 1) \setminus A } d_\alpha - e_\alpha  }   
=
 \LL_\kb^{  |  \dd_{\Sigma ( 1 ) \setminus A}  | } 
\sum_{\substack{ \ee \in \NN^{\Sigma ( 1 ) } \\ \ee \leqslant \mdeg }}   \mu_\Sigma ( \ee ) \LL_\kb^{- \left | \ee_{\Sigma ( 1 ) \setminus A } \right |}
\]
where we use again the notation of \cref{lemma:controlling-error-terms-convolution} for the restriction of $\ee \in \NN^{\Sigma ( 1 )}$ to a subset of $\Sigma ( 1 )$. 
Normalizing by $ \LL_\kb^{|\dd |}$ 
and writing $\ee_{\Sigma ( 1 ) \setminus A } = \ee - \ee_A $
one gets the sum 
\begin{equation}\label{symbol:normalised-coeff-ZZZ-A}
\sum_{\substack{ \ee \in \NN^{\Sigma ( 1 ) } \\ \ee \leqslant \dd }}   \mu_\Sigma ( \ee ) \LL_\kb^{- | \ee|} \LL_\kb^{| \ee_A | - | \dd_A |} .
\end{equation}
If $A = \varnothing $, this is exactly the $\dd$-th partial sum
\[
\sum_{\substack{ \ee \in \NN^{\Sigma ( 1 ) } \\ \ee \leqslant \dd }}   \mu_\Sigma ( \ee ) \LL_\kb^{- | \ee |} 
\]
and when $\min_{\alpha \in \Sigma ( 1 )} d_\alpha $ goes to infinity,
we know by 
\cref{ptn:convergence-Euler-product-toric-varieties}
that
this sum converges to 
\[
 \sum_{ \ee \in \NN^{\Sigma ( 1 ) } }   \mu_\Sigma ( \ee ) \LL_\kb^{- | \ee | } = \prod_{p\in \PP^1_\kb }  \frac{ \left [ V_\Sigma \otimes \kappa ( p ) \right ] }{\LL_p^n} \left ( 1 - \LL_p^{-1} \right )^r 
\tag{by \cref{remark:local-factor-Euler-product-toric-varieties}}
\]
in $\widehat{\mathscr M_\kb}^{\dim }$.
By the proof of \cite[Lemma 4.2.5]{bilu2020zeta}
we have $\dim (  \mu_\Sigma ( \ee ) \LL_\kb^{- | \ee |}  ) \leqslant - \frac{1}{2} | \ee |$
for any $\ee \in \NN^{\Sigma ( 1 ) } $.
Thus the error term
\[
\sum_{\substack{ \ee \in \NN^{\Sigma ( 1 ) } \\ \ee \nleqslant \dd }}   \mu_\Sigma ( \ee ) \LL_\kb^{- | \ee |} 
\]
has virtual dimension at most $- \frac{1}{2} \min_{\alpha \in \Sigma ( 1 )} ( d_\alpha +1 )$.
\\
If $A \neq \varnothing $, 
then by Lemma
\ref{lemma:controlling-error-terms-convolution} (taking $\mm = \ee$, $\mathfrak{c}_\mm = \mu_\Sigma ( \ee ) \LL_\kb^{-\ee}$, $a = 1/2$ and $b=0$),
the sum (\ref{symbol:normalised-coeff-ZZZ-A}) has virtual dimension at most
\[
-  \frac{1}{4} \min_{\alpha \in A} ( d_\alpha )
 \]
hence it becomes negligible in comparison with the term given by $A = \varnothing$
as $\min ( d_\alpha ) \to \infty $.

Finally, putting everything together, one concludes that the normalised class
\[
\left [\Hom ^\mdeg_\kb ( \PP^1_\kb, V_\Sigma )_U \right ] \LL_\kb^{- \, \mdeg \scdot  \omega_{V_\Sigma }^{-1} }  
\]
tends to 
\[
\frac{\LL_\kb^{\dim ( V_\Sigma )}}{\left ( 1 - \LL_\kb^{-1} \right )^{\rg ( \Pic ( V_\Sigma ) )} } \prod_{p\in \PP^1_\kb } \frac{ \left [ V_\Sigma \otimes \kappa ( p ) \right ] }{\LL_p^{\dim ( V_\Sigma ) }} \left ( 1 - \LL_p^{-1} \right )^{\rg ( \Pic ( V_\Sigma ))}
\]
in $\widehat{\mathscr M_\kb}^{\dim }$
when $\mdeg \in \CEff ( V_\Sigma )^\vee_\ZZ $ goes arbitrarily far away from the boundary $\partial \CEff ( V_\Sigma )^\vee$, with error term of virtual dimension bounded by
$
-  \frac{1}{4} \min_{\alpha \in \Sigma ( 1 )} ( \mdeg_\alpha ) + \dim ( V_\Sigma ) .
$
\end{proof}

\index{motivic height zeta function!of smooth split projective toric varieties}
\begin{mycor}
Let 
\[
\FFF ( \TT  ) = \sum_{\mdeg \in \CEff ( V_\Sigma )^\vee_\ZZ} 
\left [ \Hom_\kb^\mdeg  ( \PP^1_\kb , V_\Sigma )_U \right ]  \TT^\mdeg \prod_{\alpha \in \Sigma ( 1 ) } ( 1 - \LL_\kb t_\alpha )  .
\]
Then $\FFF ( \TT  ) $ converges at $t_\alpha = \LL_\kb^{-1}$ to $\tau_{\PP^1_\kb} ( V_\Sigma ) $. 
\end{mycor}

\begin{proof}
Let $\mathfrak b_\dd $ be the coefficient of multidegree $\dd$ of $\FFF ( \TT ) $. 
Since 
\[
\FFF ( \TT )  \prod_{\alpha \in \Sigma ( 1 ) } ( 1 - \LL_\kb t_\alpha )^{-1} = \FFF ( \TT )  \sum_{\dd \in \NN^{\Sigma ( 1 )}} \LL_\kb^{|\dd |}\TT^\dd 
\]
we have the relation 
\[
\left [ \Hom_\kb^\mdeg ( \PP^1_k , V_\Sigma )_U \right ]  
=
\sum_{\dd + \dd ' = \mdeg} \mathfrak b_\dd \LL_\kb^{ | \dd ' |} 
\]
for every $\mdeg \in \CEff ( V_\Sigma )^\vee_\ZZ$,
which becomes after normalisation 
\[
\left [ \Hom_\kb^\mdeg ( \PP^1_k , V_\Sigma )_U \right ]
\LL_\kb^{-\, \mdeg \scdot \omega_{V_\Sigma}^{-1}  } = \sum_{\ee \leqslant \mdeg } \mathfrak{b}_{\ee } \LL_\kb^{-| \ee |} . 
\]
This is exactly the $\mdeg$-th partial sum of the series $\FFF ( \LL_\kb^{-1} ) $. 
Since $\left [ \Hom_\kb^\mdeg ( \PP^1_k , V_\Sigma )_U \right ] \LL_\kb^{-\, \mdeg \scdot \omega_{V_\Sigma}^{-1}  }$ converges to $\tau_{\PP^1_\kb} ( V_\Sigma )$ when $\min ( \mdeg_\alpha ) $ tends to infinity, the claim follows. 
\end{proof}


\subsection{Equidistribution}
\label{section:equidistribution-toric}

In the remainder of this section, we prove equidistribution of rational curves on smooth split projective toric varieties. 

\begin{mythm}
\label{thm-equidistribution-toric}
\index{equidistribution of curves!on smooth split projective toric varieties}
Let 
\[
\SSS = \coprod_{p\in | \SSS |} \SSS_p
\]
be a zero-dimensional subscheme of $\PP^1_\kb$,
$(m_p)_{p\in | \SSS |}$ non-negative integers such that 
\[
\ell ( \SSS_p ) = ( m_p +1 ) [ \kappa ( p ) : \kb ]
\]
for all $p\in |\SSS |$,
and $W$ a constructible subset of $\Hom_\kb ( \SSS , V_\Sigma ) \simeq \prod_{p\in |\SSS |} \Gr_{m_p} ( V_\Sigma \otimes \kappa ( p ) )$. 

Then, when ${\mdeg \in \CEff ( V_\Sigma )_\ZZ ^\vee }$
goes arbitrarily far away from the boundary of the dual of the effective cone, the normalised class
\[
\left [ \Hom_\kb^\mdeg ( \PP^1_\kb , V_\Sigma \mid W ) _U \right ] \LL_\kb^{- \, \mdeg \, \cdot \, \omega_{V_\Sigma}^{-1} }   \in \MMM_\kb
\]
tends to the non-zero effective element
\begin{align*}
\tau ( V_\Sigma \mid W ) 
= &
\frac{\LL_\kb^{\dim ( V_\Sigma ) }}{\left ( 1  -\LL_\kb^{-1} \right )^{\rg ( \Pic ( V_\Sigma ) ) }} 
\prod_{p \notin | \SSS | } \left ( 1  - \LL_p^{-1} \right )^{\rg ( \Pic ( V_\Sigma ) ) } \frac{\left [ V_\Sigma \otimes \kappa ( p ) \right ]}{\LL_p^{\dim \left ( V_\Sigma \right )}} \\
& 
\times [ W ] 
\prod_{p\in | \SSS | } \left ( 1  - \LL_p^{-1} \right )^{\rg ( \Pic ( V_\Sigma ) ) } \LL_p^{- (m_p +1 ) \dim \left ( V_\Sigma \right )  }
\end{align*}
in the completion $\widehat{\mathscr M_\kb}^{\dim}$.

Moreover,
the error term
\[
\tau ( V_\Sigma \mid W ) 
 - 
\left [ \Hom^\mdeg_\kb ( \PP^1_\kb , V_\Sigma \mid W )_U \right ] \LL_\kb^{- \, \mdeg \, \cdot \, \omega_{V_\Sigma}^{-1} } 
\in \widehat{\mathscr M_\kb}^{\dim}
\]
has virtual dimension smaller than 
\[
 - \frac{1}{4} \min_{\alpha \in \Sigma ( 1 ) } ( \mdeg_\alpha ) + \ell ( \SSS ) + ( 1 -  \ell ( \SSS )  ) ( \dim ( V_\Sigma ) - 1 ) + \dim ( W ) 
\]
for all ${\mdeg \in \CEff ( V_\Sigma )_\ZZ ^\vee }$. 
\end{mythm} 

\begin{mycor}
For any non-zero multidegree $\mdeg \in \CEff ( V_\Sigma )_\ZZ^\vee$
such that 
\[
\min_{\alpha \in \Sigma ( 1 )} (  \mdeg_\alpha ) \geqslant 8 \ell ( \SSS ) - 4,
\]
the moduli space  
$\Hom^\mdeg_\kb ( \PP^1_\kb , V_\Sigma \mid W )_U $
has dimension
\[
\mdeg \cdot \omega_{V_\Sigma}^{-1} + \dim ( V_\Sigma ) ( 1 - \ell ( \SSS ) ) + \dim ( W ) 
\]
as expected. 
\end{mycor}

\begin{myremark}
	The upper bound on the dimension of the error term we give can be made uniform in the set $W$ of conditions, since the dimension of $W$ is bounded by $\ell ( \SSS ) \dim ( V_\Sigma )$. 
\end{myremark}

The remainder of this section is devoted to the proof of \cref{thm-equidistribution-toric}.
We see $\mathbf A^2_\kb \setminus \{ 0 \}$
as the universal $\GG_m$-torsor 
\[
\mathbf A^2_\kb \setminus \{ 0 \} \to \PP^1_\kb
.
\]
Given a cocharacter $\chi : \GG_m \to \GG_m^{\Sigma ( 1 )}$,
or equivalently a tuple $\dd = (d_\alpha ) \in \ZZ^{\Sigma ( 1 )}$ (we will switch freely between both notations),
we consider the functor from $\kb$-schemes to sets
\[
\HHom^{\chi } \left ( \mathbf A^2_\kb \setminus \{ 0 \} , \mathbf A^{\Sigma ( 1 )}_\kb \right )
: 
S \rightsquigarrow \Hom_S^{\chi } \left ( \mathbf A^2_S \setminus \{ 0 \} , \mathbf A^{\Sigma ( 1 )}_S \right ),
\]
of $\chi$-equivariant morphisms,
as well as its restriction
\[
\HHom^{\chi } \left ( \mathbf A^2_\kb \setminus \{ 0 \} , \mathbf A^{\Sigma ( 1 )}_\kb \right )^* :
S \rightsquigarrow \Hom_S^{\chi } \left ( \mathbf A^2_S \setminus \{ 0 \} , \mathbf A^{\Sigma ( 1 )}_S \right ) ^* ,
\]
to $\chi$-equivariant morphisms \textit{with no trivial coordinate}
and its second restriction
\[
\HHom^{\chi } \left ( \mathbf A^2_\kb \setminus \{ 0 \} , \mathcal T_\Sigma \right )^* :
S \rightsquigarrow \Hom_S^{\chi } \left ( \mathbf A^2_S \setminus \{ 0 \} , \mathcal T_{\Sigma } \times_\kb S \right )^*
\]
to \textit{non-degenerate} $\chi$-equivariant morphisms with no trivial coordinate.
These functors are represented 
respectively by 
the product 
\[
\prod_{\alpha \in \Sigma ( 1 )} \mathbf A_\kb ^{d_\alpha +1 }  ,
\]
its restriction
\[
\prod_{\alpha \in \Sigma ( 1 )} 
\left ( 
\mathbf A_\kb ^{d_\alpha +1 } \setminus \{ 0 \} 
\right ) ,
\]
and
the open subset 
\[
\widetilde{(\PP^1_\kb)^\dd_{B_\Sigma}} \subset \prod_{\alpha \in \Sigma ( 1 )} 
\left ( 
\mathbf A_\kb ^{d_\alpha +1 } \setminus \{ 0 \}  
\right ) 
\]
(defined \cpageref{ptn:repres-moduli-space-toric} just before \cref{ptn:repres-moduli-space-toric}).

If $\chi $ lies in $\CEff ( V_\Sigma )_\ZZ^\vee$, composition by $\pi : \mathcal T_\Sigma \to V_\Sigma$
provides a map of functors 
\[
\pi_* : \HHom^{\chi } \left ( \mathbf A^2_\kb \setminus \{ 0 \} , \mathcal T_\Sigma \right )^* \longrightarrow \HHom^{\dd}_\kb \left ( \PP^1_\kb , V_\Sigma \right )
\]
and we recover the $T_{\NS}$-torsor 
\[
\widetilde{(\PP^1_\kb)^\dd_{B_\Sigma}} 
\longrightarrow
\Hom^{\dd}_\kb \left ( \PP^1_\kb , V_\Sigma \right )_U
\]
of \cref{ptn:repres-moduli-space-toric}.

\subsubsection{Restricting to $\SSS$}
Let $\iota : \SSS \hookrightarrow \PP^1_\kb$ be a zero-dimensional subscheme of $\PP^1_\kb$.
In what follows, we will work with the restriction to $\SSS$ of the universal torsor $\mathbf A^2 \setminus \{ 0 \} \to \PP^1$, defined over any base-scheme $S$ by the following Cartesian square. 
\[
\begin{tikzcd}
(\mathbf A_S^{2}  \setminus \{ 0 \} )_{|\SSS}  \arrow[d,"\pr_\SSS"] \arrow[r] \arrow[dr, phantom, "\ulcorner", very near start] &  \mathbf A_S^{2}  \setminus \{ 0 \}  \arrow[d]  \\
\SSS \times_k S \arrow[r ,"\iota_S", hook]  & \PP^1_S
\end{tikzcd}
\]
For any cocharacter $\chi \in \mathcal X_* ( \GG_m^{\Sigma ( 1 )} ) $,
we consider the set of $\chi$-equivariant $S$-morphisms
\[
\Hom_S^\chi \left ( (\mathbf A_S^{2}  \setminus \{ 0 \} )_{|\SSS} , \mathbf A^{\Sigma ( 1 )}_S \right ) 
\]
as well as its subset of morphisms landing in the universal torsor $\mathcal T_\Sigma \subset \mathbf A_\kb^{\Sigma ( 1 )}$
\[
\Hom_S^\chi \left ( (\mathbf A_S^{2}  \setminus \{ 0 \} )_{|\SSS} , \mathcal T_{\Sigma , S} \right ). 
\]
The groups $\GG_m^{\Sigma ( 1 )} ( S ) $ and 
$T_{\NS} ( S ) $ act on these sets \textit{via} their action on the target.

\begin{mylemma}	
\label{lemma:representation-restricted-equivariant-morphisms-as-jets}
For all $\chi \in \mathcal X_* ( \GG_m^{\Sigma ( 1 )})$, 
	the functors 
	\[
	S \rightsquigarrow \Hom_S^\chi \left ( (\mathbf A_S^{2}  \setminus \{ 0 \} )_{|\SSS} , \mathbf A^{\Sigma ( 1 )}_S \right )
	\]
and	
\[
	S \rightsquigarrow \Hom_S^\chi ( (\mathbf A_S^{2}  \setminus \{ 0 \} )_{|\SSS} , \mathcal T_{\Sigma, S} )
	\]
	are respectively represented by the finite products
	\[
	\prod_{p\in |\SSS|} \Gr_{m_p} \left ( \mathbf A_{\kappa ( p )}^{\Sigma ( 1 )} \right )
	\]
	and  
	\[
	\prod_{p\in |\SSS|} \Gr_{m_p} (  \mathcal T_\Sigma \times_\kb \kappa ( p ) )
	\]
	of jet schemes. 
\end{mylemma}
\begin{proof}
	Since the restriction of $\mathbf A^2_\kb \setminus \{ 0 \} \to \PP^1_\kb$
	to $\SSS$ is a trivial bundle, 
	we may fix a section $s:\SSS \to ( \mathbf A^2_\kb \setminus \{ 0 \} )_{|\SSS}$.
	Then composition by $s$ induces a map 
\[
\begin{array}{rll}
\Hom_S^\chi \left ( (\mathbf A_S^{2}  \setminus \{ 0 \} )_{|\SSS} , \mathbf A^{\Sigma ( 1 )}_S \right ) & \longrightarrow & \mathbf A^{\Sigma ( 1 )}_\kb ( \SSS \times_\kb S ) \\
f & \longmapsto &  f \circ ( s , \mathrm{id}_S ) 
\end{array}
\]	
which is functorial in $S$. 
	
	Now remark that an element of $\Hom_S^\chi \left ( (\mathbf A_S^{2}  \setminus \{ 0 \} )_{|\SSS} , \mathbf A^{\Sigma ( 1 )}_S \right )$ is entirely determined by its restriction to the image of $s$: indeed,
	for all $y \in ( \mathbf A^2_S \setminus \{ 0 \})_{|\SSS} ( S )$, one has the relation
	\[
	f ( y ) = \chi \left ( y  \cdot ( s \circ \pr_\SSS ( y ))^{-1} \right ) f ( s \circ \pr_\SSS ( y ) )
	\]
	where $ y  \cdot ( s \circ \pr_\SSS ( y ))^{-1} \in \GG_{m,\SSS} ( S ) $ is given by the $\GG_{m,\SSS}$-torsor structure. 
	An $\SSS \times_\kb S$-point of $\mathbf A^{\Sigma ( 1 )}_S$ is the datum of such a restriction, hence it provides a unique $\chi$-equivariant morphism. 
	By definition of Greenberg schemes, the conclusion follows.
\end{proof}

\index{universal torsor of a toric variety}
Composition by $s$ and $\pi : \mathcal T_\Sigma \to V_\Sigma$ provides a map (functorial in $S$)
\[
\begin{array}{rll}
\Hom_S^\chi \left ( (\mathbf A_S^{2}  \setminus \{ 0 \} )_{|\SSS} , \mathcal T_{\Sigma, S} \right ) & \longrightarrow & V_\Sigma ( \SSS \times_\kb S ) \\
f & \longmapsto &  ( \pi , \mathrm{id}_S )  \circ f \circ ( s , \mathrm{id}_S ) . 
\end{array}
\]	
Two $\chi$-equivariant morphisms $\tilde \varphi , \tilde \varphi ' : ( \mathbf A_S^{2}  \setminus \{ 0 \} )_{|\SSS} \to  \mathbf A^{\Sigma ( 1 )}_S$ induce the same $(\SSS \times_\kb S )$-point of $V_\Sigma$ if and only if 
	there is an element $a \in T_{\NS} ( \SSS \times_\kb S ) $ such that $ \tilde \varphi  =  a \cdot \tilde \varphi '  $. 		
	This map defines a $T_{\NS  , \SSS}$-torsor over $\Hom ( \SSS , V_\Sigma )$.  

\begin{mydef}
For any constructible subset $\widetilde{W} \subset \Hom_S^\chi \left ( (\mathbf A_S^{2}  \setminus \{ 0 \} )_{|\SSS} , \mathbf A^{\Sigma ( 1 )}_S \right )$,
we denote by 
\[
\Hom_S^\chi \left ( \mathbf A_S^{2}  \setminus \{ 0 \} , \mathbf A_S^{\Sigma ( 1 )} \mid \widetilde{W} \right )^{(*)}
\]
the subset of $\Hom_S^\chi \left ( \mathbf A_S^{2}  \setminus \{ 0 \} , \mathbf A_S^{\Sigma ( 1 )} \right )^{(*)}$ of morphisms
whose restriction to $ ( \mathbf A^2_S \setminus \{ 0 \} )_{|\SSS} $ belongs to $\widetilde W$
(in the sequel the exponent $\, ^{(*)}$ will be a convenient notation to say that one can restrict to morphisms having no trivial coordinate).

We will say that a $\chi$-equivariant $S$-morphism
\textit{is non-degenerate} above $\SSS$
if its pull-back to $( \mathbf A^2_S \setminus \{ 0 \}  )_{|\SSS}$ has image in $\mathcal T_{\Sigma , S } \subset \mathbf A^{\Sigma ( 1 )}_S$. 

This defines subfunctors of $\HHom^\chi \left ( \mathbf A_\kb^{2}  \setminus \{ 0 \} , \mathbf A^{\Sigma ( 1 )}_\kb \right )$.
\end{mydef}

\subsubsection{Euclidean division}
Fixing coordinates on $\mathbf A_\kb^2$, 
we can see $\chi$-equivariant morphisms $ \mathbf A_\kb^{2}  \setminus \{ 0 \} \to  \mathbf A_\kb^{\Sigma ( 1 )} $ 
as $\Sigma ( 1 ) $-tuples of homogeneous polynomials in two indeterminates $t_0$ and $t_1$. 
Let us temporarily choose a generator $\varpi$ of the ideal defining $\SSS$ in $\PP^1_\kb$, that is to say a non-trivial homogeneous polynomial of degree $\ell ( \SSS )$ in two indeterminates. 
Up to changing coordinates on $\PP^1_\kb$ we may
assume furthermore that $[ 0 : 1 ]$ does not belong to $\SSS$ (that is to say, $t_0$ does not divide $\varpi$).

Then, the Euclidean division of a polynomial $P ( t )$ of degree at most $d$ by $\varpi ( t ) = \varpi ( 1 , t )$, is the unique decomposition of the form
\[
P ( t )= Q ( t ) \varpi ( t )  + R ( t ) 
\]
where $Q(t)$ is of degree at most $ d-\ell ( \SSS )$ and $R$ of degree strictly smaller than $\ell ( \SSS )$.
This provides a Euclidean division of $P ( t_0 , t_1 ) = t_0 ^d P ( t_1 / t_0 )$ of the form 
\begin{align*}
P ( t_0 , t_1 )
& = t_0^d Q ( t_1 / t_0 ) \varpi ( t_1 / t_0 )  + t_0^d R ( t_1 / t_0 ) \\
& = Q ( t_0 , t_1 ) \varpi ( t_0 , t_1 ) + R ( t_0 , t_1 )
\end{align*}
where 
\begin{align*}
	Q ( t_0 , t_1 ) & = t_0^{d - \ell ( \SSS ) } Q ( t_1 / t_0 ) , \\ 
	\varpi ( t_0 , t_1 ) & = t_0^{\ell ( \SSS ) } \omega ( t_1 / t_0 ), \\
\text{and } R(t_0 , t_1 ) & = t_0^d R ( t_1 / t_0 ).
\end{align*} 
Note that $Q ( t_0 , t_1 ) \varpi ( t_0 , t_1 )$ 
and $R ( t_0 , t_1 )$
are homogeneous polynomials of degree $d$ 
and that they do not depend on the choice of $\varpi$. 
The first one uniquely defines an element of 
\[
\Hom^d \left ( \mathbf A_\kb ^{2}  \setminus \{ 0 \} , \mathbf A_\kb^1 \mid \res_\SSS = 0  \right ) 
\]
while the second one uniquely defines an element of 
\[
\Hom^d \left ( \left ( \mathbf A_\kb ^{2}  \setminus \{ 0 \} \right )_{|\SSS}, \mathbf A_\kb^1 \right ),
\]
since in that case the $\kb$-vector space
\[
\kb [ t ] / ( \varpi ( 1 , t )) \simeq \prod_{p\in | \SSS |} \mathbf A_{\kappa ( p )}^{\Sigma ( 1 )} (  \kappa ( p ) [t] / (  t^{m_p +1} )   ) \simeq \prod_{p\in | \SSS |} \Gr_{m_p} \left ( \mathbf A_{\kappa ( p )}^{\Sigma ( 1 )} \right ) ( \kappa ( p ) )
\]
provides a concrete incarnation of such a space of morphisms. 
Remember as well that taking the restriction to $\SSS$ is a linear operation.

One can perform this Euclidean division simultaneously for all the coordinates of an element of $\Hom^\chi \left ( \mathbf A_\kb ^{2}  \setminus \{ 0 \} , \mathbf A_\kb ^{\Sigma ( 1 )} \right ) $.
Recall that 
we fixed a section $s:\SSS \to ( \mathbf A^2 \setminus \{ 0 \} )_{|\SSS}$
in \cref{lemma:representation-restricted-equivariant-morphisms-as-jets}
so that 
\[
\prod_{p\in |\SSS|} \Gr_{m_p} \left ( \mathbf A_{\kappa ( p )}^{\Sigma ( 1 )} \right )
\]
represents the functor 
$S 
\rightsquigarrow 
\Hom_S^\chi \left ( (\mathbf A_S^{2}  \setminus \{ 0 \} )_{|\SSS} , \mathbf A^{\Sigma ( 1 )}_S \right )
$.
We can see elements of this product of arc spaces 
as tuples
$(r_\alpha ( t ))_{\alpha \in \Sigma ( 1 )}$
of polynomials of degree at most $\ell ( \SSS ) - 1 $. 
From these remarks, 
we deduce the following lemma.
\begin{mylemma} 
\label{lemma:euclidean-division-in-univ-torsor}
\label{lemma:restriction-toric-torsor-locally-trivial}
For every $\chi \in \NN^{\Sigma ( 1 )} \subset \mathcal X_* ( \GG_m^{\Sigma ( 1 )} ) \simeq \ZZ^{\Sigma ( 1 )}$,
Euclidean decomposition 
corresponds to the exact sequence of vector spaces over $\kb$
\[
\begin{tikzcd}
   0  \rar & \ker ( \res_\SSS )  \rar &  \Hom^\chi \left ( \mathbf A_\kb ^{2}  \setminus \{ 0 \} , \mathbf A_\kb ^{\Sigma ( 1 )} \right ) \ar[out=-30, in=150,"\res_\SSS",swap]{dl} \\ 
  & \Hom^\chi \left ( (\mathbf A_\kb^{2}  \setminus \{ 0 \} )_{|\SSS} , \mathbf A^{\Sigma ( 1 )}_\kb \right ) \rar & 0   
\end{tikzcd}
\]
and $\res_\SSS$ is a piecewise trivial fibration. 
Moreover the $\alpha$-th coordinate of $\res_\SSS$ is surjective if $\chi_\alpha \geqslant \ell ( \SSS ) - 1$, injective if $\chi_\alpha \leqslant \ell ( \SSS ) - 1$,
and non-empty fibres of $\res_\SSS$
have dimension 
\[
\sum_{\alpha \in \Sigma ( 1 )} \min ( 0 , \chi_\alpha - \ell ( \SSS ) + 1) .
\]

Hence for any constructible subset $\widetilde W \subset \prod_{p\in |\SSS|} \Gr_{m_p} \left ( \mathbf A_{\kappa ( p )}^{\Sigma ( 1 )} \right ) $
we have the dimensional upper bound
\begin{equation}
\label{equation:toric-dim-bound-under-constraints}
	\dim_\kb \left (  \Hom^\chi \left ( \mathbf A_\kb^{2}  \setminus \{ 0 \} , \mathbf A_\kb^{\Sigma ( 1 )} \mid \widetilde{W} \right ) \right ) \leqslant \dim \left ( \widetilde W \right ) 
+ \sum_{\alpha \in \Sigma ( 1 )} \min ( 0 , \chi_\alpha - \ell ( \SSS ) + 1) .
\end{equation}

Assume that $\chi \geqslant \ell ( \SSS ) - 1$. Then
for all $g\in \GG_{m}^{\Sigma ( 1 )} ( \SSS )$ and constructible subset $\widetilde W \subset \prod_{p\in |\SSS|} \Gr_{m_p} \left ( \mathbf A_{\kappa ( p )}^{\Sigma ( 1 )} \right ) $ there is a commutative diagram
\[
\begin{tikzcd}
	\Hom^\chi \left ( \mathbf A_\kb^{2}  \setminus \{ 0 \} , \mathbf A_\kb^{\Sigma ( 1 )} \mid \widetilde{W} \right )^{(*)} \rar["\res_\SSS"] \dar["\tau_g",swap] 
	& \Hom^\chi \left ( (\mathbf A_\kb^{2}  \setminus \{ 0 \} )_{|\SSS} , \mathbf A^{\Sigma ( 1 )}_\kb \right ) \dar["g \scdot",swap]   \\
	   	\Hom^\chi \left ( \mathbf A_\kb^{2}  \setminus \{ 0 \} , \mathbf A_\kb^{\Sigma ( 1 )} \mid g \cdot \widetilde{W} \right )^{(*)} \rar["\res_\SSS"]  &  \Hom^\chi \left ( (\mathbf A_\kb^{2}  \setminus \{ 0 \} )_{|\SSS} , \mathbf A^{\Sigma ( 1 )}_\kb \right ) \\\ 
\end{tikzcd}
\]
in which the vertical arrows are isomorphisms.
The first one, $\tau_g$, sends a morphism of Euclidean decomposition
\[
(q_\alpha ,   r_\alpha )_{\alpha \in \Sigma ( 1 )}
\] 
to 
\[
( q_\alpha  ,   g_\alpha \cdot r_\alpha )_{\alpha \in \Sigma ( 1 )}.
\]
\end{mylemma}

Up to a multiplicative constant, rational curves of degree $d\in \NN$ in $\PP^n_\kb$
are given by $(n+1)$-tuples of degree $d$ homogeneous polynomials with no common factor. Conversely, the parameter space of $(n+1)$-tuples of degree $d$ homogeneous polynomials admits a decomposition into disjoint subspaces corresponding to the degree of the common factor. 
The Möbius inversion we performed for non-constrained curves generalises this remark.

\begin{mydef}\label{def:degeneracy-avoiding-S-and-restriction-of-curly-P1}
We will say that a zero-cycle \textit{avoids $\SSS$} if 
it has support outside $|\SSS|$. 
This terminology naturally extends to $\chi$-equivariant morphisms
by considering their image through the $\GG_m^{\Sigma ( 1 )}$-torsor
\[
\Hom^\chi \left ( \mathbf A^2_\kb \setminus \{ 0 \}  , \mathbf A^{\Sigma ( 1 )} \right ) ^* \to \prod_{\alpha \in \Sigma ( 1 )}  \PP^{\chi_\alpha}_\kb  
 \simeq  \Sym_{/\kb}^\chi ( \PP^1_\kb ) . \]
\end{mydef}

\subsubsection{The motivic Möbius inversion in details}

Equivariant morphisms with image contained in $\mathcal T_\Sigma$
are the morphisms  ${\mathbf A^2_\kb \setminus \{ 0 \} \to \mathbf A^{\Sigma ( 1 )}_\kb}$ with no forbidden zeros, what we also call having no degeneracies.
Our goal now is to adapt the Möbius inversion technique 
to the context 
of constrained curves. 
More precisely, we provide a relation 
between the classes of equivariant morphisms 
and the ones obtained by adding degeneracies to non-degenerate morphisms. 
This relation relies heavily on the following piecewise identification. Then we specialise this relation to morphisms with constraints. 
Finally, we approximate it and conclude. 

\medskip

\begin{myidentification}
\label{construction:piecewise-id-hom-zero}
Let $G = \GG_m^{\Sigma ( 1 )}$.
We stratify $\Hom^\chi ( \mathbf A^2\setminus \{ 0 \} , \mathbf A^{\Sigma ( 1 )} )^*$
	with respect to the vanishing order at $\infty$.
	Each stratum can be viewed as the subspace of tuples of homogeneous polynomials of the form
	\[
	P_\alpha ( T_0 , T_1 ) = a_\alpha T_1^{v_\infty ( P_\alpha )} Q_\alpha ( T_0 , T_1 )
	\]
	where $a_\alpha \in \kb^\times$, $v_\infty ( P_\alpha )$ is the vanishing order of $P_\alpha$ at $\infty$, and the coefficient of $T_1^{\chi_\alpha - v_\infty ( P_\alpha ) }$
	in $Q_\alpha$ is $1$, hence $Q_\alpha$ can be identified with a zero-cycle avoiding $\infty$.	
	In this decomposition, $ a_\alpha $ is the leading coefficient of $P_\alpha ( 1 , T )$. 
	Then, each stratum can be identified with ${G \times \Sym_{/ \kb }^{\chi - v_\infty } ( \PP^1_\kb \setminus \{ \infty \} ) }$
	by sending $(P_\alpha ) $ to $((a_\alpha ) , (Q_\alpha ))$.
	This piecewise correspondence identifies 
	\[
	\Hom^\chi ( \mathbf A_\kb^2\setminus \{ 0 \} , \mathbf A_\kb^{\Sigma ( 1 )}  )^*
	\]
	with 
	\[
	G 
	\times_\kb 
	\Sym_{ / \kb }^\chi
\left 
(  \PP^1_\kb   
\right )
\simeq 
\Sym_{ / G }^\chi
\left 
(  \PP^1_G  
\right )  
\]
	and 
	\[
	\Hom^\chi ( \mathbf A_\kb^2\setminus \{ 0 \} , \mathcal T_\Sigma )
	\] with 
	\[
	G 
	\times_\kb 
	\Sym_{\PP^1_\kb / \kb }^\chi
\left 
( \mathbf 1_{A ( B_\Sigma )} \PP^1_\kb   
\right )
\simeq 
\Sym_{\PP^1_G / G }^\chi
\left 
( \mathbf 1_{A ( B_\Sigma )} \PP^1_G  
\right ) 
\] 
(recall the definition of $A ( B_\Sigma ) $ given \cpageref{definition-ABsigma}).

In particular,
tuples such that $a_\alpha = 1 $
for every $\alpha \in \Sigma ( 1 )$ 
deserve to be called $\textit{unitary}$
and are identified with their tuple of zero-divisors
living in 
\[
\Sym_{\PP^1_\kb / \kb }^\chi
\left 
( \PP^1_\kb   
\right )
\simeq \{ 1_G \} \times \Sym_{\PP^1_\kb / \kb }^\chi
\left 
( \PP^1_\kb   
\right )
\simeq 
\Sym_{\PP^1_G / G }^\chi
\left 
( \PP^1_{\{1_G\}}  
\right ) . 
\]
\end{myidentification}

The proof of the following lemma is left to the reader. 
\begin{mylemma}\label{lemma:piecewise-iso-induce-same-k0}
	Let $S$ be a scheme, and $X$, $Y$ be two $S$-varieties. Assume that there exists a piecewise isomorphism $f=(f_j : X_j \to Y_j )$ between $X$ and $Y$,
	that is to say, a finite number of 
	pairwise disjoint locally closed subsets $X_j$ and $Y_j$, respectively of $X$ and $Y$, 
	such that $X=\sqcup_j X_j$ and $Y= \sqcup_j Y_j$,
	together with isomorphisms $f_j : X_j \to Y_j$.
	
	Then 
	\[
	\begin{array}{rll}
		[ Z \to Y ] & \longmapsto &  [ f^* ( Z \to Y ) ] =  \sum_j \left [ f_j^* \left ( Z \times_Y Y_j  \to Y_j \right ) \right ]
	\end{array}
	\]
	defines a ring isomorphism 	$\KVar{Y} \simeq \KVar{X} $. 
\end{mylemma}


\begin{mynotation}
Let $W$ be a constructible subset of $\Hom ( \SSS , V_\Sigma )$.
Its preimage through the map 
\[
\Hom^\chi \left ( (\mathbf A_S^{2}  \setminus \{ 0 \} )_{|\SSS} , \mathcal T_{\Sigma, S} \right )  \longrightarrow \Hom ( \SSS , V_\Sigma ) 
\]
will be written $\widetilde W$.
It is a Zariski-locally trivial $T_{\NS , \SSS }$-torsor above $W$. 
\end{mynotation}

For every $d\in \NN$,
we fix  a section of 
	\[
	\Hom^d \left ( \mathbf A^2_\kb \setminus \{ 0 \} , \mathbf A^1_\kb \right ) 
	\overset{\res_\SSS}{\longrightarrow} \im ( {\res_\SSS} ) \subset \Hom^d \left  ( \left ( \mathbf A^2_\kb \setminus \{ 0 \} \right )_{|\SSS} , \mathbf A^1_\kb \right )
	\]
	so that we are able to see reduction modulo $\SSS$ as morphisms in $\Hom^d ( \mathbf A_\kb^2 \setminus \{ 0 \} , \mathbf A^1_\kb )$.

\begin{myremark}
	Up to replacing $\SSS$ by $\SSS \cup \{ [1:0] \}$
	and $W$ by $W\times V_\Sigma$,
	we can always assume that 
	$[1:0]=\infty$ is a closed point of $\SSS$.
\end{myremark}

\begin{mynotation}
	We will freely use the convenient notation
\[
\overline{g}^{-1} = \left ( \overline{g_\alpha}^{-1} \right )_{\alpha \in \Sigma ( 1 )}
\]
for the inverse of the restriction to $\SSS$ 
of any $\chi$-equivariant morphism $g : \mathbf A^2_\kb \setminus \{ 0 \} \to \mathbf A^{\Sigma ( 1 )}_\kb$ avoiding $|\SSS|$.
\end{mynotation}
 
 \begin{mydef} \label{def-phi-chi-chi-prime}
For any $\chi , \chi '  \in \NN^{\Sigma ( 1 )} $ such that $\chi ' \leqslant \chi$,
let $\Phi^\chi_{\chi ' }$ the morphism given by the coordinatewise multiplication 
\[
\begin{array}{rll}
\Phi^\chi_{\chi ' } : 
\Hom^{\chi '} 
\left ( 
\mathbf A_\kb^2 \setminus \{ 0 \} , \mathbf A_\kb^{\Sigma ( 1 )}
\right ) 
\times \Hom^{\chi - \chi '} 
\left ( 
\mathbf A_\kb^2 \setminus \{ 0 \} , \mathbf A_\kb^{\Sigma ( 1 )}
\right ) 
& \longrightarrow 
& \Hom^\chi
\left ( 
\mathbf A_\kb^2 \setminus \{ 0 \} , \mathbf A_\kb^{\Sigma ( 1 )} \right ) \\
\left ( (f_\alpha )_{\alpha \in \Sigma ( 1 )} 
, ( g _\alpha )_{\alpha \in \Sigma ( 1 )}
\right ) 
& \longmapsto 
& \left ( 
f_\alpha  g_\alpha 
\right )_{\alpha \in \Sigma ( 1 )}  .
\end{array}
\]
For any $(\chi - \chi ')$-equivariant morphism $g$, consider the induced map
\label{def:morphism-adding-a-fixed-coarse-degeneration}
\[
\begin{array}{rll}
\Phi_g^\chi : \Hom^{\chi '} 
\left ( 
\mathbf A_\kb^2 \setminus \{ 0 \} , \mathbf A_\kb^{\Sigma ( 1 )}
\right ) 
\otimes \kappa ( g ) 
& \longrightarrow 
& \Hom^\chi
\left ( 
\mathbf A_\kb^2 \setminus \{ 0 \} , \mathbf A_\kb^{\Sigma ( 1 )} \right ) \otimes \kappa ( g ) 
 \\
(f_\alpha )_{\alpha \in \Sigma ( 1 )} 
& \longmapsto 
& \left ( 
f_\alpha  g_\alpha 
\right )_{\alpha \in \Sigma ( 1 )}  .
\end{array}
\]
\end{mydef}


\smallskip 

Our goal is to give an interpretation of the motivic Möbius inversion in terms of the $\Phi_{\chi '}^\chi$'s,
in a way compatible with the stratification and piecewise identification of $\Hom^\chi ( \mathbf A^2 \setminus \{ 0 \} , \mathbf A^{\Sigma ( 1 ) })$ from \cref{construction:piecewise-id-hom-zero}.

\smallskip 

As a first step, instead of working over $\kb$, we are going to work 
with zero-cycles on $\PP^1_G$ relatively to 
$G = \GG_m^{\Sigma ( 1 )}.$
In what follows, 
we are going to exploit the details of the proof of the multiplicative property of motivic Euler products as it is done in \cite[\S 3.9]{bilu2018motivic}. 
We take $I_0=\NN^{\Sigma ( 1 )} $ 
and $I = I_0 \setminus \{ \mathbf 0 \}$, but in general one can take $I$ to be any abelian semi-group and $I_0 = I \cup \{ 0 \}$, see \cite[\S 3.9.3]{bilu2018motivic}.

\medskip

Let $X$ be a variety above a certain base scheme $S$ and
take $\nn , \nn ' , \nn '' \in \NN^{\Sigma ( 1 ) }$ such that 
\[
\nn = \nn ' + \nn  '' .
\] 
Let $\kappa ' $ and $\kappa '' $ 
be partitions
respectively
of $\nn '  $ and $\nn ''$.
The previous relation becomes 
\[
\sum_{\ii\in I} \ii \kappa_\ii ' + \sum_{\ii\in I} \ii \kappa_\ii '' = \nn .
\]
Let $\gamma = (n_{\pp , \qq})_{(\pp , \qq) \in I_0^2 \setminus \{ 0 \}}$ 
be a collection of integers such that for all 
$\pp , \qq \in I$,
\[
\sum_{\qq\in I_0} n_{\pp , \qq} = \kappa_\pp ' 
\qquad 
\sum_{\pp\in I_0} n_{\pp , \qq} = \kappa_\qq '' . 
\]
Let $\kappa = \kappa ( \gamma ) $ be the ``overlap partition'' 
of $\nn$
given by 
\[
\kappa_i = \sum_{\pp+\qq = \ii} n_{\pp , \qq}
\]
for all $\ii \in I$. 
It is important to note that $\kappa , \kappa ' $ and $ \kappa ''$ are entirely determined by $\gamma$.
Let 
\[
\Sym_{ / S }^{\kappa '} ( X )_*
\times_\gamma \Sym_{ / S }^{\kappa ''} ( X  )_*
\]
be the locally closed subset of 
\[
\Sym_{ / S }^{\kappa '} ( X  )_*
\times \Sym_{ / S }^{\kappa ''} ( X  )_*
\]
given by points whose 
$\Sym_{ / S }^{\kappa_\pp '} ( X )$-component 
and $\Sym_{ / S }^{\kappa_\qq ''} ( X  )$-component
overlap exactly above an effective zero-cycle of degree $n_{\pp , \qq}$ \cite[\S 3.9.6]{bilu2018motivic}.
One can show by induction
\cite[\S 3.9.7.3]{bilu2018motivic}
that 
there is a canonical isomorphism
\[
\Sym_{ / S }^{\kappa '} ( X )_*
\times_\gamma \Sym_{ / S }^{\kappa ''} ( X  )_*
\simeq 
\left ( 
\prod_{\ii\in I}
\prod_{\substack{(\pp , \qq) \in I_0^2 \setminus \{ 0 \} \\ \pp + \qq = \ii }} \Sym^{ n_{\pp , \qq} }_{ / S } ( X )
\right )_* .
\]
The right-hand side will be denoted by $\Sym_{ / S }^{\gamma} ( X )_*$.
Then, there is a canonical morphism 
\begin{equation}\label{def-phi-gamma}
	 \Phi^\gamma : \Sym_{ / S }^{\kappa '} ( X )_*
\times_\gamma \Sym_{ / S }^{\kappa ''} ( X  )_*
\simeq 
\Sym_{ / S }^{\gamma} ( X )_* 
\longrightarrow 
\Sym^{\kappa ( \gamma )}_{ / S } ( X )_*
\end{equation}
induced 
by the morphisms
\[
\prod_{\pp+\qq = \ii} \Sym^{ n_{\pp , \qq} }_{ / S } ( X )
\to \Sym^{ \kappa_\ii ( \gamma ) }_{ / S } ( X ) 
\qquad
\ii \in I
.
\]
If $\AAA = ( A_\pp )_{\pp \in I}$
and $\BBB = ( B_\qq )_{\qq\in I}$
are two families of $X$-varieties,
then 
\[
\Sym_{X/S}^{\kappa ' ( \gamma )} ( \AAA ) \times_\gamma 
\Sym_{X/S}^{\kappa '' ( \gamma )} ( \BBB )
\]
is the class above $\Sym_{ / S }^{\kappa '} ( X )_*
\times_\gamma \Sym_{ / S }^{\kappa ''} ( X  )_*$ 
obtained from 
\[
\Sym_{X/S}^{\kappa ' ( \gamma )} ( \AAA ) \boxtimes
\Sym_{X/S}^{\kappa '' ( \gamma )} ( \BBB )
\]
by pull-back,
while 
$\Sym_{ / S }^{\gamma} ( \AAA \boxtimes_X \BBB )_*$
is defined as
\[
\left ( 
\prod_{\ii\in I} \prod_{\pp+\qq = \ii}
\Sym_{X/S}^{n_{\pp , \qq}} ( A_\pp \boxtimes_X B_\qq )
\right )_* . 
\]
\begin{myexample}
Taking $S=G$ and $X=\PP^1_G$,
\[
\Sym_{\PP^1_G / G }^{\kappa '} ( \PP^1_G )_*
\times_\gamma \Sym_{\PP^1_G / G }^{\kappa ''} ( \PP^1_G  )_*
\]
the canonical morphism $\Phi^\gamma$
\begin{align*}
	& \Sym^{\kappa ' }_{ / G } 
( \PP^1_G )_*
\times_\gamma
\Sym^{\kappa '' }_{ / G }
( \PP^1_G )_*
\simeq 
\left ( 
\prod_{(\pp , \qq) \in I_0 \setminus \{ 0 \} } \Sym^{ n_{\pp , \qq} }_{ / G } ( \PP^1_G )
\right )_* 
\\
& \qquad 
\overset{\Phi^\gamma} {\longrightarrow }
\Sym^{\kappa ( \gamma )}_{  / G } ( \PP^1_G )_*
\end{align*}
sends couples of tuples of zero-cycles to their coordinate-wise sum.
This is the zero-cycle version of the map $\Phi^{\chi}_{\chi '}$ we previously defined.
\end{myexample}

\begin{myptn}[Multiplicativity of motivic Euler products -- refined version]
\label{ptn:refined-multiplicativity-motivic-Euler-products}
Let $X $ be a variety above a base scheme $S$. 
Let 
$\AAA$ and $\BBB$ be families of varieties above $X$ indexed by $I$. 
Let $\gamma = (n_{\pp , \qq})_{(\pp , \qq) \in I_0^2 \setminus \{ 0 \}}$ 
be a collection of integers as above.
	Then
	\[
\Sym_{ / S }^{\gamma} ( \AAA \boxtimes_X \BBB )_*
= 
\Sym_{X/S}^{\kappa ' ( \gamma )} ( \AAA )_* \times_\gamma 
\Sym_{X/S}^{\kappa '' ( \gamma )} ( \BBB )_*
	\]
	in $\KVar{\Sym_{/S}^{\kappa ( \gamma )} ( X ) } $. 
\end{myptn}
 
\begin{proof}
	See \cite[\S 3.9.7]{bilu2018motivic}.
\end{proof} 
 
In the following remarks,
the use of the subscript ``${}_{(*)}$''
means that the claims remain valid 
when one takes the restriction of the symmetric products
to the complement of the diagonal.  
 
\begin{myremark}
\label{rmk:toric-mobius-decomposition-classes-linear-combination-elementary-pieces}
	Let $\mathfrak a$ be a class in the Grothendieck ring of varieties over a certain $S$-variety $X$.
	Then for any $k \in \NN^*$,
	\begin{align*}
	\Sym^k_{X/S} ( 2 \mathfrak a )_{(*)}
	& 
	= \sum_{k_1+k_2=k} \Sym^{k_1,k_2} (  \mathfrak a , \mathfrak a )_{(*)}
	\end{align*}
	in $\KVar{\Sym_{/S}^k (  X )_{(*)}}$,
	and more generally for any $k, \ell \in \NN^*$,
	\begin{align*}
		\Sym^k_{X/S}  ( \ell \mathfrak a  )_{(*)}
		& 
		= \sum_{k_1 + ... + k_\ell = k } \Sym^{k_1,...,k_\ell} _{X/S} ( \mathfrak a , ... , \mathfrak a )_{(*)}
	\end{align*}
	It is important to recall that if $\mathfrak a = [ Y \to X ]$ is an effective class,
	and $Y_1, ... , Y_\ell$ are $\ell$ copies of $Y$,
	these identities 
	come from the canonical decomposition of 
	\[
	\Sym^k_{X/S} \left ( Y_1 \coprod ... \coprod Y_\ell \right )_{(*)}
	\] 
	into the $\Sym^{k_1,...,k_\ell} _{X/S}  \left ( Y_1 , ... , Y_\ell \right )_{(*)} $ relatively to $\Sym_{/S}^k (  X )_{(*)}$. 
\end{myremark}

\begin{myremark}	
	Denoting by $\mathcal Q \subset \NN^{(\NN^*)}$ the set of partitions of integers (in the usual sense) without holes, 
	then for any $k \in \NN^*$ 
	we have the relation 
	\begin{align*}
		\Sym^k_{X/S} ( - \mathfrak a )_{(*)}
		= 
		\sum_{
		\substack{
		\kappa = ( k_i )_{i\in \NN} \in \mathcal Q
		\\
		\sum_i k_i = k }
		}
		(-1)^{|\{  i \in \NN \mid k_i > 0 \}|} \Sym^\kappa_{X/S} ( \mathfrak a )_{(*)}
	\end{align*}
	above $\Sym^k_{/S} ( X )_{(*)}$  \cite[Example 6.1.4]{bilu2021motivic}.
	Again, 
	it is important to remember where this relation comes from.
	
	If $\mathfrak a = [ Y \to X ]$ is an effective class in $ \KVar{X}$,
	then 
	$\Sym^k_{X/S} ( - \mathfrak a )$
	is \textit{by definition}
	the degree $k$ part 
	of the inverse of (the class of)
	\[
	\Sym^\bullet_{X/S} ( Y )
	= ( 1 , Y , \Sym^2_{X/S} ( Y ) , ... ) 
	\]
	viewed above 
	\[
	\Sym^\bullet_{/S} ( X ) = \prod_{\ell \in \NN} \Sym^\ell_{/S} ( X ) .
	\] 
	Indeed, 
	$\KVar{\Sym_{/S}^\bullet ( X )  }$
	admits a natural structure of graded $ \KVar{X}$-algebra 
	for which the product law
	is 
	induced by the natural maps
	\[
\Sym_{/S}^i ( X )  \times \Sym_{/S}^{k-i} ( X ) \to  
\Sym_{/S}^k ( X ) ,
	\]
	see \cite[\S 6.1.1]{bilu2021motivic}. 
	Then, one writes
	\begin{align*}
		\frac{1}{	\Sym^\bullet_{X/S} ( Y )} 
		& =
				\frac{1}{ 1 + \left (	\Sym^\bullet_{X/S} ( Y ) - 1 \right ) } \\
		& = \sum_{k=0}^\infty 
		( - 1 )^k 
		\left ( 0 , Y ,  \Sym^2_{X/S} ( Y ) , ... \right ) ^k 
		\\
	\end{align*}
	and takes the degree $k$ part. 
	This provides an explicit definition of $\Sym^k_{X/S} ( -[Y] )$
	in terms of the classes of the $\Sym^{\kappa}_{X/S} ( Y ) $ for $\kappa \in \mathcal Q$ partition of $k$ and above $\Sym_{/S}^k ( X )$. In particular, the arrows are explicit: they are the natural ones. 	
\end{myremark}

Now, 
starting from the relation
\[
Q_{B_\Sigma} ( \TT ) 
= 
\frac{
P_{B_\Sigma } ( \TT )
}
{\prod_{\alpha \in \Sigma ( 1 )} ( 1 - t_\alpha )}
\]
in $\ZZ [[ \TT ]]$,
that is to say,
\[
\sum_{\nn \in \NN^{\Sigma ( 1 )}} \mathbf 1_{ A ( B_\Sigma ) } ( \nn ) \TT^\nn 
= 
\left ( 
\sum_{\nn \in \NN^{\Sigma ( 1 )}} \TT^\nn 
\right ) 
\left ( 
\sum_{\nn \in \NN^{\Sigma ( 1 )}} \mu_{B_\Sigma} ( \nn ) \TT^\nn 
\right ) 
\]
(the definition of $A ( B_\Sigma ) $ was given \cpageref{definition-ABsigma}),
we take the motivic Euler products associated to the corresponding 
$\PP^1_G$-families
\[
\left ( 
\mathbf 1_{A ( B_\Sigma )} ( \nn ) [ \PP^1_G ]
\right ) ,
\qquad
\left ( 
\PP^1_G
\right ) ,
\qquad
\left ( 
\mu_{B_\Sigma} ( \nn ) \left [ \PP^1_{\{1_G\}} \right ]
\right ) ,
\]
all three of them being indexed by $\nn \in \NN^{\Sigma ( 1 )} \setminus \{ 0 \}$
(an index we will drop in the following computations).
More precisely,
we apply 
\cref{ptn:refined-multiplicativity-motivic-Euler-products}
\cpageref{ptn:refined-multiplicativity-motivic-Euler-products}
to
$\AAA = \left ( 
\PP^1_G
\right ) $
and 
${\BBB =\left ( 
\mu_{B_\Sigma} ( \nn ) \left [ \PP^1_{\{1_G\}} \right ]
\right ) } $.
If $\varpi$ is any generalized partition of a certain non-zero tuple $\nn \in \NN^{\Sigma ( 1 )}$,
then one gets the relation in $\KVar{\Sym^\varpi_{/ G} ( \PP^1_G )}$
\begin{align*}
	& \Sym_{\PP^1_G / G }^\varpi 
\left 
( \mathbf 1_{A ( B_\Sigma )} [ \PP^1_G ] 
\right )_*
\\ 
& = 
\sum_{\substack{\gamma = ( n_{\pp , \qq} ) \\ \kappa ( \gamma ) = \varpi }} 
 \Sym_{\PP^1_G / G }^\gamma
 \left ( \PP^1_G
 \boxtimes_{\PP^1_G} 
 \mu_{ B_\Sigma } \left [ \PP^1_{ \{1_G\} } \right ] 
 \right )\\
& = 
\sum_{\substack{\gamma = ( n_{\pp , \qq} ) \\ \kappa ( \gamma ) = \varpi }} 
\Sym_{\PP^1_G / G }^{\kappa ' ( \gamma )}
\left 
( \PP^1_G
\right )_*
\times_\gamma
\Sym_{/ G }^{\kappa '' ( \gamma )}
\left 
(  \mu_{ B_\Sigma } \left [ \PP^1_{ \{1_G\} } \right ]
\right ) _* .
\end{align*}
Each 
\[
\Sym_{\PP^1_G / G }^{\kappa ' ( \gamma )}
\left 
( \PP^1_G
\right )_*
\times_\gamma
\Sym_{/ G }^{\kappa '' ( \gamma )}
\left 
(  \mu_{ B_\Sigma } \left [ \PP^1_{ \{1_G\} } \right ]
\right ) _* .
\]
can be explicitly decomposed above 
\[
\Sym_{/ G }^{\kappa ' ( \gamma )}
\left 
( \PP^1_G
\right )_*
\times_\gamma 
\Sym_{/ G }^{\kappa '' ( \gamma )}
\left 
( \PP^1_{ \{ 1_G \} }
\right ) _*
\]
as a linear combination of effective classes. 
Indeed,
using the previous two remarks, we get that 
	for all $\pp \in \NN^{\Sigma ( 1 )}\setminus \{ 0 \}$
	the class 
	\[
	\Sym_{\PP^1_G / G }^{\kappa_\pp '' ( \gamma ) }
\left 
( \mu_{ B_\Sigma } ( \pp ) \left [ \PP^1_{ \{1_G\} } \right ]
\right )_*
	\]
(which is zero if $\pp \notin B_\Sigma $)
	is 
\[
	\sum_{k_1 + ... + k_{| \mu_{B_\Sigma} ( \pp)  |} = \kappa_\pp '' ( \gamma )  } 
	\Sym^{ k_1 , ... , k_{| \mu_{B_\Sigma} ( \pp)  |} }_{\PP^1_G / G } 
	\left ( 
	\mathrm{sign}( \mu_{B_\Sigma} ( \pp ) 
	 )
	\left [\PP^1_{\{ 1_G \}} \right ] 
	\right )
\]	
above $\Sym^{\kappa_\pp '' ( \gamma ) }_{/ G } ( \PP^1_G  )$. 
If $\mu_{B_\Sigma} ( \pp )$ is negative, 
we replace 
\[
\Sym^{ \kappa_\pp '' ( \gamma )  }_{\PP^1_G / G } 
	\left ( 
-	\left [
	\PP^1_{\{ 1_G \}} 
	\right ]  
	\right )
\]
by its definition
\[
\sum_{
		\substack{
		\lambda  = ( \lambda_\iota )_{\iota \in \NN^*} \in \mathcal Q
		\\
		\sum_\iota \lambda_\iota = \kappa_\pp '' ( \gamma )  }
		}
		(-1)^{|\{  \iota \in \NN^* \mid \lambda_\iota  > 0 \}|} 
		\left [
		\Sym^\lambda_{/ G } ( \PP^1_{\{ 1_G \}} )
		\right ]
\]
in the expression above
and pull it back to $\Sym^{ k_1 , ... , k_{| \mu_{B_\Sigma} ( \pp)  |} }_{\PP^1_G / G} 
	\left ( 
	\PP^1_{\{ 1_G \}}  
	\right )$.
Finally, as always one restricts to the complement of the diagonal. 
This shows explicitly that 
\[
\Sym_{\PP^1_G / G }^{\kappa ' ( \gamma )}
\left 
( \PP^1_G
\right )_*
\times_\gamma
\Sym_{/ G }^{\kappa '' ( \gamma )}
\left 
(  \mu_{ B_\Sigma } \left [ \PP^1_{ \{1_G\} } \right ]
\right ) _* .
\]
is a linear combination with integer coefficients of restrictions of 
\[
\Sym_{/ G }^{\kappa ' ( \gamma )}
\left 
( \PP^1_G 
\right )_*
\times_\gamma 
\Sym_{\PP^1_G / G }^{\kappa '' ( \gamma )}
\left 
(  \PP^1_{ \{1_G\} }
\right )_*
\overset{\Phi^\gamma} {\longrightarrow }
\Sym^{\kappa ( \gamma )}_{ / G } ( \PP^1_G )_*
\] 
to constructible subsets
of the form
\[
\Sym_{/ G }^{\kappa ' ( \gamma )}
\left 
( \PP^1_G 
\right )_*
\times_\gamma
\left ( \prod_\pp \prod_{\iota \geqslant 1} \Sym_{/G}^{\lambda_{ \pp , \iota ,  1 } , ... , \lambda_{ \pp , \iota ,  | \mu_{B_\Sigma} ( \pp ) |}} ( \PP^1_{\{ 1_G \}} ) 
\right )_* 
\]
with 
\[
\sum_\iota  ( \lambda_{ \pp , \iota ,  1 } + ... + \lambda_{ \pp , \iota , | \mu_{B_\Sigma} ( \pp )|}) = \kappa_\pp '' ( \gamma ) 
\] for all $\pp\in  I = \NN^{\Sigma ( 1 )} \setminus \{ 0 \} $
and 
\[{( \lambda_{ \pp , \iota ,  1 } + ... + \lambda_{ \pp , \iota , | \mu_{B_\Sigma} ( \pp )|})_{\iota \in \NN^*} \in \mathcal Q}
\]
In case $\mu_{B_\Sigma} ( \pp )$ is positive, as a convention, we will always take  
\[
\lambda_\pp =  {( \lambda_{ \pp , \iota ,  1 } + ... + \lambda_{ \pp , \iota , | \mu_{B_\Sigma} ( \pp )|})_{\iota \in \NN^*} \in \mathcal Q} 
\]
to be the trivial partition $\lambda_\pp = ( \kappa_\pp '' ( \gamma ) , 0 , ... )$
and we will only have to perform a sum over all possible partitions (in the usual sense)
of $ \kappa_\pp '' ( \gamma )$ of length $\mu_{B_\Sigma} ( \pp )$. 

\medskip 

	For all $\chi$, 
	we apply \cref{lemma:piecewise-iso-induce-same-k0}
	to the 
	$G$-equivariant piecewise isomorphism 
	between 
	\[
	\Hom^\chi ( \mathbf A^2\setminus \{ 0 \} , \mathbf A^{\Sigma ( 1 )} )^*
	\]
	and 
	\[
	G \times \Sym^\chi_{/ \kb} ( \PP^1_\kb ) 
	\simeq \Sym^\chi_{/ G } ( \PP^1_G ) 
	\]
	from \cref{construction:piecewise-id-hom-zero} \cpageref{construction:piecewise-id-hom-zero}.
	In particular,
	in 
	\[
	\KVar {\Hom^{\chi  } \left ( \mathbf A_\kb ^2  \setminus \{ 0 \} , \mathbf A_\kb^{\Sigma ( 1 )} \right )^* } \simeq \KVar {\Sym_{/G}^\chi ( \PP^1_G )}
	\]
	we identify the class of the space of unitary $\chi$-equivariant morphisms with the one of 
	\[
	\Sym^\chi_{/ G } \left ( \PP^1_{ \{1_G\} } \right ) \simeq \Sym^\chi_{/ \kb } ( \PP^1_\kb ) .
	\]
	Recall that this induces a piecewise identification between
	$\Hom^\chi ( \mathbf A^2\setminus \{ 0 \} , \mathcal T_\Sigma )$ 
	and $\Sym_{\PP^1_G / G }^\chi
\left 
( \mathbf 1_{A ( B_\Sigma ) } \PP^1_G  
\right )
$. 
	These piecewise isomorphisms commute with addition of cycles (the $\Phi^\gamma$'s of \eqref{def-phi-gamma} \cpageref{def-phi-gamma}) and multiplication of the corresponding polynomials (the $\Phi^\chi_{\chi '}$'s of \cref{def-phi-chi-chi-prime} \cpageref{def-phi-chi-chi-prime}). 
	Putting all this together, we get that 	
\begin{align}
	& \left [
	\Hom^{\chi  } ( \mathbf A_\kb ^2  \setminus \{ 0 \} , \mathcal T_\Sigma )
	\right ] \notag \\
	& = 
	\sum_{ \substack{\varpi \text{ partition of } \chi  \\ \gamma = ( n_{\qq,\pp} ) \\ \kappa ( \gamma ) = \varpi } }
	\Sym_{\PP^1_G / G }^{\kappa ' ( \gamma )}
\left 
( \PP^1_G
\right )_*
\times_\gamma
\Sym_{/ G }^{\kappa '' ( \gamma )}
\left 
(  \mu_{ B_\Sigma } \left [ \PP^1_{ \{1_G\} } \right ]
\right ) _*  \notag \\
& = \sum_{ \substack{\varpi \text{ partition of } \chi  \\ \gamma = ( n_{\qq,\pp} ) \\ \kappa ( \gamma ) = \varpi } } 
\quad 
\sum_{\substack{
		( ( \lambda_{ \pp,\iota} )_{\iota \in \NN^*} )  \in \mathcal Q \times \{ \pp \in I \mid \mu_{B_\Sigma} ( \pp ) < 0 \}
		\\
		\sum_\iota \lambda_{ \pp,\iota} = \kappa_\pp '' ( \gamma ) 
		\\  
		\lambda_{ \pp , \iota ,  1 } + ... + \lambda_{ \pp , \iota , | \mu_{B_\Sigma} ( \pp )|}
		= \lambda_{ \pp,\iota} }
		} \notag  \\
& \qquad 
(-1)^{\sum_{\substack{ \pp \in I \\ \mu_{B_\Sigma} ( \pp ) < 0 } } |\{  \iota \in \NN^* \mid \lambda_{ \pp,\iota}  > 0 \}|} 
\Sym_{/ G }^{\kappa ' ( \gamma )}
\left 
( \PP^1_G 
\right )_*
\times_\gamma
\left ( \prod_{\pp \in I} \prod_{\iota \geqslant 1} \Sym_{/G}^{\lambda_{ \pp , \iota ,  1 } , ... , \lambda_{ \pp , \iota ,  | \mu_{B_\Sigma} ( \pp ) |}} ( \PP^1_{\{ 1_G \}} ) 
\right )_* 	
\label{equation:refined-Möbius-inversion}
\end{align}
	in 
	$\KVar {\Hom^{\chi  } \left ( \mathbf A_\kb ^2  \setminus \{ 0 \} , \mathbf A_\kb^{\Sigma ( 1 )} \right )^* } \simeq \KVar {\Sym_{/G}^\chi ( \PP^1_\GG )}$,
	with structure morphisms given by the $\Phi^\gamma$'s, seen as restrictions of the $\Phi^\chi_{\chi'}$'s. 


\begin{myexample}
	If $V = \PP_\kb^n$, then $B_\Sigma = \{ ( 1 , ... , 1 ) \} = \{ \mathbf 1 \}$ and 
	\begin{align*}
		\mu_{B_\Sigma} ( \mathbf 0 ) & = 1,\\
		\mu_{B_\Sigma} ( \mathbf 1 ) & = - 1, \\
		\mu_{B_\Sigma} ( \nn ) & = 0 \text{ whenever } \nn \notin \{ \mathbf 0 , \mathbf 1 \}
	\end{align*}
	so in this case \eqref{equation:refined-Möbius-inversion} is the sum
	over all $\gamma = (( n_{\qq,\mathbf 0} ) , (n_{\qq,\mathbf 1} ), n_{\mathbf 0 , \mathbf 1}) \in \NN^{(I)} \times \NN^{(I)} \times \NN$  
	and all possible partitions $\lambda \in \mathcal Q$ of the number $\kappa '' ( \gamma )_\mathbf 1 = n_{\mathbf 0 , \mathbf 1 } + \sum_{\qq\in I} n_{\mathbf q , 1 }$,
	of the class 
	\[
	(-1)^{| \{ \iota \in \mathbf N^* \mid  \lambda_\iota > 0 \} |} 
	\Sym_{/G}^{\kappa ' ( \gamma )} ( \PP^1_G )_* 
	\times_\gamma  
	\Sym_{/G}^{\lambda } ( \PP^1_{\{ 1_G \}} ) _* 
	\]
	the overlap partition being 
	$\kappa  ( \gamma ) = ( n_{\qq,\mathbf 0} ) + (n_{\qq - \mathbf 1,\mathbf 1} ) $.
	The relation one obtains corresponds to the removal of certain spaces of common divisors
	of the coordinates of the morphisms. 
\end{myexample}

\begin{myexample}
	Let $V$ be the blow-up in one point of $\PP^2_\kb$. Let $L_1 , L_2$ and $L_3$ be strict transforms of three distincts lines,
	 such that the intersection of 
		$L_1$ and $L_3$
	 is the point we blew up and the third one does not contain it. Let $E=L_0$ be the exceptional line. Since $L_2 \cap L_0 = \varnothing $ and $L_1 \cap L_3 = \varnothing$ are the only \emph{empty} possible intersections between two of these four divisors, the minimal elements of $B_\Sigma$ are $(1,0,1,0)$ and $(0,1,0,1)$. 
	Then one checks that 
	\begin{align*}
		&\mu_{B_\Sigma} ( \mathbf 0 )  = 1,\\
		&\mu_{B_\Sigma} ( ( 1 , 0 , 1 , 0 ) )  = \mu_{B_\Sigma} ( (0 , 1 , 0 , 1 ) ) =  -1,  \\
		&\mu_{B_\Sigma} ( \mathbf 1 )  = 1 \\
		&\mu_{B_\Sigma} ( \nn )  = 0 \text{ otherwise } 
	\end{align*}
and hence 
\[
\sum_{\nn \in \{ 0 , 1 \}^4} \mu_{B_\Sigma} ( \nn ) \TT^\nn
= 1 - (t_0t_2 + t_1t_3) + t_0t_1t_2t_3 . 
\]
Since $|\mu_{B_\Sigma} ( \nn )| \leqslant 1$ for all $\nn$, the sets of partitions of the $\lambda_{ \pp,\iota}$'s are trivial in \eqref{equation:refined-Möbius-inversion}.
The product over $I$ 
becomes 
\[
\left ( 
\left (
\prod_{\iota \geqslant 1}
\left ( 
\Sym_{/G}^{\lambda_{(1,0,1,0),\iota}} ( \PP^1_{\{ 1_G \}} ) 
\times  
\Sym_{/G}^{\lambda_{(0,1,0,1),\iota}} ( \PP^1_{\{ 1_G \}} )
\right )
\right )
\times 
\Sym_{/G}^{\kappa_{(1,1,1,1)}'' ( \gamma )}  ( \PP^1_{\{ 1_G \}} ) 
\right )_* .
\]
\end{myexample}

In general, the combinatorics involved in our computation 
quickly become complicated, as the following second example shows. 
In particular, $|\mu_{B_\Sigma} |$ can take values different from $0$ or $1$.

\begin{myexample}\label{example-blow-up-3-pts}
Let $V$ be the blow-up of $\PP^2_\kb$ in three general points $p_1$, $p_2$ and $p_3$.
This case (for which $n+r = 6$) was already considered in the arithmetic setting in an unpublished work of Peyre from 1993 \cite{peyre1993DP6}.
	
	Let $L_1$, $L_2$ and $L_3$ be the exceptional lines in $V$ lying above the three general points, respectively $p_1$, $p_2$ and $p_3$,
	and $L_4$, $L_5$ and $L_6$ be the strict transforms of the lines $(p_2,p_3)$, $(p_1 , p_3 )$ and $(p_1 , p_2 )$.
	
	Then $\mu_{B_\Sigma}$ 
	can be more easily computed by considering the so-called graph $\GGG$ of \emph{nonintersection} of these six  exceptional lines,
	whose vertices are labeled by the $L_i$'s, $i \in \{ 1 , ... , 6\} = \Sigma ( 1 )$,
		and edges are $(L_1 , L_2 ), (L_2,L_3), (L_3, L_1 )$, $(L_4,L_5 ) , ( L_5 , L_6 ) , ( L_6 , L_4 )$ and $(L_1, L_4),(L_2,L_5),(L_3,L_6)$.

		Indeed, the set $B_\Sigma$ can be identified with the set of induced subgraphs of $\GGG$
		by sending $\nn \in B_\Sigma$ to the subgraph of $\GGG$ induced by $\{ L_i \mid n_i = 1\}$.
		Actually it is enough to compute the values of $\mu_{B_\Sigma}$ for connected ones, $\mu_{B_\Sigma}$ being multiplicative for pairs of elements having disjoint supports.
		For the convenience of the reader, we included an exhaustive list of these connected subgraphs (for $| \nn | \geqslant 2$) together with the corresponding values of $\mu_{B_\Sigma}$, in an appendix \cpageref{annexe}. 
		In the end, the polynomial $P_{B_\Sigma} ( \TT ) $ equals
	\begin{align*}
	\sum_{\nn \in \{ 0 , 1 \}^6} \mu_{B_\Sigma} ( \nn ) \TT^\nn
	= & 
	1 - (t_1t_2 + t_2t_3 + t_3t_1 + t_4t_5 + t_5t_6 + t_6t_4 + t_1t_4 + t_2t_5 + t_3t_6) \\
	& + 2 ( t_1t_2t_3 + t_4t_5t_6 ) \\
	& + (t_2t_4t_5 + t_1t_4t_2 + t_1t_3t_4 + t_3t_4t_6 + t_2t_3t_5 + t_3t_5t_6)\\ 
	& + (t_1t_4t_5 + t_1t_2t_5 + t_1t_3t_6 + t_1t_4t_6 + t_2t_3t_6 + t_2t_5t_6)\\ 
	& - ( t_1t_2t_4t_5 + t_1t_3t_4t_6 + t_2t_3t_5t_6 ) \\
	& - ( t_1t_2t_3t_6 + t_3t_4t_5t_6 + t_1t_2t_3t_4 + t_1t_4t_5t_6 + t_1t_2t_3t_5 + t_2t_4t_5t_6) \\
	& + t_1t_2t_3t_4t_5t_6.
	\end{align*}
	In particular, the monomials $t_1t_2t_3$ and $t_4t_5t_6$ contribute twice (with a positive coefficient), hence we have to sum over partitions of $\kappa_{(1,1,1,0,0,0)}''(\gamma )$ and $\kappa_{(0,0,0,1,1,1)}'' ( \gamma )$ of length two. 
	
	Taking symmetric products, one gets this sort of exclusion-inclusion relation 
	\eqref{equation:refined-Möbius-inversion}
	between our class of interest, $[\Hom^{\chi  } ( \mathbf A_\kb ^2  \setminus \{ 0 \} , \mathcal T_\Sigma )]$,
	and classes of subspaces of ${\Hom^{\chi  } ( \mathbf A_\kb ^2  \setminus \{ 0 \} , \mathbf A_\kb^6 )}$
	consisting of equivariant morphisms which 
	may take values in some non-allowed intersection of hyperplanes.
\end{myexample}

\subsubsection{Detwisting}	
	In what follows $\mu$ is any partition of $\chi - \chi '$
	\emph{whose support is contained in $B_\Sigma $}.
	We identify $\Sym^\mu_{/ \kb} 
\left (  \PP^1_{\kb} \right )_*$
with the space of unitary $(\chi - \chi ')$-equivariant morphisms whose zeros have multiplicities prescribed by $\mu$. 
	We are going to use the subscript notation
\[
( \; \cdots \; )_{\widetilde W} 
\] 
for the preimages of $\Hom^{\chi  } ( \mathbf A_\kb ^2  \setminus \{ 0 \} , \mathbf A_\kb^{\Sigma ( 1 )} \mid \widetilde W ) $
by the $\Phi^\chi_{\chi '}$'s,
as well as the concise notations
\[
S^\mu_{\widehat{|\SSS |} \, *} 
\qquad
\text{ and }
\qquad 
H^\chi
\]
respectively for $\Sym^\mu_{/ \kb} 
\left (  \PP^1_{\kb} \setminus | \SSS | \right )_*$ 
and $\Hom^\chi$, when necessary.

Let
\begin{align*}
&	\left ( \mathrm \Hom^{\chi '} ( \mathbf A^2_\kb \setminus \{ 0 \} , \mathbf A_\kb^{\Sigma ( 1 )} )^* \times \Sym^\mu_{/ \kb} 
 (  \PP^1_{\kb}  )_*\right )_{\widetilde W} \\
& = \left ( \mathrm \Hom^{\chi '} ( \mathbf A^2_\kb \setminus \{ 0 \} , \mathbf A_\kb^{\Sigma ( 1 )} )^* \times \Sym^\mu_{/ \kb} 
\left (  \PP^1_{\kb} \setminus | \SSS | \right )_*\right )_{\widetilde W} \\
& = 
\left ( \mathrm H^{\chi '} ( \mathbf A^2_\kb \setminus \{ 0 \} , \mathbf A_\kb^{\Sigma ( 1 )} )^* \times S^\mu_{\widehat{|\SSS |} \, *}  \right )_{\widetilde W}
\end{align*}
be the 
preimage of $
\mathrm \Hom^{\chi }
\left ( 
\mathbf A_\kb^2 \setminus \{ 0 \} , \mathbf A_\kb^{\Sigma ( 1 )}
\mid \widetilde W
\right )^* $
by 
\[
\Phi_{\mu}^\chi : 
 \mathrm \Hom^{\chi '} ( \mathbf A^2_\kb \setminus \{ 0 \} , \mathbf A_\kb^{\Sigma ( 1 )} )^*
 \times 
\Sym^\mu_{/\kb} ( \PP^1_\kb ) _* 
\overset{\Phi^\chi_{\chi '}}{\longrightarrow} \mathrm \Hom^{\chi } ( \mathbf A^2_\kb \setminus \{ 0 \} , \mathbf A_\kb^{\Sigma ( 1 )} )^*  ,
\]
where the first equality comes from the fact that one cannot add common zeros above $\SSS$ since $\widetilde W$ is a constructible subset of $\Hom ( \SSS , \mathcal T_\Sigma ) $.
The projection onto the second factor endows it with the structure of a {$\Sym_{/\kb}^\mu
( \mathbf P^1_\kb )_*$-variety}. 
We are going to compare its class with the one of 
\[
\Hom^{\chi '}
\left ( 
\mathbf A_\kb^2 \setminus \{ 0 \} , \mathbf A_\kb^{\Sigma ( 1 )} 
\, \left | \,
\widetilde W  
\right.
\right )^* \times \Sym_{/\kb}^\mu
( \mathbf P^1_\kb \setminus | \SSS |)_*. 
\]
The relation between all the spaces involved is summarised by the commutative diagram given \cpageref{cd:relation-moduli-spaces-reduction-conditions}. 

\begin{figure} 
	\begin{sideways} \label{cd:relation-moduli-spaces-reduction-conditions}
\begin{tikzcd} 
&
\left ( \mathrm H^{\chi '} ( \mathbf A^2_\kb \setminus \{ 0 \} , \mathbf A_\kb^{\Sigma ( 1 )} )^* \times S^\mu_{\widehat{|\SSS |} \, *}  \right )_{\widetilde W}
\ar[hook']{dl}[swap, sloped, near start]{}
\ar{rr}{\Phi^\chi_{\mu}}
\ar[]{dd}[near end]{(\res_\SSS , \, \pr_2)}
& &  \mathrm H^\chi ( \mathbf A^2_\kb \setminus \{ 0 \} , \mathbf A_\kb^{\Sigma ( 1 )} \mid \widetilde W )^* 
\ar{dd}{\res_\SSS}
\ar[hook']{dl}[swap, sloped, near start]{}
\\
\mathrm H^{\chi '} ( \mathbf A^2_\kb \setminus \{ 0 \} , \mathbf A_\kb^{\Sigma ( 1 )} )^* \times S^\mu_{\widehat{|\SSS |} \, *}  
\ar[crossing over]{rr}[near start]{\Phi^\chi_{\mu}}
\ar{dd}[swap]{(\res_\SSS , \, \pr_2)}
& & \mathrm H^\chi ( \mathbf A^2_\kb \setminus \{ 0 \} , \mathbf A_\kb^{\Sigma ( 1 )} )^* 
\\
&
\left ( 
\Hom ( \SSS , \mathbf A_\kb^{\Sigma ( 1 )} ) \times S^\mu_{\widehat{|\SSS |} \, *}  
\right )_{\widetilde W}
\ar[near start]{rr}{ ( \varphi , D ) \mapsto \varphi \overline D}
\ar[sloped, near end]{dl}{}
& & \widetilde W
\ar[hook']{dl}
\\
\Hom ( \SSS , \mathbf A_\kb^{\Sigma ( 1 )} ) \times S^\mu_{\widehat{|\SSS |} \, *}  
\ar{rr}{ ( \varphi , D ) \mapsto \varphi \overline D}
& & \Hom ( \SSS , \mathbf A_\kb^{\Sigma ( 1 )} )
\ar[crossing over, leftarrow, near start]{uu}{}
\end{tikzcd}
\end{sideways}
\end{figure}

For any point $D\in \Sym^\mu_{/ \kb} 
\left (  \PP^1_{\kb} \setminus | \SSS | \right )_*$,
the preimage of 
$\Hom^{\chi }
\left ( 
\mathbf A_\kb^2 \setminus \{ 0 \} , \mathbf A_\kb^{\Sigma ( 1 )} \mid \widetilde W \right )^* $
by $\Phi_D^\chi$
is the subspace
\[
\Hom^{\chi ' }
\left ( 
\mathbf A_{\kappa ( D )}^2 \setminus \{ 0 \} , \mathbf A_{\kappa ( D )}^{\Sigma ( 1 )} \mid 
\overline D ^{-1} 
\widetilde W \right )^*
\]
of $\chi '$-equivariant morphisms whose reduction modulo $\SSS$ lies in 
$
\overline{D}^{-1} \widetilde W 
$.
Moreover, 
when $\chi ' \geqslant \ell ( \SSS ) $,
this preimage is isomorphic to 
\[
\Hom^{\chi '}
\left ( 
\mathbf A_{\kappa ( D )}^2 \setminus \{ 0 \} , \mathbf A_{\kappa ( D )}^{\Sigma ( 1 )} \mid  \widetilde W \right )^* ,
\] the isomorphism being explicitly given in \cref{lemma:restriction-toric-torsor-locally-trivial}: it sends an element of Euclidean decomposition modulo $\varpi$
\[  
\left (
q_\alpha ,  \overline{D_\alpha}^{-1} w_\alpha 
\right ) _{\alpha \in \Sigma ( 1 )}
\]
with $(w_\alpha )\in \widetilde W$,
to 
\[
\left (
q_\alpha , w_\alpha 
\right ) _{\alpha \in \Sigma ( 1 )} .
\]
Now we can prove the following version of the motivic Möbius inversion.
\begin{myptn}\label{ptn:classes-of-constrained-morphisms-fixed-degen-degree}
\label{proposition:toric-stratification}
For all $\chi \in \NN^{\Sigma ( 1 )}$, we have
\begin{align*}
&	\left [
\Hom^\chi 
\left ( 
\mathbf A_\kb^2 \setminus \{ 0 \} , \mathcal T_\Sigma  
\left | \; 
\widetilde W
\right. 
\right ) ^* 
\right ]\\
& = 
\sum_{\chi ' , \, \mu } 
\left [ 
\Hom^{\chi '}
\left ( 
\mathbf A_\kb^2 \setminus \{ 0 \} , \mathbf A_\kb^{\Sigma ( 1 )}
\, \left | \,
\widetilde W  
\right.
\right )^* 
\right ]
\Sym^\mu_{\PP^1_\kb} 
\left ( \mu_{B_\Sigma} \cdot \left [ \PP^1_{\kb} \setminus | \SSS | \right ] 
\right )_*\\
& \quad + E^\chi 
\end{align*}
in
	$\KVar {\kb}$,
with the error term $E^\chi $ of bounded dimension
\[
\dim_{\kb} ( E^\chi ) 
\leqslant 
- \min_{\alpha \in \Sigma ( 1 ) } ( \chi_{\alpha } ) 
+ | \chi | 
+ \left ( 1 - | \Sigma ( 1 ) | \right )\left ( \ell ( \SSS ) - 1 \right )
+ \dim \left ( \widetilde W \right ) .
\]
\end{myptn}

\begin{proof}
	First, note that the class $\Sym^\mu_{\PP^1_\kb} 
\left ( \mu_{B_\Sigma} \cdot \left [ \PP^1_{\kb} \setminus | \SSS | \right ] 
\right )_*$ is zero if the support of $\mu$ is not contained in $B_\Sigma$.	
 Then, using the decomposition 
 \eqref{equation:refined-Möbius-inversion} \cpageref{equation:refined-Möbius-inversion}
 of the motivic Möbius inversion 
 we made explicit in the previous pages (starting from \cpageref{def-phi-gamma} concerning the notations), 
 and pulling it back to $\Hom^\chi 
\left ( 
\mathbf A_\kb^2 \setminus \{ 0 \} , \mathbf A_\kb^{\Sigma ( 1 )}  
\left | \; 
\widetilde W
\right. 
\right ) ^* $,
 we get that 
 $\left [
\Hom^\chi 
\left ( 
\mathbf A_\kb^2 \setminus \{ 0 \} , \mathcal T_\Sigma  
\left | \; 
\widetilde W
\right. 
\right ) ^* 
\right ]$
is a linear combination with 
integer coefficients of restrictions of the $\Phi^\gamma$'s.
More explicitly,
	\begin{align}
	& \left [
	\Hom^{\chi  } \left ( \mathbf A_\kb ^2  \setminus \{ 0 \} , \mathcal T_\Sigma \mid \widetilde W \right )^*
	\right ] \notag \\
	& = 
	\sum_{ \substack{\varpi \text{ partition of } \chi  \\ \gamma = ( n_{\qq,\pp} ) \\ \kappa ( \gamma ) = \varpi } }
	\left (
\Sym_{/ G }^{\kappa ' ( \gamma )}
\left 
( \PP^1_G 
\right ) _* 
\times_\gamma
\Sym_{\PP^1_G / G }^{\kappa '' ( \gamma )}
\left 
( \mu_{ B_\Sigma } \left [ \PP^1_{ \{1_G\} } \right ]
\right )_*
\right )_{\widetilde W} \notag \\
& = \sum_{ \substack{\varpi \text{ partition of } \chi  \\ \gamma = ( n_{\qq,\pp} ) \\ \kappa ( \gamma ) = \varpi } } 
\quad 
\sum_{\substack{
		( ( \lambda_{ \pp,\iota} )_{\iota \in \NN^*} )  \in \mathcal Q \times \{ \pp \in I \mid \mu_{B_\Sigma} ( \pp ) < 0 \}
		\\
		\sum_\iota \lambda_{ \pp,\iota} = \kappa_\pp '' ( \gamma ) 
		\\  
		\lambda_{ \pp , \iota ,  1 } + ... + \lambda_{ \pp , \iota , | \mu_{B_\Sigma} ( \pp )|}
		= \lambda_{ \pp,\iota} }
		}   \notag \\
& \qquad 
(-1)^{\sum_{\substack{ \pp \in I \\ \mu_{B_\Sigma} ( \pp ) < 0 }} |\{  \iota \in \NN \mid \lambda_{ \pp,\iota}  > 0 \}|} 
\underbrace{\left ( 
\Sym_{/ G }^{\kappa ' ( \gamma )}
\left 
( \PP^1_G 
\right )_*
\times_\gamma
	\left ( \prod_{\pp \in I} \prod_{\iota \geqslant 1} \Sym_{/G}^{\lambda_{ \pp , \iota ,  1 } , ... , \lambda_{ \pp , \iota ,  | \mu_{B_\Sigma} ( \pp ) |}} ( \PP^1_{\{ 1_G \}} ) 
\right )_* 	
\right )_{\widetilde W}}
_{ = \left ( \Sym_{/ G }^{\kappa ' ( \gamma )}
\left 
( \PP^1_G 
\right )_*	
\times_\gamma
\left ( \prod_{\pp \in I} \prod_{\iota \geqslant 1} \Sym_{/G}^{\lambda_{ \pp , \iota ,  1 } , ... , \lambda_{ \pp , \iota ,  | \mu_{B_\Sigma} ( \pp ) |}} ( \PP^1_{\{ 1_G \}} \setminus | \SSS | ) 
\right )_* 
\right )_{\widetilde W} } . 
\label{equation:refined-Möbius-inversion-constraints}
\end{align}
For every $\gamma$ appearing in the first sum,
we apply \cref{construction:piecewise-id-hom-zero}
to the factors of the locally closed subsets 
\[
\left ( \Sym_{/ G }^{\kappa ' ( \gamma )}
\left 
( \PP^1_G 
\right )_*	
\times_\gamma
\left ( \prod_{\pp \in I} \prod_{\iota \geqslant 1} \Sym_{/G}^{\lambda_{ \pp , \iota ,  1 } , ... , \lambda_{ \pp , \iota ,  | \mu_{B_\Sigma} ( \pp ) |}} ( \PP^1_{\{ 1_G \}} \setminus | \SSS | ) 
\right )_* 
\right )_{\widetilde W}
\]
of 
\[
\left (\Sym_{/ G }^{\kappa ' ( \gamma )}
\left 
( \PP^1_G 
\right ) _*
\times
\Sym_{\PP^1_G / G }^{\kappa '' ( \gamma )}
\left 
(  \PP^1_{ \{1_G\} } 
\right )_*
\right )_{\widetilde W} .
\]
In particular,
it is sufficient to approximate the class of  
 \begin{equation}
 \label{equation:elementary-piece-to-approximate}
 	 \left ( 
	\Hom^{\chi ' } ( \mathbf A_\kb ^2  \setminus \{ 0 \} , \mathbf A_\kb^{\Sigma ( 1 )} ) ^*
	\times
\Sym^{\kappa '' ( \gamma )}_{ / \kb }
\left ( 
 \PP^1_\kb 
\setminus | \SSS | 
\right ) _*
\right )_{\widetilde W}
\overset{\Phi^\chi_{\kappa '' ( \gamma ) }}{\longrightarrow}
 \Hom^\chi ( \mathbf A_\kb ^2  \setminus \{ 0 \} , \mathbf A_\kb^{\Sigma (1)} \mid \widetilde W  ) ^* 
 \end{equation}
 whenever $\kappa ' ( \gamma )$ is a partition of $\chi ' $,
 since all the other classes involved are obtained 
 from this one by pull-back and restriction.
 Moreover,
 since we are going to perform a motivic sum over ${\Hom^\chi ( \mathbf A^2 \setminus \{ 0 \} , \mathbf A^{\Sigma ( 1 )} \mid \widetilde W )^* }$,
 we can view \eqref{equation:elementary-piece-to-approximate} as a variety 
above $\Sym^{\kappa '' ( \gamma ) }_{/ \kb }
\left ( 
 \PP^1_\kb 
\setminus | \SSS | 
\right ) _*
$ and then take the motivic sum, 
using the commutativity of the following diagram.
\[
\begin{tikzcd}
		\left ( 
		\Hom^{\chi ' } ( \mathbf A_\kb ^2  \setminus \{ 0 \} , \mathbf A_\kb^{\Sigma (1)}  ) ^*
		\times 
	\Sym^{\kappa '' ( \gamma ) } _{/ \kb} 
\left (  \PP^1_{\kb} \setminus | \SSS | \right )_*
 \right )_{\widetilde{W}} 
 \rar["\pr_1"]	\dar["\Phi_{{\kappa '' ( \gamma ) }}^\chi"] &  \Sym^{\kappa '' ( \gamma ) }_{/ \kb} 
\left (  \PP^1_{\kb} \setminus | \SSS | \right )_*  \dar \\
		 \Hom^\chi ( \mathbf A_\kb ^2  \setminus \{ 0 \} , \mathbf A_\kb^{\Sigma (1)} \mid \widetilde W  ) ^*	 \rar & \Spec ( \kb ) 
\end{tikzcd}
\]
By \cite[Lemma 2.5.5]{bilu2021motivic}
it is enough to argue fibre by fibre, so that the previous diagram becomes
\[
\begin{tikzcd}
 \Hom^{\chi ' } \left ( \mathbf A_{\kappa ( D  )} ^2  \setminus \{ 0 \} , \mathbf A_{\kappa ( D  )}^{\Sigma (1)}  \mid  {\overline D}^{-1} \widetilde{W} \right ) ^*	  
 \rar	\dar["\Phi_D^\chi"] &  \Spec ( \kappa ( D ) )   \dar \\
		 \Hom^\chi ( \mathbf A_\kb ^2  \setminus \{ 0 \} , \mathbf A_\kb^{\Sigma (1)}  ) ^*	 \otimes_\kb \kappa ( D )  \rar & \Spec ( \kb ) 
\end{tikzcd}
\]	
and our argument will be entirely compatible with restrictions to constructible subsets of $\Sym^{\kappa '' ( \gamma ) }_{/ \kb} 
\left (  \PP^1_{\kb} \setminus | \SSS | \right )_*
$
such as $ \prod_{\pp \in I} \prod_{\iota \geqslant 1} \Sym_{/G}^{\lambda_{ \pp , \iota ,  1 } , ... , \lambda_{ \pp , \iota ,  | \mu_{B_\Sigma} ( \pp ) |}} ( \PP^1_{\{ 1_G \}} \setminus | \SSS | ) $.

We can apply the second part of \cref{lemma:euclidean-division-in-univ-torsor} \cpageref{lemma:euclidean-division-in-univ-torsor} only when $\chi ' \geqslant \ell ( \SSS ) - 1 $,
which in that case gives 
\[
\Hom^{\chi ' } \left ( \mathbf A_{\kappa ( D )}^2  \setminus \{ 0 \} , \mathbf A_{\kappa ( D )}^{\Sigma (1)}  \mid  {\overline D}^{-1} \widetilde{W} \right ) ^*
\overset{\sim}{\underset{\tau_{\overline D}}{\longrightarrow}} \Hom^{\chi ' } 
\left ( \mathbf A_\kb ^2  \setminus \{ 0 \} , \mathbf A_\kb^{\Sigma (1)}  \mid \widetilde{W} 
\right ) ^* \otimes \kappa ( D ) 
\]
as ${\kappa ( D )}$-schemes.
In general, we consider the error term
\begin{align*}
	E^{\chi ' }_\mu
 & =
\left [ 
\left (  		
\Hom^{\chi ' } 
 \left ( \mathbf A_\kb ^2  \setminus \{ 0 \} , \mathbf A_\kb^{\Sigma (1)} 
 \right ) ^*	
\times 
   \Sym^{\kappa '' ( \gamma )}_{/ \kb} 
\left (  \PP^1_{\kb} \setminus | \SSS | \right )_*
\right )_{\widetilde W}
\right ]
\\
& \quad - 
	\left [  
	\Hom^{\chi ' } 
 \left ( \mathbf A_\kb ^2  \setminus \{ 0 \} , \mathbf A_\kb^{\Sigma (1)} \mid \widetilde W 
 \right ) ^*
 \times 
	\Sym^{\kappa '' ( \gamma )} _{/ \kb} 
\left (  \PP^1_{\kb} \setminus | \SSS | \right )_*
\right ] \\
& \qquad 
\in 
\KVar {\Sym^{\kappa '' ( \gamma )}_{\PP^1_\kb} 
\left ( \PP^1_\kb \setminus | \SSS |\right )_*} 
\end{align*}
as well as its restrictions 
\begin{align*}
	E^\gamma_{\lambda} 
	= & 
	\left [
	\left ( \Sym_{/ G }^{\kappa ' ( \gamma )}
\left 
( \PP^1_G 
\right )_*	
\times_\gamma
\left ( \prod_{\pp \in I} \prod_{\iota \geqslant 1} \Sym_{/G}^{\lambda_{ \pp , \iota ,  1 } , ... , \lambda_{ \pp , \iota ,  | \mu_{B_\Sigma} ( \pp ) |}} ( \PP^1_{\{ 1_G \}} \setminus | \SSS | ) 
\right )_* 
\right )_{\widetilde W} 
\right ] \\
& -
\left [ 
 \left ( \Sym_{/ G }^{\kappa ' ( \gamma )}
\left ( 
\PP^1_G 
\right )_*	\right )_{\widetilde W}
\times_\gamma
\left ( \prod_{\pp \in I} \prod_{\iota \geqslant 1} \Sym_{/G}^{\lambda_{ \pp , \iota ,  1 } , ... , \lambda_{ \pp , \iota ,  | \mu_{B_\Sigma} ( \pp ) |}} ( \PP^1_{\{ 1_G \}} \setminus | \SSS | ) 
\right )_* 
\right ]
\end{align*}
for every $( \lambda_\pp  ) = ( ( \lambda_{ \pp,\iota} )_{\iota \in \NN^*} )  \in \mathcal Q \times I 
		$ such that $
		\sum_\iota \lambda_{ \pp,\iota} = \kappa_\pp '' ( \gamma ) $ and $
		\lambda_{ \pp , \iota ,  1 } + ... + \lambda_{ \pp , \iota , | \mu_{B_\Sigma} ( \pp )|}
		= \lambda_{ \pp,\iota} $
		(still with the convention that we only consider the trivial partition of $\kappa_\pp '' ( \gamma )$ if $\mu_{B_\Sigma} ( \pp ) > 0$).
The previous argument 
shows that $E^{\chi '}_\mu = 0$
if $\chi ' \geqslant \ell ( \SSS ) - 1$.
By \eqref{equation:toric-dim-bound-under-constraints}
of \cref{lemma:euclidean-division-in-univ-torsor},
the relative dimension of $E^{\chi ' }_\mu$
(hence also of $E^\gamma_{\lambda} $)
is bounded by 
\[
\dim \left ( \widetilde W \right ) 
+ 
\sum_{\alpha \in \Sigma ( 1 ) } 
\max ( 0 , \chi '_\alpha - \ell ( \SSS ) + 1) 
\] 
and 
if $\chi ' \ngeqslant \ell ( \SSS ) - 1 $, then there is at least one term of this sum which is equal to zero. Hence under this assumption it is bounded by 
\begin{equation}
\label{bound-toric-twisted-error-term}
- \min_{\alpha  \in \Sigma ( 1 ) } ( \chi_\alpha ' ) 
+ \ell ( \SSS) - 1 + | \chi ' | + | \Sigma ( 1 ) |\left ( 1 - \ell ( \SSS )\right )
+ \dim \left ( \widetilde W \right )  .  
\end{equation}
Now we bound the total dimension of $E^{\chi '}_\mu$,
that is to say 
\[
\dim_\kb ( E^{\chi ' }_\mu )
=
\dim_{\Sym^{\mu }_{\PP^1_\kb} 
\left ( \PP^1_\kb \setminus | \SSS |\right )_* } ( E^{\chi '}_\mu ) + \dim_\kb 
\left ( \Sym^{\mu }_{\PP^1_\kb} 
\left ( \PP^1_\kb \setminus | \SSS |\right )_*
\right ) .
\] 
Since in practice we work with the family $\mu_{B_\Sigma} [ \PP^1_\kb ]$, remembering that $\mu_{B_\Sigma} ( \nn ) = 0 $ if $\nn \notin B_\Sigma $,
we can assume that $\mu$
is a partition of the form 
\[
\mu 
= \left ( \delta_J \right )_{J\in B_\Sigma}  \in \NN^{B_\Sigma} . 
\]
Let
$\varphi ( \underline \delta ) \in \NN^{\Sigma ( 1 ) } \subset \mathcal X_* ( \GG_m^{\Sigma ( 1 )} )$ be the cocharacter whose $\alpha$-coordinate is
 \[
 \varphi ( \underline \delta )_\alpha = \sum_{J\ni \alpha } \delta_J.
 \]
Then, by \eqref{bound-toric-twisted-error-term} $E^{\chi '}_\mu$ has dimension over $\kb$ bounded by
\begin{align*}
	& - \min_{\alpha  \in \Sigma ( 1 ) } ( \chi_{\alpha } ' ) 
+ \ell ( \SSS) -1 + | \chi ' | + | \Sigma ( 1 ) | ( 1 - \ell ( \SSS ) ) 
+ \dim \left ( \widetilde W \right ) + | \underline \delta |  \\ 
= & - \min_{\alpha  \in \Sigma ( 1 ) } \left ( \chi_{\alpha  } - \varphi ( \underline \delta )_{\alpha  } \right ) \\
& + | \chi | - | \varphi ( \underline \delta ) | + | \underline \delta |    \\
& + ( | \Sigma ( 1 ) | - 1 ) ( 1 - \ell ( \SSS ))
+ \dim \left ( \widetilde W \right ) \\
\leqslant & - \min_{\alpha  \in \Sigma ( 1 ) } ( \chi_{\alpha } ) 
+ | \chi | + ( | \Sigma ( 1 ) | - 1 ) ( 1 - \ell ( \SSS ))
+ \dim \left ( \widetilde W \right )
\end{align*}
where the first equality is given by 
\[
\chi = \chi ' + \varphi ( \underline \delta )
\]
and the last inequality comes from the expression
\[
 | \varphi ( \underline \delta ) |
=
\sum_{J\in B_\Sigma } | J | \delta_J 
\]
together with the fact that
 $|J|\geqslant 2$ for all $J\in B_\Sigma$, hence 
 \[
- \min_{\alpha  \in \Sigma ( 1 ) } \left ( \chi_{\alpha  } - \varphi ( \underline \delta )_{\alpha  } \right ) 
-   | \varphi ( \underline \delta ) |
\leqslant 
- \min_{\alpha  \in \Sigma ( 1 )} ( \chi_{\alpha } ) 
+ \max_{\alpha  \in \Sigma ( 1 )} ( \varphi ( \underline \delta )_\alpha )
- | \varphi ( \underline \delta ) |
\leqslant 
- \min_{\alpha  \in \Sigma ( 1 )} ( \chi_{\alpha } ) 
 \]
 and the proposition is finally proved
 for $E^\chi = \sum_{\gamma , \lambda } E^\gamma_\lambda$
 by
 replacing every
 \[
 \left ( \Sym_{/ G }^{\kappa ' ( \gamma )}
\left 
( \PP^1_G 
\right )_*	
\times_\gamma
\left ( \prod_{\pp \in I} \prod_{\iota \geqslant 1} \Sym_{/G}^{\lambda_{ \pp , \iota ,  1 } , ... , \lambda_{ \pp , \iota ,  | \mu_{B_\Sigma} ( \pp ) |}} ( \PP^1_{\{ 1_G \}} \setminus | \SSS | ) 
\right )_* 
\right )_{\widetilde W} 
 \]
 with 
 \[
 \left ( \Sym_{/ G }^{\kappa ' ( \gamma )}
\left 
( \PP^1_G 
\right )_*	\right )_{\widetilde W}
\times_\gamma
\left ( \prod_{\pp \in I} \prod_{\iota \geqslant 1} \Sym_{/G}^{\lambda_{ \pp , \iota ,  1 } , ... , \lambda_{ \pp , \iota ,  | \mu_{B_\Sigma} ( \pp ) |}} ( \PP^1_{\{ 1_G \}} \setminus | \SSS | ) 
\right )_* 
 \]
in  \eqref{equation:refined-Möbius-inversion-constraints} \cpageref{equation:refined-Möbius-inversion-constraints}.
\end{proof}

\begin{proof}[Final computation] 
\index{motivic height zeta function!of smooth split projective toric varieties!with constraints}

Let 
\[
E_W ( \TT )
= \sum_{\chi \in \NN^{\Sigma ( 1 )}}
E^\chi \TT^\chi . 
\]
The previous proposition can be rewritten
\begin{align*}
& \sum_{ \dd \in \NN^{\Sigma ( 1 )} } \left [ \Hom^{\dd} \left ( \mathbf A^2_\kb \setminus \{ 0 \} , \mathcal T_\Sigma \mid \widetilde{W} \right )^* \right ] \TT^\dd \\
& =
\left ( \prod_{p\notin \SSS }  P_{B_\Sigma} ( \TT )  \right ) 
\times 	
\left ( \sum_{ \dd  \in \NN^{\Sigma ( 1 )}} \left [ \Hom^{\dd} \left ( \mathbf A^2_\kb \setminus \{ 0 \} , \mathbf A_\kb^{\Sigma ( 1 )} \mid \widetilde{W} \right )^* \right ] \TT^\dd \right ) \\
& \qquad + E_W ( \TT )
\end{align*}
in $\KVar \kb [[ \TT ]]$.
We define $\mu_{\Sigma }^{| \SSS | } ( \ee ) $, $\ee \in \NN^{\Sigma ( 1 )}$, to be the coefficients of the motivic Euler product 
\[
\prod_{p\notin | \SSS | }  P_{B_\Sigma}( \TT ) .
\]
By definition of $\widetilde W$, together with \cref{ptn:repres-moduli-space-toric} and the equivalent description of the functor ${S \rightsquigarrow \Hom_S^{\chi } ( \mathbf A^2_S \setminus \{ 0 \} , \mathcal T_{\Sigma , S}   )^* }$ we gave,
we have the relation
\[
( \LL_\kb - 1 ) ^r 
\left [ \Hom_\kb^\dd  ( \PP^1_\kb , V_\Sigma \mid W )_U \right ]
=  \left [ \Hom_\kb^\dd  ( \mathbf A^2 \setminus \{ 0 \} , \mathcal T_\Sigma \mid \widetilde{W} )^*  \right ] 
\]
as soon as $\dd \in \CEff ( V )^\vee_\ZZ$,
and by \cref{lemma:restriction-toric-torsor-locally-trivial}
\[
\left [ \Hom_\kb^\dd  ( \mathbf A^2 \setminus \{ 0 \} , \mathbf A ^{\Sigma ( 1 )} \mid \widetilde{W} ) ^* \right ] = \left [ \widetilde W \right ] ( \LL - 1 )^{|\Sigma ( 1 )|} \prod_{\alpha \in \Sigma ( 1 ) } \left [ \PP^{d_\alpha - \ell ( \SSS ) }_\kb \right ]
\]
whenever $\dd \geqslant \ell ( \SSS )$.
Thus we decompose the following series into two parts: 
\[
	 \sum_{ \dd  \in \NN^{\Sigma ( 1 )}} \left [ \Hom^{\dd} \left ( \mathbf A^2_\kb \setminus \{ 0 \} , \mathbf A_\kb^{\Sigma ( 1 )} \mid \widetilde{W} \right ) \right ] \TT^\dd \\
	 = \left [ \widetilde W \right ] ( \LL - 1 )^{|\Sigma ( 1 )|} \prod_{\alpha \in \Sigma ( 1 ) } t_\alpha^{\ell ( \SSS ) }Z^\text{Kapr}_{\PP^1_\kb}  ( t_\alpha )  +  H_W ( \TT )  
\]
where 
\[
H_W ( \TT ) = 	 \sum_{\substack{ \dd  \in \NN^{\Sigma ( 1 )} \\ \dd \ngeqslant \ell ( \SSS ) }} \left [ \Hom^{\dd} ( \mathbf A^2_\kb \setminus \{ 0 \} , \mathbf A_\kb^{\Sigma ( 1 )} \mid \widetilde{W} )^*  \right ] \TT^\dd . 
\]
Then, we use again the decomposition \eqref{equation:decomposition-product-Kapr-P1} given \cpageref{equation:decomposition-product-Kapr-P1}
\[
\prod_{\alpha \in \Sigma ( 1 )} Z^{\text{Kapr}}_{\PP^1_\kb} ( t_\alpha ) = \sum_{A\subset \Sigma ( 1 )} \frac{( - \LL )^{| \Sigma ( 1 ) | - | A |}}{ ( 1 - \LL )^{| \Sigma ( 1 ) |}} Z_A ( \TT ) 
\]
of this product of Kapranov zeta functions, 
where for any $A\subset \Sigma ( 1 ) $
\[
Z_A ( \TT ) = \prod_{\alpha \in A } ( 1 - t_\alpha )^{-1} \prod_{\alpha \notin A} ( 1  - \LL t_\alpha )^{-1} .
\]
By identification, the coefficient of order $\dd $ of 
\[
\TT^{\ell ( \SSS ) } Z_A ( \TT ) \times \prod_{p\notin | \SSS |} P_{B_\Sigma } ( \TT )  
\]
is the sum
\[
\mathfrak s^A_\dd = 
\sum_{\ee \leqslant \dd} \mu_{\Sigma}^{|\SSS|} ( \ee )  \LL^{\sum_{\alpha \notin A} d_\alpha - \ell ( \SSS ) - e_\alpha }
\]
whenever $\dd \geqslant \ell ( \SSS )$, and zero otherwise.\\
If $A=\varnothing$, then after dividing by $\LL^{-|\dd |}$ it becomes 
\[
\mathfrak s^A_\dd \LL^{-|\dd |}
= 
\LL^{-|\Sigma ( 1 )|\ell ( \SSS )} \sum_{\ee \leqslant \dd} \mu_{\Sigma}^{|\SSS|} ( \ee )  \LL^{ - | \ee | }
\]
which is, up to the factor $\LL^{-|\Sigma ( 1 )|\ell ( \SSS )}$, the $\dd$-th partial sum of  $  \prod_{p\notin | \SSS |} P_{B_\Sigma } ( \LL^{-1} )$.
The corresponding error term
\[
\LL^{-|\Sigma ( 1 )|\ell ( \SSS )} \sum_{\ee \nleqslant \dd} \mu_{\Sigma}^{|\SSS|} ( \ee )  \LL^{ - | \ee | }
\]
has virtual dimension at most
\[
-|\Sigma ( 1 )|\ell ( \SSS ) - \frac{1}{2} \min_{\alpha \in \Sigma ( 1 ) } ( d_\alpha + 1 ). 
\]
If $A\neq \varnothing$, then one gets instead
\[
\LL^{-( |\Sigma ( 1 )| - |A|) \ell ( \SSS )} \sum_{\ee \leqslant \dd} \mu_{\Sigma}^{|\SSS|} ( \ee )  \LL^{ - | \ee | } \LL^{ | \dd_A - \ee_A |}. 
\]
In that case,
recalling that $\dim (  \mu_\Sigma ( \ee ) \LL_\kb^{- | \ee |}  ) \leqslant - \frac{1}{2} | \ee |$ for all $\ee\in\NN^{\Sigma ( 1 )}$,
 \cref{lemma:controlling-error-terms-convolution} \cpageref{lemma:controlling-error-terms-convolution} gives 
\[
\dim ( \sum_{\ee \leqslant \dd} \mu_{\Sigma}^{|\SSS|} ( \ee )  \LL^{ - | \ee | } \LL^{ | \dd_A - \ee_A |} ) \leqslant - \frac{1}{4} \min_{\alpha \in A } ( d_\alpha ). 
\]
Now we consider the terms coming from $H_W$. The coefficient of order $\dd $ of 
\[
H_W ( \TT ) \times \prod_{p\notin | \SSS |} P_{B_\Sigma } ( \TT )  
\]
is 
\[
\mathfrak h_\dd = \sum_{\substack{ \ee \leqslant \dd  \\ \dd \ngeqslant \ell ( \SSS ) + \ee  }}  \mu_{\Sigma}^{|\SSS|} ( \ee )  \left [ \Hom^{\dd - \ee } ( \mathbf A^2_\kb \setminus \{ 0 \} , \mathbf A_\kb^{\Sigma ( 1 )} \mid \widetilde{W} )^* \right ]. 
\]
Dividing by $\LL^{|\dd |}$, we get 
\[
\mathfrak h_\dd \LL^{-|\dd|}
=
\sum_{\substack{ \ee \leqslant \dd  \\ \dd \ngeqslant \ell ( \SSS ) + \ee  }}  \mu_{\Sigma}^{|\SSS|} ( \ee ) \LL^{-|\ee |} \left [ \Hom^{\dd - \ee } ( \mathbf A^2_\kb \setminus \{ 0 \} , \mathbf A_\kb^{\Sigma ( 1 )} \mid \widetilde{W} )^* \right ] \LL^{-|\dd - \ee |} . 
\]
As for the coefficients of $E_W ( \TT )$, we have the dimensional upper bound 
coming from \eqref{equation:toric-dim-bound-under-constraints} of \cref{lemma:euclidean-division-in-univ-torsor}
\[	
\dim \left ( \Hom^{\dd ' } ( \mathbf A^2_\kb \setminus \{ 0 \} , \mathbf A_\kb^{\Sigma ( 1 )} \mid \widetilde{W} ) ^* \right )   \leqslant \dim ( \widetilde W ) + \sum_{\alpha \in \Sigma ( 1 ) } \min ( 0 ,  d_\alpha '  - \ell ( \SSS ) + 1  ) .
\]
Given $\dd ' \in \NN^{\Sigma ( 1 )}$ 
such that $\dd ' \ngeqslant \ell ( \SSS ) $, there exists at least one element $\alpha \in \Sigma ( 1 )$ such that $d_\alpha < \ell ( \SSS ) $,
that is to say such that $\min ( 0 ,  d_\alpha '  - \ell ( \SSS ) + 1  ) = 0$.
Hence the sum 
\[
\sum_{\alpha \in \Sigma ( 1 ) } \min ( 0 ,  d_\alpha '  - \ell ( \SSS ) + 1  )
\]
is bounded by 
\[
| \dd '  | + | \Sigma ( 1 ) | ( 1 - \ell ( \SSS ) ) - \min_{\substack{\alpha \in \Sigma ( 1 )\\ d_\alpha ' < |\SSS |}} (d'_\alpha ) + \ell ( \SSS ) - 1  
\]
and the bound on $\dim \left ( \Hom^{\dd ' } ( \mathbf A^2_\kb \setminus \{ 0 \} , \mathbf A_\kb^{\Sigma ( 1 )} \mid \widetilde{W} ) ^* \right )$ becomes
\begin{align}
&	\dim \left ( \Hom^{\dd ' } ( \mathbf A^2_\kb \setminus \{ 0 \} , \mathbf A_\kb^{\Sigma ( 1 )} \mid \widetilde{W} ) ^* \right ) \notag \\
& \leqslant \dim ( \widetilde W ) + | \dd '  | + | \Sigma ( 1 ) | ( 1 - \ell ( \SSS ) ) - \min_{\substack{\alpha \in \Sigma ( 1 )\\ d_\alpha ' < |\SSS |}} (d'_\alpha ) + \ell ( \SSS ) - 1  \label{ineq:toric-second-error-term-dim-bound}
\end{align}
for all $\dd ' \in \NN^{\Sigma ( 1 )}$ 
such that $\dd ' \ngeqslant \ell ( \SSS ) $. Thus the dimension of 
\[
\left [ \Hom^{\dd - \ee } ( \mathbf A^2_\kb \setminus \{ 0 \} , \mathbf A_\kb^{\Sigma ( 1 )} \mid \widetilde{W} ) \right ] \LL^{-|\dd - \ee |}
\]
in the expression of $\mathfrak h_\dd$ above is bounded. We can be more precise and argue as  we did in the proof of \cref{lemma:controlling-error-terms-convolution}.
If $2 \ee \leqslant \dd $ then $2 ( \dd - \ee ) \geqslant \dd$ and 
\[
\dim \left ( \mu_{\Sigma}^{|\SSS|} ( \ee ) \LL^{-|\ee |} \left [ \Hom^{\dd - \ee } ( \mathbf A^2_\kb \setminus \{ 0 \} , \mathbf A_\kb^{\Sigma ( 1 )} \mid \widetilde{W} )^* \right ] \LL^{-|\dd - \ee |} \right ) 
\]
is at most 
\[
 - \frac{1}{2} \min_{\alpha \in \Sigma ( 1 ) } (d_\alpha ) +  \dim ( \widetilde W ) + ( 1 -   \ell ( \SSS ) ) ( | \Sigma ( 1 ) |  - 1 )
\]
while if $2 \ee \nleqslant \dd$  we use the coarse upper bound deduced from \eqref{ineq:toric-second-error-term-dim-bound}
\[
\dim \left ( \left [ \Hom^{\dd - \ee } ( \mathbf A^2_\kb \setminus \{ 0 \} , \mathbf A_\kb^{\Sigma ( 1 )} \mid \widetilde{W} )^* \right ] \LL^{-|\dd - \ee |} \right ) 
\leqslant 
 \dim ( \widetilde W ) + ( 1 -   \ell ( \SSS ) ) ( | \Sigma ( 1 ) |  - 1 )
\]
together with 
\[
\dim (  \mu_{\Sigma}^{|\SSS|} ( \ee ) \LL^{-|\ee |} ) \leqslant - \frac{1}{2} | \ee | < - \frac{1}{4} \min_{\alpha \in \Sigma	 ( 1 )} (d_\alpha ) .
\]
Therefore, for any $\dd$ we have
\[
\dim ( \mathfrak h_\dd \LL^{-|\dd|} ) 
\leqslant - \frac{1}{4} \min_{\alpha \in \Sigma ( 1 ) } (d_\alpha )  + \dim ( \widetilde W ) + ( 1 -   \ell ( \SSS ) ) ( | \Sigma ( 1 ) |  - 1 ) .
\]
Remember 
from \cref{ptn:classes-of-constrained-morphisms-fixed-degen-degree}
that the $\dd$-th term of $E_W ( \TT ) $ has dimension bounded by 
\[
- \min_{\alpha  \in \Sigma ( 1 ) } ( d_{\alpha } ) 
+ | \chi | + ( \Sigma ( 1 ) - 1 ) ( 1 - \ell ( \SSS ))
+ \dim \left ( \widetilde W \right ) . 
\]
Now we rewrite the motivic density of $\widetilde W$ as follows:
\begin{align*}
\left [ \widetilde W \right ]   \LL^{- |\Sigma ( 1 )| \ell ( \SSS )} 
& =
[ W ] \LL^{-\ell ( \SSS ) \dim ( V_\Sigma )} \times [ T_{\NS , \SSS } ] \LL^{-r \ell ( \SSS )} \\
& = \frac{[ W ]}{ \LL^{\ell ( \SSS ) \dim ( V_\Sigma )}  } \prod_{p\in | \SSS |} \left ( 1  - \LL_p ^{-1}\right )^r 
\end{align*}
Putting everything together, we get
\begin{align*}
& \left [ \Hom_\kb^\dd  ( \PP^1_\kb , V_\Sigma \mid W )_U \right ] \LL^{-|\dd |}\\
& =  ( \LL - 1 )^{-r} \left (\LL^{| \Sigma ( 1 ) |} [ \widetilde W] 
\sum_{A\subset \Sigma ( 1 )} ( - \LL )^{ - | A |} \mathfrak s^A_\dd \LL^{-|\dd |} 
+ \mathfrak h_d \LL^{-|\dd |} + \mathfrak e_{\dd}\LL^{-|\dd|}\right ) \\
 & = \frac{\LL^{\dim (V_\Sigma )}}{( 1 - \LL^{-1} ) ^r }  [ W ] \LL^{-\ell ( \SSS ) \dim ( V_\Sigma )}   \prod_{p\in | \SSS |} ( 1  - \LL_p ^{-1})^r  \prod_{p\notin | \SSS |} P_{B_\Sigma } ( \LL^{-1} ) \\
& + \text{ an error term of dimension at most: }\\
 & \qquad - \frac{1}{4} \min_{\alpha \in \Sigma ( 1 ) } (d_\alpha ) + ( 1 -  \ell ( \SSS )  ) ( \dim ( V_\Sigma )  - 1 )+ \dim ( W ) 
\end{align*}
for all $\dd \in \CEff ( V_\Sigma )_\ZZ^\vee$. 
This concludes the proof of \cref{thm-equidistribution-toric}. 
\end{proof}

\bigskip 


\section{Twisted products of toric varieties}

The goal of this section is to apply the notion of equidistribution of (rational) curves 
to the case of a certain kind of twisted products.
It provides a going-up theorem answering the following question in a particular setting: given a fibration, if a Batyrev-Manin-Peyre principle holds for the base and for the fibres, does it hold for the entire fibration ?

First, we recall the construction and geometric properties of such a twisted product, as it is done in \cite{chambert2001torseurs}. 
Then we study the moduli space of rational curves and apply the change of model 
\cref{thm:equidistribution-and-models}
to this context. 

\subsection{Generalities on twisted products} 
\label{section:generalities-twisted-products}
In this section we adapt the framework of \cite{chambert2001torseurs} and \cite{strauch1997height} to the study of rational curves. 
Concerning torsors, we will refer to \cite{colliot1987descente}. 

In this paragraph $S$ is a scheme, $G$ is a linear flat group scheme over $S$, with connected fibres, and $g : \BBB \to S$ a flat scheme over $S$.
Recall that a $G$-torsor over $\BBB$ is a scheme $\TTT \to \BBB$ over $\BBB$ which is faithfully flat and locally of finite presentation, endowed with a $G$-action $\tau : G\times_S \TTT \to \TTT$ over $\BBB$ such that the induced morphism 
\[
(\tau, \pr_2 )  : G \times_S \TTT \to \TTT \times_\BBB \TTT
\]
is an isomorphism. Moreover, in this article we will only consider torsors which are locally trivial for the Zariski topology. 

\subsubsection{Twisted products and twisted invertible sheaves}
\index{twisted product $\TTT \times^G X$}
Following \cite[\S 2]{chambert2001torseurs}, let $f : X \to S$ be a flat (quasi-compact and quasi-separated) $S$-scheme endowed with an action of $G/S$. Let $\TTT \to \BBB $ be a $G$-torsor locally trivial for the Zariski topology. 
We construct a fibration $\pi : \XXX = \TTT \times^G X \to \BBB$ locally isomorphic to $X$ over $\BBB$ in the following manner.
Let $(U_i)_{i\in I}$ be a Zariski-covering of $\BBB$ together with trivialisation $\phi_i : G \times _S U_i \to \TTT_{U_i}$.
For all $i,j\in I$ there exists a unique section $g_{ij}$ of $G$ over $U_i \cap U_j$ such that $\phi_i = g_{ij} \phi_j$ on $U_i \cap U_j$. This data provides a cocycle whose class in $H^1 ( \BBB_\Zar , G )$ represents the isomorphism class of $\TTT $ as a $G$-torsor. 
Then we set $\XXX_i = X \times_S U_i$. The action of $G/S$ on $X / S $ induces an action of $g_{ij}$ over $X \times_S ( U_i \cap U_j )$ 
and the later 
yields an isomorphism $\varphi_{ij} : {\XXX_j}_{| U_i \cap U_j } \simeq {\XXX_i}_{| U_i \cap U_j }$. Gluing \textit{via} the $\varphi_{ij}$'s defines $\pi : \XXX \to \BBB$. 
Up to a unique isomorphism, this construction does not depend on the choice of the open sets $(U_i)$.

\smallskip

There exists a functor $\vartheta $ from the category of $G$-linearised quasi-coherent sheaves over $X$ to the category of quasi-coherent sheaves over $\XXX$ \cite[Construction 2.1.7]{chambert2001torseurs} 
which is compatible with the standard operations for sheaves (direct sum, tensor product, localisation).
It sends a $G$-linearised quasi-coherent sheaf over $X$
to its twisted version over $\XXX$,
the gluing isomorphisms being given by the $\varphi_{i,j}$'s. 
This functor induces a map on the isomorphism classes. 
In particular, such a map sends $\Omega^1_{X / S}$ to $\Omega^1_{\XXX / \BBB}$ \cite[Proposition 2.1.8]{chambert2001torseurs}.

When $X = S $ then $\XXX = \BBB$ and this functor $\vartheta $ is written $\eta_\TTT$. It induces a group morphism $\mathcal X^* ( G) \to \Pic ( \BBB ) $, also written $\eta_\TTT$,
sending $\chi$ to the (isomorphism class of the) line bundle on $\BBB$ obtained \textit{via} the gluing morphisms
\[
(  u, t ) 
	\in ( U_i \cap U_j ) \times_S \mathbf A^1_S 
\mapsto 
( u , \chi ( g_{ij} ) t ) 
	\in ( U_j \cap U_i ) \times_S \mathbf A^1_S 
\]
where $( g_{ij} ) \in H^1 ( \BBB_\Zar , G )$ is the cocycle defined above.

We define $\Pic ^G ( X ) $ to be the group of isomorphism classes of $G$-linearised invertible sheaves on $X$.
If $X/S$ and $\BBB / S $ are smooth, the canonical sheaf over $X/S$ is endowed with a canonical $G$-linearisation
and 
\[
\omega_{\XXX / S } \simeq \vartheta ( \omega_{X/S} ) \otimes \pi^* \omega_{\BBB / S}
\]
by \cite[Proposition 2.1.8]{chambert2001torseurs}.
The forgetful functor $\varpi$
induces a forgetful morphism 
\[
\varpi : \Pic^G ( X ) \to \Pic ( X ) .
\]
Let 
\[
\iota : \mathcal X^* ( G ) \to \Pic ^G ( X ) 
\]
be the group morphism sending $\chi$ to the (isomorphism class of the) trivial bundle 
\[
X\times_S \mathbf A^1_S
\] 
together with the action of $G $ given by
\[
g \cdot ( x , t  ) = ( g \cdot x , \chi ( g ) t )  
\]
for all $(x, t )  \in X \times_S  \mathbf A^1_S$ and $g\in G$. 

\medskip

Putting $\iota$, $\eta_\TTT$, $\vartheta$ and $\pi^*$ together, we get morphisms
\[
\begin{tikzcd}
\mathcal X^* ( G ) \rar{	\left ( \iota , - \eta_\TTT \right )} & \Pic^G ( X ) \oplus \Pic ( \BBB )
\end{tikzcd}
\]
and
\[
\begin{tikzcd}
\Pic^G ( X ) \oplus \Pic ( \BBB ) \rar{\vartheta \otimes \pi^*} & \Pic ( \XXX ) .
\end{tikzcd}
\]
For every character $\chi$,
there is a canonical isomorphism of invertible sheaves on $\XXX$
\begin{equation} \label{equation:identification-functions-theta-iota-and-pi-star-eta}
\vartheta ( \iota ( \chi ) )\simeq \pi^* \eta_\TTT ( \chi )
\end{equation}
by \cite[Proposition 2.1.11]{chambert2001torseurs}.

\subsubsection{Twisted products over a field: an exact sequence}
Assume that $S$ is the spectrum of a field $\kb$, that $\BBB$ is a smooth proper and geometrically integral variety over $\kb$ and $G$ is a multiplicative group. Then, there is an exact sequence \cite[(2.0.2) \& Theorem 1.5.1]{colliot1987descente}
\[
0 
\longrightarrow 
H^1 ( \kb , G ) 
\longrightarrow 
H^1 ( \BBB , G ) 
\longrightarrow 
\Hom ( \mathcal X^* ( G ) , \Pic ( \overline \BBB ) )
\longrightarrow  
H^2 ( \kb , G )
\longrightarrow 
H^2 ( \BBB , G). 
\]
Assume moreover that $\BBB$ admits an open subset $U$ such that $\Pic ( \overline U ) =0 $. Then by \cite[Remark 2.2.7 \& Proposition 2.2.8]{colliot1987descente}, in the previous exact sequence $H^2 ( k , G ) \to H^2 ( \BBB , G ) $ is injective and the resulting short exact sequence 
\begin{equation}\label{equation:exact-sequence-obstruction-torsor-class-type}
0 
\longrightarrow 
H^1 ( \kb , G ) 
\longrightarrow 
H^1 ( \BBB , G ) 
\longrightarrow 
\Hom ( \mathcal X^* ( G ) , \Pic ( \overline \BBB ) )
\longrightarrow  
0 
\end{equation}
splits. 
It is the case if the base $\BBB$ has a $\kb$-rational point, the splitting being given by the evaluation map $H^1 ( \BBB , G)  \to H^1 ( \kb , G)$ at this point. 

\subsubsection{H90 multiplicative groups}

We again take $S$ to be the spectrum of a field $\kb$ 
and 
$G$ is a linear connected group over $\kb$. 

\begin{mydef}\label{def:h90-multiplicative-group}
\index{H90 multiplicative algebraic group}
We will say that $G$ is an H90 multiplicative algebraic group if
\begin{itemize}
\item $H^1 ( \kb , G ) $ is trivial;
\item $G$ is multiplicative and solvable over $\kb$. 
\end{itemize}
If $G$ acts on a projective $\kb$-variety $V$, we will always assume that
every line bundle on $V$ admits a $G$-linearisation. 
\end{mydef}

\begin{myexample}
If 
$\kb$ has cohomological dimension at most $1$ 
and 
$G$ is  a linear connected group which is solvable over $\kb$,
then by \cite[Théorème 1']{serre1994cohomologie} 
the first cohomology group $H^1 ( \kb , G ) $ is trivial. 
\end{myexample}

\begin{myexample}
If $V$ is a split smooth toric variety and $G$ is its torus, then $G$ is an H90 multiplicative algebraic group, by Hilbert 90. 
\end{myexample}

\subsection{Twisted models of $X$ over $\PP^1_\kb$}

From now on we assume that $X$ is Fano-like, that $G$ is H90 multiplicative and acts on $X$, and that $\BBB ( \kb ) $ is Zariski-dense in $\BBB$. 
Moreover, we assume that the Picard groups of $\BBB$ and $X$ coincide respectively with their geometric Picard group: $\Pic ( \overline \BBB ) \simeq \Pic ( \BBB )$ and $\Pic ( \overline X ) \simeq \Pic ( X )$. 

Then the sequence 
\begin{equation}\label{equation:exact-sequence-picard-groups}
\begin{tikzcd}
	0 \rar & \mathcal X^* ( G )  \rar{\left (\iota , - \eta_\TTT \right )}  & \Pic^G ( X ) \oplus \Pic ( \BBB ) \rar{\vartheta \otimes \pi^*} & \Pic ( \XXX ) \rar & 0  
\end{tikzcd}
\end{equation}
is exact by \cite[Théorème 2.2.4]{chambert2001torseurs}.
As a corollary,
we get exact sequences 
\[
\begin{tikzcd}
0 \rar & \mathcal X^* ( G ) \rar{\iota} & \Pic^G ( X ) \rar{\varpi} & \Pic ( X ) 	\rar & 0
\end{tikzcd}
\]
(by taking $\BBB = \Spec ( \kb )$ in \eqref{equation:exact-sequence-picard-groups})
and
\begin{equation}
\label{equation:exact-sequence-picard-groups-detwisted}
\begin{tikzcd}
0 \rar & \Pic ( \BBB ) \rar{\pi^*} & \Pic ( \XXX )  \rar{\widetilde{\varpi}} & \Pic ( X ) 	\rar & 0 . 
\end{tikzcd}
\end{equation}
The map $\widetilde{\varpi} : \Pic ( \XXX ) \to \Pic ( X ) $ above is the map
sending the class of a line bundle of the form 
\[
\vartheta ( L ) \otimes \pi^* ( \LLL ),
\]
with $\LLL$ a line bundle on $\BBB$ and $L$ a $G$-linearised line bundle on $X$, 
to the one 
of
$\varpi ( L ) $. 

\subsubsection{Pulling-back}
Let $f:\PP^1_\kb \to \BBB$ be a rational curve on $\BBB$. It induces a morphism $\degg f : \Pic ( \BBB ) \to \Pic ( \PP^1_\kb ) \simeq \ZZ $. 
Since in our situation we assume $G$ to be an H90 multiplicative group,
type and class coincide by \eqref{equation:exact-sequence-obstruction-torsor-class-type}, so that $\eta_\TTT $ and the type ${\alpha  \in \Hom ( \mathcal X^* ( G ) , \Pic ( \BBB ) )}$ of the $G$-torsor $\TTT \to \BBB$ can be identified
(we refer the interested reader to \cite[\S 2]{colliot1987descente} for the precise definition of type). 
The pulling-back operation
\[
\begin{tikzcd}
\TTT_f \arrow[r ] \arrow[d] \arrow[dr, phantom, "\ulcorner", very near start] & \TTT \arrow[d] \\
\PP^1_\kb 	\arrow[r,"f"] & \BBB
\end{tikzcd} 
\]induces a $G$-torsor $\TTT_f$ 
whose type (or class) is given by $( \degg f  ) \circ \alpha \in \mathcal X_* ( G ) $, together with functors on quasi-coherent sheaves
$\vartheta_f = f_\XXX^* \circ \vartheta$ and $\eta_{\TTT_f } = f^* \circ \eta_{\TTT}$.

\smallskip

Then, remark that the pull-back
\[
\begin{tikzcd}
\XXX_{f} \arrow[r,"f_\XXX"] \arrow[d] \arrow[dr, phantom, "\ulcorner", very near start] & \XXX \arrow[d] \\
\PP^1_\kb 	\arrow[r,"f"] & \BBB .
\end{tikzcd} 
\]
only depends on the multidegree of $f$
since it is exactly the twisted product obtained 
by starting from the $G$-torsor $\TTT_f \to \PP^1_\kb$ 
of class $f^* \circ \eta_\TTT = \degg ( f ) \circ \alpha $. 

In order to compare degrees
of line bundles on models of $X$ above $\PP^1_\kb$
 coming from different $f$ of same multidegree $\mdeg_\BBB$, we need to find canonical isomorphisms between the Picard groups of these different models.
So we take $f$ and $f'$ to be two rational curves $\PP^1\to \BBB$ of equal multidegree $\mdeg_\BBB$.
They induce pull-backs $\TTT_f$, $\TTT_{f'}$ and $\XXX_f$, $\XXX_{f'}$. 
Since $\TTT_f$ and $\TTT_{f'}$ have equal types and classes, $\XXX_f$ and $\XXX_{f'}$ are isomorphic as $G$-varieties.
We get a
commutative diagram of exact sequences
\[
\begin{tikzcd}[swap,bend angle=60]
0 \rar 
& \mathcal X^* ( G )  \dar[equal] \rar{\left (\iota , - \eta_\TTT \right )} 
& \Pic ^G ( X ) \oplus \Pic ( \BBB )  \dar{(\mathrm{id}, \mdeg_\BBB )} \rar{\vartheta \otimes \pi^* }
& \Pic ( \XXX ) \rar  \dar[swap]{f_\XXX^*}
& 0 \\
0 \rar 
& \mathcal X^* ( G ) \dar[equal] \rar{\left (\iota , - \eta_{\TTT_f} \right )} 
& \Pic ^G ( X ) \oplus \Pic ( \PP^1_\kb ) \rar{ \vartheta_f \otimes \pi_f^* }   \dar[equal]
& \Pic ( \XXX_f )  \arrow[d,dashed,swap,"\simeq","\exists !"'] \rar & 0 \\
0 \rar 
& \mathcal X^* ( G ) \rar{\left (\iota , - \eta_{\TTT_{f'}} \right )} 
& \Pic ^G ( X ) \oplus \Pic ( \PP^1_\kb ) \rar{ \vartheta_{f'} \otimes \pi_{f'}^* }   
& \Pic ( \XXX_{f'} ) \arrow[from=uu,bend left,crossing over,"{f_\XXX '}^*" near end,swap]
 \rar & 0 
\end{tikzcd}
\]
providing a canonical isomorphism $\Pic ( \XXX_f ) \simeq \Pic ( \XXX_{f'} )$.

\subsubsection{Multidegrees of sections of $\XXX_f \to \PP^1_\kb$}
Assume now that $f:\PP^1_\kb \to \BBB$ comes from a morphism ${ g:\PP^1_\kb \to \XXX }$, 
that is to say, $f=\pi \circ g$. Then we obtain the following Cartesian square
\[
\begin{tikzcd}
 \XXX_f \arrow[r,"f_\XXX"] \arrow[d,right,"\pi_f",swap] \arrow[dr, phantom, "\ulcorner", very near start] & \XXX \arrow[d,"\pi "] \\
\PP^1_\kb \arrow[ru,"g"'] \arrow[u,shift right=1.5ex,dashrightarrow,"\sigma",swap] \arrow[r,"f"] & \BBB 
\end{tikzcd}
\]
in which $g$ induces a unique section $ \sigma : \PP^1_\kb \to \XXX_f $ such that $g = f_\XXX \circ \sigma $.  
We deduce from this square the relations on degree maps
\begin{align*}
	 \degg ( g ) &=  \degg ( \sigma ) \circ f^*_\XXX  \\
	 \degg ( f ) &  = \degg (g ) \circ \pi^* \\
	 			&= \degg ( \sigma ) \circ f^*_\XXX \circ \pi^*\\
	\mathrm{id}_{\Pic ( \PP^1_\kb )}  & = \degg ( \sigma ) \circ \pi_f^*  .
\end{align*}

\begin{mysetting}\label{setting:line-bundles-on-twisted-products}
	Let $L_1, ... , L_{r_X}$ be a family of line bundles on $X$ whose classes form a $\ZZ$-basis of $\Pic ( X ) $.
	We fix a section $s : \Pic ( X ) \to \Pic ^G ( X ) $ by choosing a $G$-linearisation on each $L_i$. 
\end{mysetting}

From \eqref{equation:exact-sequence-picard-groups} 
and \eqref{equation:exact-sequence-picard-groups-detwisted} 
one deduces the commutative diagram of exact sequences
\[
\begin{tikzcd}
0 \rar & \Pic ( \BBB ) \rar{\pi^*} \dar{\degg ( f )} & \Pic ( \XXX ) \rar{\widetilde \varpi} \dar{f^*_\XXX } & \Pic^G ( X ) / \iota ( \mathcal X^* ( G )) \simeq \Pic ( X  ) \dar[equal] \rar & 0 	\\
0 \rar & \Pic ( \PP^1_\kb ) \rar{\pi_f^*} & \Pic ( \XXX_f ) \rar{\widetilde{\varpi_f}} & \Pic^G ( X ) / \iota ( \mathcal X^* ( G )) \simeq \Pic ( X  ) \rar & 0 	
\end{tikzcd}
\]
where the arrow $\widetilde{\varpi_f} : \Pic ( \XXX_f ) \to \Pic ( X ) \simeq \ZZ$ 
is obtained from $\widetilde \varpi $ by replacing $\vartheta$ by $\vartheta_f$ and $\pi$ by $\pi_f$.

Furthermore $\degg ( \sigma ) : \Pic ( \XXX_f ) \to \Pic ( \PP^1_\kb )$ induces
by composition with ${s : \Pic ( X ) \to \Pic^G ( X )} $ and ${\vartheta_f : \Pic ^G ( X ) \to \Pic ( \XXX_f )} $
 a multidegree 
\[
\mdeg_X (\sigma ) : \Pic ( X ) \to \Pic ( \PP^1_\kb ) \simeq \ZZ 
\]
sending the class of a line bundle $L$ on $X$
to 
\[ 
\degg ( \sigma ) \scdot 
( \vartheta_f \circ  s ) ( [ L ] )  = 
\degg ( g ) \scdot   (\vartheta \circ   s ) ( [ L ] )  .
\]
If $\sigma ' $ is another section obtained in this way, 
that is to say from another $g'$ and another $f'$ such that $\pi \circ g '= f' $ and $\degg ( g ) = \degg ( g ' ) $,
then $\degg ( f) = \degg ( f ' ) = \delta_\BBB$, and $\XXX_f \simeq \XXX_{f'} $.
The following commutative diagram summarises the situation and shows that ${\delta_X ( \sigma ) = \delta_X ( \sigma ' ) = \degg ( g ) \circ \vartheta \circ s }$, an element of $\Pic ( X )^\vee $ which will be denoted by $\mdeg_X ( g )$.
\[s
\begin{tikzcd}
0 \rar & \Pic ( \BBB ) \rar{\pi^*} \dar[swap]{\mdeg_\BBB} & \Pic ( \XXX ) \arrow[dl,swap,"\degg ( g )"] \rar \dar{f^*_\XXX } & \arrow[l,bend right,swap,"\vartheta \circ s"] \Pic ( X  ) \arrow[dl,"\upsilon_f \circ s"] \dar[equal] \rar & 0 	\\
0 \rar & \Pic ( \PP^1_\kb ) \dar[equal] \rar{\pi_f^*} & \Pic ( \XXX_f ) \dar{\simeq} \arrow[l,bend left=20,"\degg ( \sigma )"]  \rar &  \Pic ( X  ) \rar \dar[equal] & 0 	\\
0 \rar & \Pic ( \PP^1_\kb ) \rar{\pi_{f'}^*} & \Pic ( \XXX_{f'} ) \arrow[l,bend left=20,"\degg ( \sigma ' )"] \arrow[from=uu,bend left=60,swap,crossing over,"{f_\XXX '}^*" near end,swap] \rar &  \Pic ( X  ) \rar & 0 	
\end{tikzcd}
\]
By duality, from \eqref{equation:exact-sequence-picard-groups} we get an exact sequence 
\[
0 \longrightarrow \Pic ( \XXX ) ^\vee \overset{(\vartheta^\vee ,   \pi_*) }{\longrightarrow} \Pic ^G ( X )^ \vee \oplus \Pic ( \BBB ) ^\vee \longrightarrow \mathcal X_* ( G ) \longrightarrow 0 
\]
which allows us to decompose a multidegree on $\XXX$. 
\begin{mylemma} 
\label{lemma:relation-delta-X-G-delta-X}
Let $\degg ( g )  = (\delta^G_X ( g ), \delta_\BBB ( g ) )$ viewed in $\Pic^G (X )^\vee \oplus \Pic ( \BBB )^\vee$.
Then 
the morphism
	\[
	[L] \in \Pic^G ( X ) 
	\longmapsto 
	\delta_X^G ( g ) 
	\scdot 
	[ L ]  
	- 
	\delta_X ( g ) 
	\scdot 
	[\varpi (  L  )] 
	\in \Pic ( \PP^1_\kb ) \simeq \ZZ 
	\]
	defines a cocharacter of $G$ given by 
	\[
	\chi \in \mathcal X^* ( G ) 
	\longmapsto
	\delta_\BBB ( g ) 
	\scdot 
	(\eta_\TTT \circ \iota ) ( \chi )
	 .
	\]
\end{mylemma}
\begin{proof}We use again our favorite exact sequences to get the following diagram
\[
\begin{tikzcd}
  &   & 				0  \dar  &   0 \dar &     & \\
  &   &			\mathcal X^* ( G ) \dar{\left ( \iota , - \eta_\TTT  \right )}	\rar[equal]	&  \mathcal X^* ( G ) \dar{\iota} & 		\\
  &   & \Pic^G ( X ) \oplus \Pic ( \BBB ) \dar[swap]{\vartheta \otimes \pi^*} \dlar \rar & \Pic^G ( X ) \dlar[swap]{\vartheta} \dar{\varpi} &  	\\
 0 \rar & \Pic ( \BBB ) \rar{\pi^*} \dar[swap]{\mdeg_\BBB (g)} & \Pic ( \XXX ) \arrow[dl,swap,"\degg ( g )"] \rar \dar &  \Pic ( X  ) \arrow[u,bend left,"s"] \dar \rar & 0 	\\
0 \rar & \Pic ( \PP^1_\kb ) \arrow[from=uurr,bend left=15,crossing over,near end,"\delta_X^G (g)"] \arrow[from=urr,bend left=30,crossing over,near end,"\delta_X (g)"] & 0   & 0   & \\
\end{tikzcd}
\]
from which one reads $\delta_X ( g ) = \delta_X^G ( g ) \circ s $ 
and $\delta_X^G ( g ) = \degg ( g ) \circ \vartheta $.

Let $L$ be a $G$-linearised line bundle on $X$.
Since $s$ is a section of $\varpi$,
there exists a unique character $\chi \in \mathcal X^* ( G ) $
such that 
\[
[ L ] = ( s \circ \varpi ) ( [ L ] ) + \iota ( \chi ) 
\]
in $\Pic^G ( X ) $. Taking intersection degrees, we get
\begin{align*}
\delta^G_X ( g ) \scdot  [ L ]  & = 
\delta_X^G ( g ) \scdot (s \circ  \varpi ) ( [ L ] )    +  \delta_X^G ( g ) \scdot \iota ( \chi ) 
\\
& = \delta_X ( g ) \scdot  \varpi ( [ L ] )   +  \degg ( g ) \scdot ( \vartheta \circ \iota \, ( \chi )) \\
& = \delta_X ( g ) \scdot  \varpi ( [ L ] ) +  \degg ( g ) \scdot ( \pi^* \circ \eta_\TTT \, ( \chi ) ) 
\tag{by \eqref{equation:identification-functions-theta-iota-and-pi-star-eta} \cpageref{equation:identification-functions-theta-iota-and-pi-star-eta} }
\\
& = \delta_X ( g ) \scdot  \varpi ( [ L ] )  +   \delta_\BBB ( g ) \scdot \eta_\TTT ( \chi ) .
\end{align*}
Hence 
\[
( \delta_X^G ( g ) \scdot \vartheta - \delta_X ( g )  \scdot \varpi ) : [L] \longmapsto \delta_\BBB \scdot  \eta_\TTT ( [L] - s \circ \varpi [ L ]  ) 
\]
lies in $\mathcal X_* ( G )$.
\end{proof}

\smallskip 

We reformulate these remarks in terms of moduli spaces of rational curves in the following paragraph. 

\subsection{Moduli spaces of morphisms and sections}
\label{subsetion:moduli-spaces-of-morphims-on-toric-products}
In what follows $\UUU$ is a dense open subset of $\BBB$, $U$ a dense open subset of $X$ which is stable under the action of $G$, and $\widetilde \UUU$ is the intersection of the preimage of $\UUU $ with $\TTT \times^G U$ in $\XXX$. 

Let $\kb ' $ be an extension of $\kb$ and $f: \PP^1_{\kb '} \to  \BBB_{\kb '}$ be a $\kb '$-morphism given by a $\kb '$-point $x$ of $\Hom ( \PP^1_\kb , \BBB)_\UUU $. As before, the morphism $f$ defines pull-backs
$\TTT_f = \TTT \times_f \PP^1_{\kb '} $
and
 $\XXX_f = \PP^1_{\kb '}\times_f \XXX_{\kb ' }$ over $\PP^1_{\kb '}$. 
We fix once and for all a representative $\TTT_{\delta_\BBB}$
of the isomorphism class of $\TTT_f$,
whenever $\delta_\BBB = \degg ( f ) : \Pic ( \BBB_{\kb ' } ) \to \Pic ( \PP^1_{\kb '} ) $,
as well as a corresponding twisted product $\XXX_{\delta_\BBB}$.
We have canonical isomorphisms 
$\TTT_{\delta_\BBB} \simeq \TTT_f$,
$\XXX_{\delta_\BBB} \simeq \XXX_f$
and 
$\Pic ( \XXX_{\delta_\BBB} ) \simeq \Pic ( \XXX_f )$.

\begin{mylemma}\label{lemma:schematic-fibre-of-projection-moduli-spaces}
The schematic fibre $M_{x}$ 
of \[
\pi_* : \Hom ( \PP^1_\kb , \XXX )_{\widetilde{\UUU}} \to \Hom ( \PP^1_\kb , \BBB)_\UUU 
\]
over the $\kb '$-point $x $ corresponding to $f$
is canonically isomorphic to
\[ 
\Hom_{\PP^1_{\kb '}} ( \PP^1_{\kb '} , \XXX_f )_{U} \simeq \Hom_{\PP^1_{\kb '}} ( \PP^1_{\kb '} , \XXX_{\delta_\BBB} )_{U}.
\] 
\end{mylemma}

\begin{proof}
On the one hand, if we consider $T$-points of $M_{\kb '}$, where $T$ is a scheme over $\kb '$, 
we get $T$-morphisms $g : T  \times_{\kb '} \PP^1_{\kb '} \to T  \times_{\kb '} \XXX_{\kb '}$ 
and a commutative diagram 
\[
\begin{tikzcd}
 \XXX_T \times_{\BBB_T} \PP^1_T   \arrow[r,"\pr_{  \XXX_{T}}"] \arrow[d,right,"\pr_{\PP^1_T}"'] \arrow[dr, phantom, "\ulcorner", very near start] &  \XXX_{T} \arrow[d," \pi_{T}"] \\
\PP^1_{T} \arrow[ru,"g"'] \arrow[r,"(\mathrm{id}_T \times f)"] & \BBB_{T} .
\end{tikzcd}
\] 
The product $\XXX_T \times_{\BBB_T} \PP^1_T $ is nothing else but the extension of scalars of $\XXX_f$ to $T$ 
and the previous square is nothing else but
\begin{equation}\label{diagram:pull-back-point-of-Hom}
\begin{tikzcd}
T \times_{\kb '} \XXX_f  \arrow[r] \arrow[d,right,"\pr_{\PP^1_T}"'] \arrow[dr, phantom, "\ulcorner", very near start] &  T \times_{\kb '} \XXX_{\kb '} \arrow[d,"(\mathrm{id} \times \pi_{\kb '})"] \\
T \times_{\kb '}  \PP^1_{\kb '} \arrow[ru,"g"'] \arrow[u,shift right=1.5ex,dashrightarrow,"\exists !"'] \arrow[r,"(\mathrm{id}_T \times f)"] & T \times_{\kb '} \BBB_{\kb '} 
\end{tikzcd}
\end{equation}
giving the existence of a unique $T$-section $\sigma : \PP^1_{T} \to T \times_{\kb '} \XXX_f $. 
On the other hand, such a $T$-section defines a unique $T$-morphism $g : \PP_{T}^1 \to T \times \XXX_{\kb '}$
 making the bottom right triangle commutative, 
 that is, a unique $T$-point of $M_{\kb '}$. 
Thus the schematic fibre of $\pi_*$ at the ${\kb '}$-point $x$ corresponding to $f : \PP^1_{\kb ' } \to \BBB_{\kb ' }$ is canonically isomorphic to $\Hom_{\PP^1_{\kb '} } ( \PP^1_{\kb '} ,   \XXX_f )$ as a $\kb '$-scheme. 
\end{proof}

In particular, 
the previous argument shows that for every $\kb'$-scheme $T$ there is a map of sets   
\[
\Sigma_\delta ( T ) : \Hom ^{\delta}_\kb ( \PP^1_\kb , \XXX ) ( T ) \to \Hom_{\PP^1_\kb} ^{\delta_X^G \circ s } ( \PP^1_\kb , \XXX_{\delta_\BBB} ) ( T )
\]
sending a $T$-point $g :  \PP^1_T \to \XXX \times_\kb T $ of multidegree $\delta  \in \Pic ( \XXX ) ^\vee $ to the unique $T$-point $\sigma $ of $\Hom_{\PP^1_\kb} ( \PP^1_\kb , \XXX_f )  \simeq \Hom_{\PP^1_\kb} ( \PP^1_\kb , \XXX_{\delta_\BBB} ) $ given by the dashed arrow in \eqref{diagram:pull-back-point-of-Hom}.   
Note that this construction is functorial in $T$, leading to a morphism of schemes 
\[
\Sigma_\delta : 
\Hom^{\delta}_\kb ( \PP^1_\kb , \XXX )  
\to 
\Hom_{\PP^1_\kb} ^{\delta_X^G \circ s } ( \PP^1_\kb , \XXX_{\delta_\BBB} ) .
\]

From \cref{ptn:classes-of-piecewise-trivial-fibrations}
\cpageref{ptn:classes-of-piecewise-trivial-fibrations}
we deduce:
\begin{myptn} \label{ptn:class-of-moduli-space-is-product-of-the-classes}
Let $s : \Pic ( X ) \to \Pic ^G ( X ) $ be a section of the forgetful morphism $\varpi$.
Then
for any class 
$\mdeg = ( \mdeg_X^G , \mdeg_\BBB ) $  
we have
\begin{align*}
& \left [ \Hom ^{\mdeg_\XXX}_\kb ( \PP^1_\kb , \XXX )_{\widetilde \UUU}  \right ] \\
& =  
\left [\Hom ^{\delta_\BBB }_\kb ( \PP^1_\kb , \BBB )_{\UUU} \right ]
\left [ \Hom ^{\delta_X^G \circ s }_{\PP^1_\kb} ( \PP^1_\kb , \XXX_{\mdeg_\BBB} )_U \right ]
\end{align*}
in $\KVar \kb$.
\end{myptn}

\subsection{Asymptotic behaviour}
We assume that the motivic constants
$\tau_{\PP^1_\kb} ( X ) $
and 
$ \tau_{\PP^1_\kb} ( \BBB ) $
are well-defined in $\widehat{\mathscr M_\kb} = \widehat{\mathscr M_\kb}^{w}$ or $\widehat{\mathscr M_\kb} = \widehat{\mathscr M_\kb}^{\dim }$. 

\begin{mylemma}
The symbol $\tau_{\PP^1_\kb} ( \XXX ) $ is well-defined 
and one has
\[
\tau_{\PP^1_\kb} ( \XXX )  =  \tau_{\PP^1_\kb} ( X ) \tau_{\PP^1_\kb} ( \BBB )
\]
in $\widehat{\mathscr M_\kb}$. 
\end{mylemma}
\begin{proof}
As abstract series, 
the equality $\tau_{\PP^1_\kb} ( \XXX )  =  \tau_{\PP^1_\kb} ( X ) \tau_{\PP^1_\kb} ( \BBB )$
is a consequence of the multiplicative property of the motivic Euler product given by \cref{proposition:euler-product:multiplicativity}.
Indeed, by local triviality of the fibration, one has the relation $[ \XXX \times_\kb \PP^1_\kb ]   = [ X  \times_\kb \PP^1_\kb ] [ \BBB \times_\kb \PP^1_\kb ]$ in $\mathscr M_{\PP^1_\kb }$. 
Since $\tau_{\PP^1_\kb} ( X )$ and $ \tau_{\PP^1_\kb} ( \BBB )$ both converge in $\widehat{\mathscr M_\kb}$, so does $\tau_{\PP^1_\kb} ( \XXX ) $. 
\end{proof}

\begin{mythm}\label{thm:BMP-twisted-products}
Let $X$ and $\BBB$ be two Fano-like varieties\footnote{See \cref{def:Fano-like}, \cpageref{def:Fano-like}.}
defined over the base field $\kb$.
Assume that $G$ is  an $H90$-multiplicative group\footnote{See \cref{def:h90-multiplicative-group}, \cpageref{def:h90-multiplicative-group}. For example, take $G$ to be a split torus.} acting on $X$ and that every line bundle on $X$ admits a $G$-linearisation.
Let $U$ and $\UUU$ be dense open subsets respectively of $X$ and $\BBB$, with $U$ stable under the action of $G$.

Let $\TTT$ be a $G$-torsor over $\BBB$ and 
\[
\XXX = \TTT \times ^G X 
\]
the twisted product\footnote{See \cref{section:generalities-twisted-products}, \cpageref{section:generalities-twisted-products}.} of $X$ and $\TTT$ over $\BBB$.
Let $\widetilde \UUU$ be the intersection 
of the preimage of $\UUU$ with $\TTT \times ^G U $ in $\XXX$. 

Assume that the 
motivic Batyrev-Manin-Peyre principle for rational curves\footnote{See \cref{BMP-motivic-weak}, \cpageref{BMP-motivic-weak}, and more generally \cref{conj:BMP-motivic-strong-weight-top}, \cpageref{conj:BMP-motivic-strong-weight-top}.} 
holds both for $ X  $ and $ \BBB $, 
for curves generically intersecting $U$ and $\UUU$ respectively,
which means that 
\begin{align*}
\left [\Hom ^{\mdeg_X}_\kb ( \PP^1_\kb , X ) _U \right  ] \LL^{ - \,  \mdeg_X \scdot \omega_X^{-1}  }  & \longrightarrow \tau ( X ) \\
\left [\Hom ^{\mdeg_\BBB}_\kb ( \PP^1_\kb , \BBB )_\UUU  \right ] \LL^{  - \,  \mdeg_\BBB \scdot \omega_\BBB^{-1}   }& \longrightarrow \tau ( \BBB ) 
\end{align*}
when $d (  \mdeg_X  ,  \partial \mathrm{Eff} ( X )  ^\vee   ) $ and $d (   \mdeg_\BBB   ,  \partial \mathrm{Eff} ( \BBB )   ^\vee  ) $ both tend to infinity,
in $\widehat{\MMM_\kb} = \widehat{\MMM_\kb}^w $ or $\widehat{\MMM_\kb}^{\dim}$.

Assume furthermore that equidistribution of rational curves\footnote{See \cref{def:equidistribution-products-constructible-subsets}, \cpageref{def:equidistribution-products-constructible-subsets}.} holds for $X$.
\smallskip 
 
Then for $\mdeg  \in \CEff ( \XXX ) _\ZZ ^\vee$
the normalized class 
\[
\left [ \Hom ^{\mdeg }_\kb ( \PP^1_\kb , \XXX )_{ \widetilde \UUU } \right  ] \LL^{ - \, \mdeg \scdot \omega_\XXX^{-1}  }  
\]
tends to the non-zero effective element 
\[
\tau ( \XXX )  = \tau ( X  ) \tau ( \BBB ) \in \widehat{\MMM_\kb} 
\]
 when the distance $d (  \mdeg  ,  \partial \mathrm{Eff} ( \XXX )  ^\vee   ) $ goes to infinity.
\end{mythm}

Together with \cref{thm-equidistribution-toric}, we get the following. 

\begin{mycor}
	Let $X$ be a smooth projective split toric variety with open orbit $U \simeq \GG_m^n$
	and $\TTT \to \BBB$ 
	a $\GG_m^n$-torsor
	above a Fano-like variety $\BBB$. 
	Assume that the Batyrev-Manin-Peyre principle holds for rational curves on $\BBB$.
	Then it holds as well for rational curves on the twisted product $\XXX = \TTT \times ^G X $. 
\end{mycor}

\begin{proof}[Proof of \cref{thm:BMP-twisted-products}]
Let $\delta_\XXX = ( \delta_X^G , \delta_\BBB ) \in \Pic ( \XXX )^\vee $ viewed in $ \Pic^G ( X ) ^\vee \oplus \Pic ( \BBB)^\vee$. 
Fix a section 
\[
 s : \Pic ( X ) \to \Pic^G ( X )
 \]
of the forgetful morphism $\varpi$
as in \cref{setting:line-bundles-on-twisted-products}.
That is to say,
fix line bundles $L_1 , ... , L_{r_X}$ forming a basis of $\Pic ( X )$ 
	together with a $G$-linearisation on each of them.
	
 Given a curve $f : \PP^1_\kb \to \BBB$, we know from the previous sections that the isomorphism class (as a scheme over $\PP^1_\kb$) of the pull-back $\XXX_f $ only depends on its multidegree $\delta_\BBB $, and that it is a twisted model of $X$ over $\PP^1_\kb$. 
 In the beginning of \cref{subsetion:moduli-spaces-of-morphims-on-toric-products}, 
 we chose once and for all a representative $ \XXX_{\delta_\BBB}$ of its isomorphism class.
\[
\begin{tikzcd}
 \XXX_{\delta_\BBB} \arrow[r] \arrow[d,right] \arrow[dr, phantom, "\ulcorner", very near start] & \XXX \arrow[d,"\pi "] \\
\PP^1_\kb   \arrow[r] & \BBB 
\end{tikzcd}
\]
This model comes with functors $\vartheta_{\mdeg_\BBB}$ so that 
	$ \vartheta_{\mdeg_\BBB} ( L_ i )$ is a twisted model of $L_i$ on $\XXX_{\mdeg_\BBB}$,
and $\vartheta ( s \circ \varpi ( \omega_X^{-1} ) )$ a model of $\omega_X^{-1}$.
We fix $\mdeg_\BBB$ and consider sections $\sigma$ of $\XXX_{\mdeg_\BBB}$ of corresponding multidegree
	\[
	\mdeg_X^G \circ s : 
	[L] \mapsto \delta_X^G \scdot   s (  [ L ] )  = \delta \scdot \vartheta_{\mdeg_\BBB} ( s (  [ L ] ) ) .
	\]
By \cref{thm:equidistribution-and-models},
\[
\left [\Hom ^{ \mdeg_X^G \circ s }_{\PP^1_\kb} ( \PP^1_\kb , \XXX_{\mdeg_\BBB} )_U \right ] \LL^{- ( \mdeg_X^G \circ s )  \scdot \omega_V^{-1} } 
\]	
tends to $\tau_{ \vartheta ( s \circ \varpi ( \omega_X^{-1} ) )  } ( \XXX_{\mdeg} ) $
as $d ( \mdeg_X^G \circ s , \partial \CEff ( X )^\vee ) \to \infty $.
	Note that doing this way we obtain a motivic Tamagawa number with respect to the model $\vartheta ( s \circ \varpi ( \omega_X^{-1} ) )$
	of $\omega_X^{-1}$
	and that we can apply \cref{lemma:relation-delta-X-G-delta-X} to get the relation
	\[
	\tau_{ \vartheta ( s \circ \varpi ( \omega_X^{-1} ) )  } ( \XXX_{\mdeg_\BBB} ) 
	= 
	\LL_\kb^{ \delta_\BBB \scdot  \eta_\TTT (  [\omega_V^{-1}] - s \circ \varpi [ \omega_V^{-1} ]  ) } 	\tau_{ \vartheta ( \omega_X^{-1} )   } ( \XXX_{\mdeg_\BBB} ) 
	=  \LL_\kb^{ \delta_\BBB \scdot  \eta_\TTT (  [\omega_V^{-1}] - s \circ \varpi [ \omega_V^{-1} ]  ) }  \tau ( X ) . 
	\]
	
By \cref{ptn:class-of-moduli-space-is-product-of-the-classes} we have the equality of classes
	\begin{align*}
& \left [ \Hom ^{\mdeg}_\kb ( \PP^1_\kb , \XXX )_{\widetilde \UUU}  \right ] \\
& =  
\left [\Hom ^{\delta_\BBB}_\kb ( \PP^1_\kb , \BBB )_{\UUU} \right ]
\left [\Hom ^{ \mdeg_X^G \circ s }_{\PP^1_\kb} ( \PP^1_\kb , \XXX_{\mdeg_\BBB} )_U \right ].
\end{align*}
Moreover, the expression
	 \[
	 \omega_\XXX = \vartheta ( \omega_X ) \otimes \pi^* ( \omega_\BBB ) 
	 \]
and the projection formula
	 provide the decomposition of anticanonical degrees
	 \begin{align*}
	 \mdeg \scdot  \omega_\XXX^{-1} & = \mdeg_X^G \scdot \omega_X^{-1} + \mdeg_\BBB \scdot \omega_\BBB^{-1} \\
	 & = ( \delta_X^G \circ s ) \scdot \varpi [\omega_V^{-1} ] +  \delta_\BBB \scdot  \eta_\TTT (  [\omega_V^{-1}] - s \circ \varpi [ \omega_V^{-1} ]  ) +  \mdeg_\BBB \scdot \omega_\BBB^{-1}
	 \end{align*}
	 so that the normalised class 
	 \[ 
	 \left [ \Hom ^\mdeg_\kb ( \PP^1_\kb , \XXX )_{\widetilde \UUU } \right  ] \LL^{ - \mdeg \scdot  \omega_\XXX^{-1}}
	 \]
	 is the product 
	 \[
	 \left [\Hom ^{\delta_\BBB}_\kb ( \PP^1_\kb , \BBB )_{\UUU} \right ]
\LL^{ - \mdeg_\BBB \scdot \omega_\BBB^{-1}}
\times 
\left [\Hom ^{ \mdeg_X^G \circ s }_{\PP^1_\kb} ( \PP^1_\kb , \XXX_{\mdeg_\BBB} )_U \right ]
\LL^{ - ( \mdeg_X^G \circ s ) \scdot \omega_V^{-1} }\LL^{ - \delta_\BBB \scdot  \eta_\TTT (  [\omega_V^{-1}] - s \circ \varpi [ \omega_V^{-1} ]  ) }	 
	 \]
	 of  well-normalised classes, as expected.

To conclude the proof, we use \cite[Théorème 2.2.9]{chambert2001torseurs} ensuring that under our assumptions we have
\[
\mathrm{Eff} ( \XXX )_\ZZ^\vee = \vartheta ( \mathrm{Eff} ^G ( X ) )^\vee_\ZZ \oplus \pi^* (\mathrm{Eff} ( \BBB ) )^\vee_\ZZ 
\]
in $\Pic^G ( X ) ^\vee \oplus \Pic ( \BBB)^\vee$. 
Hence the condition 
	 $d ( \mdeg  , \partial \CEff ( \XXX )^\vee ) \to \infty$
	 means 
\[ 
d ( \mdeg_X^G  , \partial \CEff ^G ( X )^\vee ) \to \infty
\text{ and } 
d ( \mdeg_\BBB , \partial \CEff ( X )^\vee ) \to \infty,
\]
the first of these two conditions implying  
$d ( \mdeg_X^G \circ s , \partial \CEff ( X )^\vee ) \to \infty $
by \cref{lemma:relation-delta-X-G-delta-X} for $\mdeg_\BBB $ fixed. 
	  The result follows by continuity of the multiplication.
\end{proof}


\section*{Appendix: values of $\mu_{B_\Sigma}$ for the blow-up of $\PP^2_\kb$ in three general points (\cref{example-blow-up-3-pts})}
\label{annexe}

\begin{center}
\includegraphics{A2-section-7-tab-standalone.pdf}
\end{center}


\bibliography{motivic-distribution-of-rational-curves-pub-arXiv}


\end{document}